\documentclass[10pt]{amsart}
\usepackage{amsmath}
\usepackage{amsxtra}
\usepackage{amscd}
\usepackage{amsthm}
\usepackage{amsfonts}
\usepackage{amssymb}
\usepackage{eucal}
\usepackage[all]{xy}
\usepackage{graphicx}

\usepackage[usenames]{color}

\newtheorem{cor}[subsubsection]{Corollary}
\newtheorem{lem}[subsubsection]{Lemma}
\newtheorem{prop}[subsubsection]{Proposition}

\newtheorem{thm}[subsubsection]{Theorem}
\newtheorem{quest}[subsubsection]{Question}

\newcommand{\iso}{\buildrel{\sim}\over{\longrightarrow}}
\newcommand{\mono}{\hookrightarrow}

\theoremstyle{definition}
\newtheorem{defn}[subsubsection]{Definition}

\theoremstyle{remark}
\newtheorem{rem}[subsubsection]{Remark}

\newcommand{\thmref}[1]{Theorem~\ref{#1}}
\newcommand{\secref}[1]{Sect.~\ref{#1}}
\newcommand{\lemref}[1]{Lemma~\ref{#1}}
\newcommand{\propref}[1]{Proposition~\ref{#1}}
\newcommand{\corref}[1]{Corollary~\ref{#1}}

\numberwithin{equation}{section}

\newcommand{\nc}{\newcommand}
\nc{\renc}{\renewcommand}
\nc{\ssec}{\subsection}
\nc{\sssec}{\subsubsection}
\nc{\on}{\operatorname}

\nc\ol{\overline}
\nc\wt{\widetilde}
\nc\tboxtimes{\wt{\boxtimes}}
\nc{\alp}{\alpha}

\nc{\ZZ}{{\mathbb Z}}
\nc{\NN}{{\mathbb N}}
\nc{\OO}{{\mathbb O}}
\renc{\SS}{{\mathbb S}}
\nc{\DD}{{\mathbb D}}
\nc{\GG}{{\mathbb G}}

\nc{\Fq}{{\mathbb F}_q}
\nc{\Fqb}{\ol{{\mathbb F}_q}}
\nc{\Ql}{\ol{{\mathbb Q}_\ell}}
\nc{\id}{\text{id}}
\nc\X{\mathcal X}

\nc{\Ad}{\on{Ad}}
\nc{\Hom}{\on{Hom}}
\nc{\Hull}{\on{Hull}}
\nc{\Mat}{\on{Mat}}
\nc{\Lie}{\on{Lie}}
\nc{\Loc}{\on{Loc}}
\nc{\Pic}{\on{Pic}}
\nc{\Bun}{\on{Bun}}
\nc{\IC}{\on{IC}}
\nc{\Aut}{\on{Aut}}
\nc{\rk}{\on{rk}}
\nc{\Sh}{\on{Sh}}
\nc{\Stab}{\on{Stab}}
\nc{\Perv}{\on{Perv}}
\nc{\pos}{{\on{pos}}}
\nc{\Conv}{\on{Conv}}
\nc{\Sph}{\on{Sph}}
\nc{\Sym}{\on{Sym}}
\nc{\BunBb}{\overline{\Bun}_B}
\nc{\BunBbm}{\overline{\Bun}_{B^-}}
\nc{\BunBbel}{\overline{\Bun}_{B,el}}
\nc{\BunBbmel}{\overline{\Bun}_{B^-,el}}
\nc{\Buno}{\overset{o}{\Bun}}
\nc{\BunPb}{{\overline{\Bun}_P}}
\nc{\BunBM}{\Bun_{B(M)}}
\nc{\BunBMb}{\overline{\Bun}_{B(M)}}
\nc{\BunPbw}{{\widetilde{\Bun}_P}}
\nc{\BunBP}{\widetilde{\Bun}_{B,P}}
\nc{\GUb}{\overline{G/U}}
\nc{\GUPb}{\overline{G/U(P)}}

\nc{\Hhom}{\underline{\on{Hom}}}
\nc\syminfty{\on{Sym}^{\infty}}
\nc\lal{\ol{\lambda}}
\nc\xl{\ol{x}}
\nc\thl{\ol{\theta}}
\nc\nul{\ol{\nu}}
\nc\mul{\ol{\mu}}
\nc\Sum\Sigma
\nc{\oX}{\overset{o}{X}{}}
\nc{\hl}{\overset{\leftarrow}h{}}
\nc{\hr}{\overset{\rightarrow}h{}}
\nc{\M}{{\mathcal M}}
\nc{\N}{{\mathcal N}}
\nc{\F}{{\mathcal F}}
\nc{\D}{{\mathcal D}}
\nc{\Q}{{\mathcal Q}}
\nc{\Y}{{\mathcal Y}}
\nc{\G}{{\mathcal G}}
\nc{\E}{{\mathcal E}}
\nc{\CalC}{{\mathcal C}}
\nc\Dh{\widehat{\D}}

\nc{\C}{{\mathcal C}}
\nc{\K}{{\mathcal K}}
\renewcommand{\H}{{\mathcal H}}

\nc{\T}{{\mathcal T}}
\nc{\V}{{\mathcal V}}
\renc{\P}{{\mathcal P}}
\nc{\A}{{\mathcal A}}
\nc{\B}{{\mathcal B}}
\nc{\U}{{\mathcal U}}

\nc{\Gr}{{\on{Gr}}}

\nc{\frn}{{\check{\mathfrak u}(P)}}
\nc{\p}{\mathfrak p}
\nc{\q}{\mathfrak q}
\nc\f{{\mathfrak f}}

\nc{\qo}{{\mathfrak q}}
\nc{\po}{{\mathfrak p}}
\nc{\s}{{\mathfrak s}}
\nc\w{\text{w}}

\nc\mathi\iota
\nc\Spec{\on{Spec}}
\nc\Mod{\on{Mod}}
\nc{\tw}{\widetilde{\mathfrak t}}
\nc{\pw}{\widetilde{\mathfrak p}}
\nc{\qw}{\widetilde{\mathfrak q}}
\nc{\jw}{\widetilde j}

\nc{\grb}{\overline{\Gr}}
\nc{\I}{\mathcal I}

\nc{\lambdach}{{\check\lambda}}
\nc{\Lambdach}{{\check\Lambda}{}}
\nc{\much}{{\check\mu}}
\nc{\omegach}{{\check\omega}}
\nc{\nuch}{{\check\nu}}
\nc{\etach}{{\check\eta}}
\nc{\alphach}{{\check\alpha}}
\nc{\betach}{{\check\beta}}
\nc{\rhoch}{{\check\rho}}
\nc{\ch}{{\check h}}

\nc{\Hb}{\overline{\H}}

\nc{\BA}{{\mathbb{A}}}
\nc{\BC}{{\mathbb{C}}}
\nc{\BG}{{\mathbb{G}}}
\nc{\BL}{{\mathbb{L}}}
\nc{\BM}{{\mathbb{M}}}
\nc{\BD}{{\mathbb{D}}}
\nc{\BN}{{\mathbb{N}}}
\nc{\BP}{{\mathbb{P}}}
\nc{\BQ}{{\mathbb{Q}}}
\nc{\BR}{{\mathbb{R}}}
\nc{\BX}{{\mathbb{X}}}
\nc{\BZ}{{\mathbb{Z}}}
\nc{\BS}{{\mathbb{S}}}

\nc{\CA}{{\mathcal{A}}}
\nc{\CB}{{\mathcal{B}}}

\nc{\CE}{{\mathcal{E}}}
\nc{\CF}{{\mathcal{F}}}

\nc{\CL}{{\mathcal{L}}}
\nc{\CC}{{\mathcal{C}}}
\nc{\CM}{{\mathcal{M}}}
\nc{\CN}{{\mathcal{N}}}
\nc{\CK}{{\mathcal{K}}}
\nc{\CO}{{\mathcal{O}}}
\nc{\CP}{{\mathcal{P}}}
\nc{\CQ}{{\mathcal{Q}}}
\nc{\CR}{{\mathcal{R}}}
\nc{\CS}{{\mathcal{S}}}
\nc{\oCS}{\overset{\circ}{\mathcal{S}}}
\nc{\CT}{{\mathcal{T}}}
\nc{\CU}{{\mathcal{U}}}
\nc{\CV}{{\mathcal{V}}}
\nc{\CW}{{\mathcal{W}}}
\nc{\CX}{{\mathcal{X}}}
\nc{\CY}{{\mathcal{Y}}}
\nc{\CZ}{{\mathcal{Z}}}
\nc{\CI}{{\mathcal{I}}}

\nc{\csM}{{\check{\mathcal A}}{}}
\nc{\oM}{{\overset{\circ}{\mathcal M}}{}}
\nc{\obM}{{\overset{\circ}{\mathbf M}}{}}
\nc{\oCA}{{\overset{\circ}{\mathcal A}}{}}
\nc{\obA}{{\overset{\circ}{\mathbf A}}{}}
\nc{\ooM}{{\overset{\circ}{M}}{}}
\nc{\osM}{{\overset{\circ}{\mathsf M}}{}}
\nc{\vM}{{\overset{*}{\mathcal M}}{}}
\nc{\nM}{{\underset{*}{\mathcal M}}{}}
\nc{\oD}{{\overset{\circ}{\mathcal D}}{}}
\nc{\obD}{{\overset{\circ}{\mathbf D}}{}}
\nc{\oA}{{\overset{\circ}{\mathbb A}}{}}
\nc{\op}{{\overset{*}{\mathbf p}}{}}
\nc{\cp}{{\overset{\circ}{\mathbf p}}{}}
\nc{\oU}{{\overset{*}{\mathcal U}}{}}
\nc{\oZ}{{\overset{\circ}{\mathcal Z}}{}}
\nc{\ofZ}{{\overset{\circ}{\mathfrak Z}}{}}
\nc{\oF}{{\overset{\circ}{\fF}}}

\nc{\fa}{{\mathfrak{a}}}
\nc{\fb}{{\mathfrak{b}}}
\nc{\fd}{{\mathfrak{d}}}
\nc{\fg}{{\mathfrak{g}}}
\nc{\fgl}{{\mathfrak{gl}}}
\nc{\fh}{{\mathfrak{h}}}
\nc{\fj}{{\mathfrak{j}}}
\nc{\fl}{{\mathfrak{l}}}
\nc{\fm}{{\mathfrak{m}}}
\nc{\fn}{{\mathfrak{n}}}
\nc{\fu}{{\mathfrak{u}}}
\nc{\fp}{{\mathfrak{p}}}
\nc{\fr}{{\mathfrak{r}}}
\nc{\fs}{{\mathfrak{s}}}
\nc{\ft}{{\mathfrak{t}}}
\nc{\fsl}{{\mathfrak{sl}}}
\nc{\hsl}{{\widehat{\mathfrak{sl}}}}
\nc{\hgl}{{\widehat{\mathfrak{gl}}}}
\nc{\hg}{{\widehat{\mathfrak{g}}}}
\nc{\chg}{{\widehat{\mathfrak{g}}}{}^\vee}
\nc{\hn}{{\widehat{\mathfrak{n}}}}
\nc{\chn}{{\widehat{\mathfrak{n}}}{}^\vee}

\nc{\fA}{{\mathfrak{A}}}
\nc{\fB}{{\mathfrak{B}}}
\nc{\fD}{{\mathfrak{D}}}
\nc{\fE}{{\mathfrak{E}}}
\nc{\fF}{{\mathfrak{F}}}
\nc{\fG}{{\mathfrak{G}}}
\nc{\fK}{{\mathfrak{K}}}
\nc{\fL}{{\mathfrak{L}}}
\nc{\fM}{{\mathfrak{M}}}
\nc{\fN}{{\mathfrak{N}}}
\nc{\fP}{{\mathfrak{P}}}
\nc{\fU}{{\mathfrak{U}}}
\nc{\fV}{{\mathfrak{V}}}
\nc{\fZ}{{\mathfrak{Z}}}

\nc{\bb}{{\mathbf{b}}}
\nc{\bc}{{\mathbf{c}}}
\nc{\bd}{{\mathbf{d}}}
\nc{\be}{{\mathbf{e}}}
\nc{\bj}{{\mathbf{j}}}
\nc{\bn}{{\mathbf{n}}}
\nc{\bp}{{\mathbf{p}}}
\nc{\bq}{{\mathbf{q}}}
\nc{\bu}{{\mathbf{u}}}
\nc{\bv}{{\mathbf{v}}}
\nc{\bx}{{\mathbf{x}}}
\nc{\bs}{{\mathbf{s}}}
\nc{\by}{{\mathbf{y}}}
\nc{\bw}{{\mathbf{w}}}
\nc{\bA}{{\mathbf{A}}}
\nc{\bK}{{\mathbf{K}}}
\nc{\bB}{{\mathbf{B}}}
\nc{\bC}{{\mathbf{C}}}
\nc{\bG}{{\mathbf{G}}}
\nc{\bD}{{\mathbf{D}}}
\nc{\bH}{{\mathbf{H}}}
\nc{\bM}{{\mathbf{M}}}
\nc{\bN}{{\mathbf{N}}}
\nc{\bP}{{\mathbf{P}}}
\nc{\bV}{{\mathbf{V}}}
\nc{\bW}{{\mathbf{W}}}
\nc{\bX}{{\mathbf{X}}}
\nc{\bZ}{{\mathbf{Z}}}
\nc{\bS}{{\mathbf{S}}}

\nc{\sA}{{\mathsf{A}}}
\nc{\sB}{{\mathsf{B}}}
\nc{\sC}{{\mathsf{C}}}
\nc{\sD}{{\mathsf{D}}}
\nc{\sF}{{\mathsf{F}}}
\nc{\sK}{{\mathsf{K}}}
\nc{\sM}{{\mathsf{M}}}
\nc{\sO}{{\mathsf{O}}}
\nc{\sW}{{\mathsf{W}}}
\nc{\sQ}{{\mathsf{Q}}}
\nc{\sP}{{\mathsf{P}}}
\nc{\sZ}{{\mathsf{Z}}}
\nc{\sfj}{{\mathsf{j}}}
\nc{\sfp}{{\mathsf{p}}}
\nc{\sfq}{{\mathsf{q}}}
\nc{\sfr}{{\mathsf{r}}}
\nc{\osfr}{\overset{\circ}{\mathsf{r}}}
\nc{\sr}{{\mathsf{r}}}
\nc{\sfi}{{\mathsf{i}}}
\nc{\sk}{{\mathsf{k}}}
\nc{\sg}{{\mathsf{g}}}
\nc{\sff}{{\mathsf{f}}}
\nc{\sfb}{{\mathsf{b}}}
\nc{\sfc}{{\mathsf{c}}}
\nc{\sd}{{\mathsf{d}}}

\nc{\BK}{{\bar{K}}}

\nc{\tA}{{\widetilde{\mathbf{A}}}}
\nc{\tB}{{\widetilde{\mathcal{B}}}}
\nc{\tg}{{\widetilde{\mathfrak{g}}}}
\nc{\tG}{{\widetilde{G}}}
\nc{\TM}{{\widetilde{\mathbb{M}}}{}}
\nc{\tO}{{\widetilde{\mathsf{O}}}{}}
\nc{\tU}{{\widetilde{\mathfrak{U}}}{}}
\nc{\TZ}{{\tilde{Z}}}
\nc{\tx}{{\tilde{x}}}
\nc{\tbv}{{\tilde{\bv}}}
\nc{\tfP}{{\widetilde{\mathfrak{P}}}{}}
\nc{\tz}{{\tilde{\zeta}}}
\nc{\tmu}{{\tilde{\mu}}}

\nc{\urho}{\underline{\rho}}
\nc{\uB}{\underline{B}}
\nc{\uC}{{\underline{\mathbb{C}}}}
\nc{\ui}{\underline{i}}
\nc{\uj}{\underline{j}}
\nc{\ofP}{{\overline{\mathfrak{P}}}}
\nc{\oB}{{\overline{\mathcal{B}}}}
\nc{\og}{{\overline{\mathfrak{g}}}}
\nc{\oI}{{\overline{I}}}

\nc{\eps}{\varepsilon}
\nc{\hrho}{{\hat{\rho}}}

\nc{\one}{{\mathbf{1}}}
\nc{\two}{{\mathbf{t}}}

\nc{\Rep}{{\mathop{\operatorname{\rm Rep}}}}
\nc{\Tot}{{\mathop{\operatorname{\rm Tot}}}}
\nc{\Ker}{{\mathop{\operatorname{\rm Ker}}}}
\nc{\Hilb}{{\mathop{\operatorname{\rm Hilb}}}}
\nc{\End}{{\mathop{\operatorname{\rm End}}}}
\nc{\Ext}{{\mathop{\operatorname{\rm Ext}}}}
\nc{\CHom}{{\mathop{\operatorname{{\mathcal{H}}\it om}}}}
\nc{\GL}{{\mathop{\operatorname{\rm GL}}}}
\nc{\gr}{{\mathop{\operatorname{\rm gr}}}}
\nc{\Id}{{\mathop{\operatorname{\rm Id}}}}
\nc{\de}{{\mathop{\operatorname{\rm def}}}}
\nc{\length}{{\mathop{\operatorname{\rm length}}}}
\nc{\supp}{{\mathop{\operatorname{\rm supp}}}}

\nc{\Cliff}{{\mathsf{Cliff}}}
\nc{\Fl}{\on{Fl}}
\nc{\Fib}{{\mathsf{Fib}}}
\nc{\Coh}{{\mathsf{Coh}}}
\nc{\FCoh}{{\mathsf{FCoh}}}

\nc{\reg}{{\text{\rm reg}}}

\nc{\cplus}{{\mathbf{C}_+}}
\nc{\cminus}{{\mathbf{C}_-}}
\nc{\cthree}{{\mathbf{C}_*}}
\nc{\Qbar}{{\bar{Q}}}
\nc\Eis{\on{Eis}}
\nc\Eisb{\ol\Eis{}}
\nc\wh{\widehat}
\nc{\Def}{\on{Def_{\check{\fb}}(E)}}
\nc{\barZ}{\overline{Z}{}}
\nc{\barbarZ}{\overline{\barZ}{}}
\nc{\barpi}{\overline\pi}
\nc{\barbarpi}{\overline\barpi}
\nc{\barpip}{\overline\pi{}^+}
\nc{\barpim}{\overline\pi{}^-}

\nc{\fq}{\mathfrak q}

\nc{\sfqb}{\ol{\sfq}{}}
\nc{\sfpb}{\ol{\sfp}{}}

\nc{\hattimes}{\wh\otimes}

\nc{\bh}{{\bar{h}}}
\nc{\bOmega}{{\overline{\Omega(\check \fn)}}}

\nc{\seq}[1]{\stackrel{#1}{\sim}}

\nc{\cT}{{\check{T}}}
\nc{\cG}{{\check{G}}}
\nc{\cM}{{\check{M}}}
\nc{\cB}{{\check{B}}}

\nc{\ct}{{\check{\mathfrak t}}}
\nc{\cg}{{\check{\fg}}}
\nc{\cb}{{\check{\fb}}}
\nc{\cn}{{\check{\fn}}}

\nc{\cLambda}{{\check\Lambda}}

\nc{\cla}{{\check\lambda}}
\nc{\cmu}{{\check\mu}}
\nc{\cnu}{{\check\nu}}
\nc{\ceta}{{\check\eta}}
\nc{\cP}{\check{P}}

\nc{\DefbE}{{\on{Def}_{\cB}(E_\cT)}}

\nc{\imathb}{{\ol{\imath}}}
\nc{\Dmod}{\on{D-mod}}
\nc{\Maps}{\on{Maps}}
\nc{\Vect}{\on{Vect}}
\nc{\sotimes}{\overset{!}\otimes}
\nc{\IndCoh}{\on{IndCoh}}
\nc{\LocSys}{\on{LocSys}}

\nc{\red}{\on{red}}

\begin{document}

\title{Geometric constant term functor(s)}

\author{V.~Drinfeld and D.~Gaitsgory}

\dedicatory{To Joseph Bernstein with deepest gratitude}

\date{\today}

\begin{abstract}
We study the Eisenstein series and constant term functors in the framework of geometric
theory of automorphic functions. Our main result says that for a parabolic $P\subset G$
with Levi quotient $M$, the !-constant term functor
$$\on{CT}_!:\Dmod(\Bun_G)\to \Dmod(\Bun_M)$$ is canonically isomorphic to the
*-constant term functor $$\on{CT}^-_*:\Dmod(\Bun_G)\to \Dmod(\Bun_M),$$
taken with respect to the opposite parabolic $P^-$.
\end{abstract} 

\maketitle

\section*{Introduction}

\ssec{Conventions}   \label{ss:Conventions}
We fix an algebraically closed field $k$ of characteristic $0$. Unless stated otherwise, we say 
``scheme" instead of ``scheme locally of finite type over $k$." 
When we say ``stack" we mean an algebraic stack locally of finite type over $k$ such that the automorphism 
group of any $k$-point is affine.
\footnote{The latter condition is called ``locally QCA" in \cite{DrGa1}. 
It ensures that the theory of D-modules on our stack is reasonable.} 

\medskip

A morphism of stacks $f:Y_1\to Y_2$ is said to be \emph{representable} (resp.~ \emph{schematic}) if for any scheme $S$ the stack $Y_1\underset{Y_2}\times S$ 
is an algebraic space (resp. scheme).

\ssec{Posing the problem}

\sssec{}

One of the main tools in studying automorphic functions on an ad\`ele group $G$ is the pair of mutually
adjoint operators, called  ``Eisenstein series" and ``constant term" that connect this space to similar spaces
for Levi subgroups.

\medskip

The goal of this paper is to make several basic observations regarding the analogs of these operators
in the geometric context.

\sssec{}

Let $X$ be a smooth projective connected curve over $k$, and $G$ a reductive group. 
Let $\Bun_G$ be the moduli stack of principal $G$-bundles on $X$. Our geometric analog of the space of automorphic
functions is the DG category of (not necessarily holonomic) D-modules on $\Bun_G$, denoted by 
$\Dmod(\Bun_G)$.  We refer the reader to \secref{sss:conventions}
for the explanation of what exactly we understand by ``DG category of D-modules." Here we just mention that the 
homotopy category of this DG category is the derived category of D-modules.

\medskip

Let $P$ be a parabolic in $G$ with Levi quotient $M$. We have the following fundamental diagram of stacks:
\begin{equation} \label{e:basic diag}
\xy
(15,0)*+{\Bun_G}="X";
(-15,0)*+{\Bun_M}="Y";
(0,15)*+{\Bun_P}="Z";
{\ar@{->}^{\sfp} "Z";"X"};
{\ar@{->}_{\sfq} "Z";"Y"};
\endxy
\end{equation}

Recall that the set of connected components of the stack $\Bun_M$ is in bijection with $\pi_1(M)$. For an element $\mu\in \pi_1(M)$
we let $\Bun_M^\mu$ denote the corresponding connected component; let $\Bun_P^\mu$ denote the preimage of
$\Bun_M^\mu\subset \Bun_M$ under the above map $\sfq:\Bun_P\to \Bun_M$; one can show that $\Bun_P^\mu$ is connected,
i.e., is a single connected component of $\Bun_P$. 

\medskip

By a slight abuse of notation we will denote
by the same symbols $\sfq$ and $\sfp$, respectively, the restrictions of the corresponding maps to $\Bun_P^\mu\subset \Bun_P$. 
Thus, we obtain a diagram
\begin{equation} \label{e:basic diag mu}
\xy
(15,0)*+{\Bun_G}="X";
(-15,0)*+{\Bun^\mu_M}="Y";
(0,15)*+{\Bun^\mu_P}="Z";
{\ar@{->}^{\sfp} "Z";"X"};
{\ar@{->}_{\sfq} "Z";"Y"};
\endxy
\end{equation}

\sssec{}

Applying pull-push along diagram \eqref{e:basic diag mu} we obtain the functors
$$\Eis^\mu_*:\Dmod(\Bun^\mu_M)\to \Dmod(\Bun_G), \quad \Eis^\mu_*:=\sfp_*\circ \sfq^!$$
and
$$\on{CT}^\mu_*: \Dmod(\Bun_G)\to \Dmod(\Bun^\mu_M), \quad  \on{CT}^\mu_*:=\sfq_*\circ \sfp^!$$
(here $\Eis$ stands for ``Eisenstein" and $\on{CT}$ stands for ``constant term").

\medskip

We will address the following questions:

\medskip

\begin{quest}   \label{quest1new}
Can one define the functor
$$\Eis^\mu_!:=\sfp_!\circ \sfq^*: \Dmod(\Bun^\mu_M)\to \Dmod(\Bun_G)\,,$$
left adjoint to $\on{CT}^\mu_*$\,\emph{?}
\end{quest}

\begin{quest}   \label{quest2new}
Can one define the functor 
$$\on{CT}^\mu_!:=\sfq_!\circ \sfp^*:\Dmod(\Bun_G)\to \Dmod(\Bun^\mu_M)\,,$$
left adjoint to $\Eis^\mu_*\,$\emph{?}
\end{quest}

The next few remarks explain why these questions are non-obvious.

\begin{rem}   \label{r:1}
For a morphism
$f:Y_1\to Y_2$ between schemes or stacks the functors
$$f_*:\Dmod(Y_1)\to \Dmod(D_2) \text{ and }f^!:\Dmod(Y_1)\to \Dmod(D_2)$$
are well-defined. But their left adjoints
$$f^*:\Dmod(Y_2)\to \Dmod(D_1) \text{ and }f_!:\Dmod(Y_2)\to \Dmod(D_1)$$
are only \emph{partially} defined (because we work with non-necessarily holonomic D-modules).
\end{rem}

\begin{rem}   \label{r:2}
If $f$ is quasi-compact, representable 
and proper then $f_!$ is always defined and equals $f_*\,$.
If $f$ is smooth then $f^*$ is always defined and equals $f^!$ up to a cohomological shift.
 \end{rem}
 
\begin{rem}   \label{r:3}
If a morphism  of stacks $f:Y_1\to Y_2$ is \emph{safe} then 
the functor $f_*$ is ``well-behaved" (more precisely, $f_*$ is \emph{continuous}).
For more details (including the definitions of safety and continuity), see \secref{sss:safety}.
\end{rem}
 
\begin{rem}  \label{r:4}
In diagram \eqref{e:basic diag mu} the morphism $\sfp$ is quasi-compact and representable 
(and moreover, schematic), but neither smooth nor proper. 
The morphism $\sfq$ is smooth and safe 
\footnote{For any surjection of algebraic groups $\phi:H_1\to H_2$
the corresponding map $\Phi:\Bun_{H_1}\to \Bun_{H_2}$ is smooth, which can be seen through the calculation of
its differential. If $\on{ker}(\phi)$ is unipotent, then $\Phi$ is safe.} (but not representable).  
 \end{rem}

\ssec{Statement of the results}

\sssec{}

First, we show that the functor $\Eis^\mu_!:=\sfp_!\circ \sfq^*$ is well-defined. (More precisely, since 
$\sfq$ is smooth the functor $\sfq^*$ is well-defined, and
we will show that $\sfp_!$ is well-defined on the essential image of $\sfq^*$).

\sssec{}

We now turn to $\on{CT}^\mu_!$. We do eventually show that the functor $\on{CT}^\mu_!$ is well-defined,
but this comes as a result of a more precise assertion (see \thmref{t:prel main} below) that describes this functor explicitly. 

\medskip

Here we note that a direct attempt to define $\on{CT}^\mu_!=\sfq_!\circ \sfp^*$ brings problems different from 
those in the case of $\Eis^\mu_!$. Namely, it turns out that the second functor, i.e., $\sfq_!$, is
always well-defined (see \secref{sss:contraction on moduli simple}), while the first one, namely $\sfp^*$, is not. However, 
we will show that their composition is well-defined in a certain sense (see \secref{sss:hyperb stacks} for the general 
paradigm when such compositions are well-defined; we will show that this paradigm is applicable to $\Bun_G$ in 
\secref{ss:proof via Braden} and \secref{sss:contraction on moduli simple}). 

\sssec{}

Now we are ready to state the main result of this paper, namely, \thmref{t:main}:

\begin{thm}  \label{t:prel main}
The functor $\on{CT}^\mu_!$ exists and is isomorphic to the functor $\on{CT}^{\mu,-}_*$, where the superscript 
``$\,^-$" means that instead of $P$ we are considering the opposite parabolic $P^-$.
\end{thm}

\medskip

The assertion of this theorem can be viewed as some kind of non-standard functional equation.
It does not have an immediate analog in the classical theory of automorphic functions
(where one has only one type of pullback operator and one type of push forward).

\medskip

\thmref{t:prel main} has an implication to the relation between the functors
$\Eis_!$ and $\Eis_*$, which is discussed in \cite[Theorem 4.1.2]{Ga}.
Here we will only mention that this implication \emph{does} have a manifestation in the classical 
theory of automorphic functions.

\ssec{Relation to the geometric Langlands conjecture}

The contents of this subsection play a motivational role and may be skipped by the reader.

\sssec{}

In addition to the functors 
$$\Eis^\mu_*,\quad \on{CT}^\mu_*,\quad \Eis_!^\mu,\quad \on{CT}_!^\mu$$
one can consider the full Eisenstein and constant term functors:
$$\Eis_*:=\underset{\mu\in \pi_1(M)}\oplus\, \Eis^\mu_*,\,\,\,\,\Eis_!:=\underset{\mu\in \pi_1(M)}\oplus\, \Eis^\mu_!$$
and 
$$\on{CT}_*:=\underset{\mu\in \pi_1(M)}\oplus\, \on{CT}^\mu_*\simeq \underset{\mu\in \pi_1(M)}\Pi\, \on{CT}^\mu_*,$$
$$\on{CT}_!:=\underset{\mu\in \pi_1(M)}\oplus\, \on{CT}^\mu_!\simeq \underset{\mu\in \pi_1(M)}\Pi\, \on{CT}^\mu_!,$$
(where the sum equals the product because each $\on{CT}^\mu_*$ and $\on{CT}^\mu_!$ lands in its own component
of $\Bun_M^\mu$).

\medskip

Tautologically, the functors $(\Eis_!,\on{CT}_*)$ form an adjoint pair. 

\medskip

Note, however, that the functors $\on{CT}_!$ and $\Eis_*$ \emph{do not} form an adjoint pair. Indeed, the right adjoint of
$\on{CT}_!$ is given by $\underset{\mu\in \pi_1(M)}\Pi\, \Eis^\mu_*$ (which is different from 
$\Eis_*:=\underset{\mu\in \pi_1(M)}\oplus\, \Eis^\mu_*$);
in particular, the above right adjoint is not continuous. 

\medskip

Define also 
$$\on{CT}_*^-:=\underset{\mu\in \pi_1(M)}\oplus\, \on{CT}^{\mu,-}_*.$$

Of course, \thmref{t:prel main} implies that we have a canonical isomorphism
$$\on{CT}_!\simeq \on{CT}_*^-.$$

\sssec{}

It turns out that the functor $\Eis_!$, although less straightforward to define than $\Eis_*$, \emph{plays a more fundamental role from the point of the geometric Langlands conjecture}. 

\medskip

Namely, according to \cite{AG}, this conjecture predicts 
an equivalence of categories 
$$\BL_G:\Dmod(\Bun_G)\to \IndCoh_{\on{Nilp}_{glob}}(\LocSys_\cG),$$
where $\IndCoh_{\on{Nilp}_{glob}}(\LocSys_\cG)$ is a certain modification of the DG category of quasi-coherent sheaves
on the stack $\LocSys_\cG$ of local systems on $X$ with respect to the Langands dual group $\cG$. 

\medskip 

Now, the equivalence $\BL_G$ is expected to be compatible with the corresponding equivalence $\BL_M$ for its 
Levi subgroup $M$
via the diagram
\begin{equation} \label{e:Langlands Eis}
\CD
\Dmod(\Bun_G)   @>{\BL_G}>> \IndCoh_{\on{Nilp}_{glob}}(\LocSys_\cG) \\
@A{\Eis_!}AA    @AA{\Eis_{\on{spec}}}A   \\
\Dmod(\Bun_M)   @>{\BL_M}>> \IndCoh_{\on{Nilp}_{glob}}(\LocSys_\cM), 
\endCD
\end{equation}
which commutes up to an auto-equivalence of $\IndCoh_{\on{Nilp}_{glob}}(\LocSys_\cM)$,
given by tensoring with a certain canonically defined graded line bundle on $\LocSys_\cM$.

\medskip

In the above diagram, $\Eis_{\on{spec}}$ is the \emph{spectral Eisenstein series functor}, defined as pull-push
along the diagram
\begin{equation} \label{e:basic diagram spec}
\xy
(-15,0)*+{\LocSys_\cG}="X";
(15,0)*+{\LocSys_{\cM},}="Y";
(0,15)*+{\LocSys_{\cP}}="Z";
{\ar@{->}_{\sfp_{\on{spec}}} "Z";"X"};
{\ar@{->}^{\sfq_{\on{spec}}} "Z";"Y"};
\endxy
\end{equation}
see \cite[Sect. 13.2]{AG} for more details. 

\medskip

The point that we would like to emphasize is that \emph{the commutation in diagram \eqref{e:Langlands Eis}
takes place for the functor $\Eis_!$ and not for $\Eis_*$} (moreover, this cannot be remedied by any auto-equivalence of the category $\Dmod(\Bun_G)$\,\footnote{See, however, \secref{sss:! functor} below.}.)

\medskip

It is possible to explicitly describe the functor
$$\IndCoh_{\on{Nilp}_{glob}}(\LocSys_\cM)\to \IndCoh_{\on{Nilp}_{glob}}(\LocSys_\cG)$$
that corresponds via $\BL_G$ and $\BL_M$ to 
$$\Eis_*:\Dmod(\Bun_M)\to \Dmod(\Bun_G),$$
but this description is more involved.

\sssec{}

The fact that the functor $\Eis_!$ is more fundamental than $\Eis_*$ can also be explained as follows.

\medskip

Recall (see \cite[Sect. 4.3.3]{DrGa2}) that in addition to the category $\Dmod(\Bun_G)$, there exists another
DG category, denoted $\Dmod(\Bun_G)_{\on{co}}$, that can be naturally assigned to the stack $\Bun_G$.

\medskip

Furthermore, according to \cite[Sect. 4.4.3]{DrGa2}, there is a canonically defined functor
$$\on{Ps-Id}_{\Bun_G,\on{naive}}:\Dmod(\Bun_G)_{\on{co}}\to \Dmod(\Bun_G),$$
which is \emph{not} an equivalence, unless $G$ is a torus.

\medskip

Now, it is not difficult to see that there is a naturally defined functor
$$\Eis_{*,\on{co}}:\Dmod(\Bun_M)_{\on{co}}\to \Dmod(\Bun_G)_{\on{co}}$$
that makes the following diagram commute:
$$
\CD
\Dmod(\Bun_G)_{\on{co}}    @>{\on{Ps-Id}_{\Bun_G,\on{naive}}}>>  \Dmod(\Bun_G)  \\
@A{\Eis_{*,\on{co}}}AA      @A{\Eis_*}AA    \\
\Dmod(\Bun_M)_{\on{co}}    @>{\on{Ps-Id}_{\Bun_M,\on{naive}}}>> \Dmod(\Bun_M).
\endCD
$$

Thus, the functor $\Eis_*$ is a coarsening of $\Eis_{*,\on{co}}$, since, informally, the functor
$\on{Ps-Id}_{\Bun_G,\on{naive}}$ ``loses more information" than $\on{Ps-Id}_{\Bun_M,\on{naive}}$. 

\sssec{}  \label{sss:! functor}

Finally, we note that according to \cite[Sect. 4.4.8]{DrGa2}, in addition to the functor $\on{Ps-Id}_{\Bun_G,\on{naive}}\,$,
there is another canonically defined functor
$$\on{Ps-Id}_{\Bun_G,!}:\Dmod(\Bun_G)_{\on{co}}\to \Dmod(\Bun_G).$$

It is shown in \cite[Theorem 3.1.5]{Ga} that the functor $\on{Ps-Id}_{\Bun_G,!}$ \emph{is actually an equivalence of categories}.

\medskip

In  \cite[Theorem 4.1.2]{Ga} it is also shown that the following diagram commutes:
\begin{equation} \label{e:! functor}
\CD
\Dmod(\Bun_G)_{\on{co}}    @>{\on{Ps-Id}_{\Bun_G,!}}>>  \Dmod(\Bun_G)  \\
@A{\Eis_{*,\on{co}}}AA      @A{\Eis_!^-}AA    \\
\Dmod(\Bun_M)_{\on{co}}    @>{\on{Ps-Id}_{\Bun_M,!}}>> \Dmod(\Bun_M).
\endCD
\end{equation}

So, to summarize, although the naive functor $\Eis_*$ plays an inferior role to that of $\Eis_!$,
its counterpart 
$$\Eis_{*,\on{co}}:\Dmod(\Bun_M)_{\on{co}}\to \Dmod(\Bun_G)_{\on{co}}$$
is on par with $\Eis_!$ by virtue of being intertwined by the equivalences
$\on{Ps-Id}_{\Bun_G,!}$ and $\on{Ps-Id}_{\Bun_M,!}$. 

\medskip

In \cite[Sect. 0.2]{Ga}, it is also explained how diagram \eqref{e:! functor} expresses
the compatibility of the Langlands correspondence functors $\BL_G$ (resp., $\BL_M$)
with Verdier duality on $\Bun_G$ (resp., $\Bun_M$) and Serre duality on
$\LocSys_\cG$ (resp., $\LocSys_\cM$). 

\ssec{Method of proof} \label{ss:Method of proof} 

The proof of \thmref{t:prel main} is a variation on the theme of a theorem of T.~ Braden on hyperbolic
restrictions (see \cite{Br}), recently revisited in \cite{DrGa3}. 
In fact, we give two proofs, in Sects. \ref{s:proof of adj}-\ref{s:Verifying} and Sect. \ref{s:mainproof}, respectively. 

\sssec{}

The first proof mimics the new proof of Braden's theorem given in \cite{DrGa3}, and it directly establishes the 
$(\on{CT}_*^{\mu,-},\Eis_*^\mu)$-adjunction by specifying the unit and co-unit morphisms. 

\medskip

As in the case of the new proof of Braden's theorem given in \cite{DrGa3},
the co-unit morphism is straightforward, and essentially corresponds to the embedding
of the big Bruhat cell into $P\backslash G/P^-$. 

\sssec{}

The unit of the adjunction uses a certain geometric construction, namely, an $\BA^1$-family of subgroups
$\wt{G}$ of $G\times G$, whose fiber at $1\in \BA^1$ is the diagonal copy of $G$, and whose fiber at $0\in \BA^1$
is the subgroup
$$P\underset{M}\times P^-\subset P\times P^-\subset G\times G.$$

\medskip

For the definition of the group-subscheme $\wt{G}\subset \BA^1\times G\times G$, see Sects.~\ref{sss:hyperbolas}, 
\ref{sss:tilde Z}, and \ref{ss:Interpolating}.  Here let us just mention that $\wt{G}$
depends on the choice of a co-character 
\begin{equation} \label{e:cochar preview}
\gamma:\BG_m\to M,
\end{equation}
which lands in the center of $M$, and which is dominant and regular with respect to $P$. 

\sssec{}  \label{sss:Vin}

The group-scheme $\wt{G}$ is not new in Lie theory. Namely, it can be recovered from the
\emph{Vinberg semi-group} correspondingg to $G$ (a.k.a. the enveloping semi-group of $G$); 
see Appendix \ref{s:Vinberg}, where this is explained. 

\medskip

The Lie algebra $\Lie (\wt{G})$ can be directly recovered from the ``wonderful compactification" of $G$ defined in \cite{DCP}: 

\medskip

Let $W$ denote the variety of Lie subalgebras of 
$\fg\times\fg$, where $\fg:=\Lie (G)$. Let $\gamma:\BG_m\to M$ denote the co-character 
\eqref{e:cochar preview}. For $t\in\BG_m$ let $\Gamma_t\subset\fg\times\fg$ denoted the graph of
$\Ad_{\mu (t)}:\fg\iso\fg\,$. Since $W$ is projective the map $\BG_m\to W$ defined by 
$t\mapsto\Gamma_t$ extends to a morphism $\BA^1\to W$, whose image is contained in the ``wonderful compactification." 
Thus one gets an $\BA^1$-family of Lie subalgebras of $\fg\times\fg$. This $\BA^1$-family is 
$\Lie (\wt{G})$ (because $\wt{G}$ is smooth over $\BA^1$, see \propref{p:smooth}).

\sssec{}

The statement that the functors $(\on{CT}_*^{\mu,-},\Eis^\mu_*)$ form an adjoint pair bears a strong resemblance
to the Second Adjointness Theorem in the theory of $\fp$-adic groups. 

\medskip

By a slight abuse of notation, let us temporarily denote by $G$ the set of points of a reductive group over a 
local non-archimedian field, and consider the corresponding groups 
$$M\twoheadleftarrow P\hookrightarrow G.$$ We have the usual pair of adjoint functors 
$$(r^G_M,i^G_M),$$
where $i^G_M$ is the functor of parabolic induction, and $r^G_M$ is the Jacquet functor. 

\medskip

Now, a theorem of J.~Bernstein (unpublished) says that, in addition to being the \emph{right} adjoint of $r^G_M$,
the functor $i^G_M$ is also the \emph{left} adjoint of $r^{G,-}_M$, where the latter is the Jacquet functor
with respect to the opposite parabolic $P^-$. 

\medskip

The proof of this result, given recently in \cite{BKa} (which is different from the original proof of Bernstein), 
is very close in spirit to our proof of the $(\on{CT}_*^{\mu,-},\Eis^\mu_*)$-adjunction: 

\medskip

One constructs the unit of the adjunction using the big Bruhat cell in $N(P)\backslash G/N(P^-)$.

\medskip

The co-unit of the adjunction uses (the set of points over our field of) the family of schemes over $\BA^1$
$$(\BA^1\times G\times G)/\wt{G},$$
acted on by $G\times G$. Note that this family interpolates between
$G$ (the fiber at $1\in \BA^1$) and $\left(G/N(P)\times G/N(P^-)\right)/M$ (the fiber at $0\in \BA^1$). 

\sssec{Remark}
The proof of Braden's theorem given in \cite{DrGa3} is paraphrased in \cite[Sect.~C.14]{Dr} using the categorical formalism of ``lax actions by correspondences". This formalism can also be used to prove Theorem~\ref{t:prel main}. However, we will not use it in this article.

\sssec{}

We now turn to the second proof of \thmref{t:prel main}, given in \secref{s:mainproof}. This proof
is obtained by deducing the isomorphism 
$$\on{CT}^\mu_!\simeq \on{CT}^{\mu,-}_*$$
from Braden's theorem, which involves \emph{schemes} acted on by $\BG_m$. 

\medskip

The schemes in question are obtained by replacing $\Bun_G$, $\Bun_P$, $\Bun_{P^-}$ and $\Bun_M$
by their versions when one considers a sufficiently deep level structure at one point of the curve. 

\sssec{}

Finally, we remark that \thmref{t:prel main} is analogous to the  corresponding 
theorem of Lusztig on \emph{restriction of character sheaves}, see \cite[Theorem 4.1 (iv)]{Gi}.

\medskip

In fact, the two statements admit a common generalization when instead of 
$\Bun_G$ we consider the moduli stack of $G$-bundles with level structure at a finite collection of 
points of $X$; the case of character sheaves on $G$ (resp., the Lie algebra $\fg$) corresponds to the case 
of $X=\BP^1$ with structure of level $1$ at $(0,\infty)\in \BP^1$ (resp., structure of  level $2$ ar $\infty\in \BP^1$). 

\medskip

Just as \thmref{t:prel main}, Lusztig's theorem can be deduced from Braden's theorem.
We learned this very simple proof of Lusztig's theorem from folklore and wrote it up in 
\cite[Sect. 0.2]{DrGa3}.

\ssec{Recollections on D-modules on stacks}

\sssec{DG categories}

Our conventions and notation pertaining to DG categories follow those of \cite[Sect. 1]{DrGa2}.

\sssec{The category $\Dmod(\CY)$, where $\CY$ is a stack} \label{sss:conventions}

For any stack
$\CY$ let $\Dmod(\CY)$ denote the DG category of D-modules on $\CY$ as defined in 
\cite{DrGa1} (the case of quasi-compact schemes is considered in \cite[Sect. 5.1]{DrGa1}  
and the general case in \cite[Sect. 6.2]{DrGa1}).

\medskip

If $\CY=S$ is a quasi-compact scheme then the homotopy category of the DG category $\Dmod(S)$ is the usual derived 
category of D-modules on $S$, 
but note that we impose \emph{no boundedness or coherence conditions \footnote{This is because we want $\Dmod(S)$ to 
be cocomplete, see below.}.} 

\medskip

If $\CY$ is a stack then $\Dmod(\CY)$ is defined to be the (projective) limit of the 
categories $\Dmod(S)$ over the indexing category of quasi-compact schemes $S$ mapping smoothly to $\CY$.
Informally, an object of $\Dmod(\CY)$ is a compatible collection of objects of $\Dmod(S)$ for all quasi-compact schemes 
$S$ mapping smoothly to $\CY$.

\medskip

As is explained in \cite{DrGa1}, the DG category $\Dmod(\CY)$ is \emph{cocomplete,} i.e., it has arbitrary colimits.

\sssec{Direct images and safety}  \label{sss:safety}
For any morphism $f:\CY_1\to \CY_2$ between stacks one has the direct image functor
$f_*:\Dmod(\CY_1)\to\Dmod(\CY_2)$ (see \cite[Sect. 7.4]{DrGa1} for the definition of $f_*$ in the
case that $f$ is not necessarily quasi-compact schematic).

\medskip

The functor $f_*$ is not necessarily  \emph{continuous}. (Recall that a functor between cocomplete DG categories is 
said to be continuous if it commutes with (infinite) colimits, or equivalently, with (infinite) direct sums.)

\medskip

However, according to \cite[Sect. 10.2]{DrGa1}, $f_*$ is continuous if $f$ is \emph{safe}. By definition,  safety means that $f$ is 
quasi-compact and has the following property: for any $y\in \CY_2(k)$  the neutral 
connected component of the automorphism group of any $k$-point of the fiber $(\CY_1)_y$ is unipotent. In particular, any 
quasi-compact representable morphism is safe. 

\sssec{Terminological remark} 
Following the conventions of higher category
theory, we call a morphism between two objects in a DG category an \emph{isomorphism} if and only 
if it is such in the homotopy category. 

\ssec{Organization of the article} 
In Sect.~\ref{s:The Statement} we reformulate Theorem~\ref{t:prel main} as an existence of an adjunction (see Theorem~\ref{t:main adj}), and we describe the natural transformation that will turn out to be the co-unit of the adjunction. We also discuss the notion of cuspidal object of $\CF\in \Dmod(\Bun_G)$; the main point is that two \emph{a priori} different notions of cuspidality coincide.

\medskip

As already said in \ref{ss:Method of proof}, we give two proofs of Theorem~\ref{t:main adj}. 

\medskip

The first one is given in Sections~\ref{s:proof of adj}-\ref{s:Verifying}. (In Sect.~\ref{s:proof of adj} we construct the natural transformation that will turn out to be the unit, and in Sect.~\ref{s:Verifying} we verify that the natural transformations constructed in Sections~\ref{s:The Statement} and \ref{s:proof of adj}  satisfy the properties of unit and count of an adjunction.)

\medskip

The second proof of Theorem~\ref{t:main adj} is given in  Sect.~\ref{s:mainproof}.

\medskip

In Appendix \ref{s:check hyperb} we prove a technical Theorem~\ref{t:hyperb}, which is used in the second proof of 
Theorem~\ref{t:main adj}. 

\medskip

In Appendix \ref{s:support} we prove \propref{p:cuspidality} that describes support of cusipdal objects of 
$\Dmod(\Bun_G)$.

\medskip

In Appendix \ref{s:quasiaffine} we prove \propref{p:quasiaffine}, which says that the
quotient of $\BA^1\times  G\times G$ by the ``interpolating" group-subscheme 
$\wt{G}\subset \BA^1\times  G\times G$ is a quasi-affine scheme.

\medskip

In Appendix \ref{s:Vinberg} we describe $\wt{G}$ and $(\BA^1\times  G\times G)/\wt{G}$ in terms of the Vinberg semigroup of $G$ (a.k.a. enveloping semigroup of $G$).

\ssec{Acknowledgements} 
The research of V. D. is partially supported by NSF grants DMS-1001660 and DMS-1303100. The research of
D. G. is partially supported by NSF grant DMS-1063470.

\section{The Statement}  \label{s:The Statement}

\ssec{The functor $\Eis^\mu_!$} 

\sssec{}

Consider the diagram 
\begin{equation} \label{e:basic diagram again}
\xy
(15,0)*+{\Bun_G}="X";
(-15,0)*+{\Bun^\mu_M.}="Y";
(0,15)*+{\Bun^\mu_P}="Z";
{\ar@{->}^{\sfp} "Z";"X"};
{\ar@{->}_{\sfq} "Z";"Y"};
\endxy
\end{equation}
By Remark~\ref{r:4},
the map $\sfq$ is smooth, so
the functor $\sfq^*$, left adjoint to $\sfq_*$, is well-defined. We are going to prove:

\begin{prop}  \label{existence of Eis !}
The partially defined left adjoint $\sfp_!$ to $\sfp^!$ is defined on the essential
image of the functor $\sfq^*$.
\end{prop}

\begin{cor}
The functor
$\Eis^\mu_!:=\sfp_!\circ \sfq^*:\Dmod(\Bun^\mu_M)\to \Dmod(\Bun_G)$, left adjoint to $\on{CT}^\mu_*$, is well-defined.
\end{cor}

The proof of Proposition~\ref{existence of Eis !}, given below, is based on some results of \cite{BG}.
Let us recall them. 

\sssec{}  \label{sss:BunPtilde}

First,
the diagram \eqref{e:basic diagram again} was extended  in \cite[Sect. 1.2]{BG}  to a diagram
\begin{gather}  \label{comp diag}
\xy
(20,0)*+{\Bun_G}="X";
(-20,0)*+{\Bun^\mu_M}="Y";
(0,20)*+{\wt{\Bun}^\mu_P}="Z";
(-20,20)*+{\Bun^\mu_P}="W";
{\ar@{->}^{\wt{\sfp}} "Z";"X"};
{\ar@{->}_{\wt\sfq} "Z";"Y"};
{\ar@{^{(}->}^r "W";"Z"};  
\endxy
\end{gather}
so that $\wt\sfp$ is proper, $r$ is an open embedding, and $$\sfp=\wt\sfp\circ r, \quad\quad\quad \sfq=\wt\sfq\circ r.$$
Moreover, the following basic fact was established (see \cite[Theorem 5.1.5]{BG}):

\medskip

\noindent{\it The object $r_!(k_{\Bun^\mu_P})\in \Dmod(\wt\Bun^\mu_P)$ is universally locally acyclic (ULA)
with respect to the map $\wt\sfq$.} Here for a stack $\CY$, we denote by $k_{\CY}\in \Dmod(\CY)$ the ``constant sheaf" D-module on 
$\CY$ (i.e., the Verdier dual of $\omega_\CY$). 

\sssec{}

Let us recall the definition of the ULA property. First, recall that on $\Dmod(\CY)$ there
 are two tensor products, namely
$$\CF_1\overset{!}\otimes \CF_2:=\Delta_\CY^!(\CF_1\boxtimes \CF_2) \quad \mbox{ and }\quad
\CF_1\overset{*}\otimes \CF_2:=\Delta_\CY^*(\CF_1\boxtimes \CF_2), \quad\quad \CF_1,\CF_2\in \Dmod(\CY),$$
where $\Delta_\CY :\CY\to \CY\times \CY$ is the diagonal morphism. (Note that $\Delta^*$ and $\overset{*}\otimes$
are only partially defined.) 

\medskip

Suppose now that we have a morphism $\phi:\CY\to \CZ$ with $\CZ$ smooth. 
According to \cite[Sect. 5.1.1]{BG}, $\CF\in \Dmod(\CY)$ is said to be \emph{ULA} 
with respect to $\phi$ if for every $\CF'\in \Dmod(\CZ)$ the following holds: $\CF\overset{*}\otimes \phi^*(\CF')$ is a well-defined object of $\Dmod(\CY)$ 
and a certain canonical morphism $$\CF\overset{*}\otimes \phi^*(\CF')\to 
\CF\overset{!}\otimes \phi^!(\CF')[2\dim(\CZ)],$$ defined in \cite[Sect. 5.1.1]{BG}, is an isomorphism.

\sssec{Proof of Proposition~\ref{existence of Eis !}}
Since $\sfp=\wt\sfp\circ r$ and the functor $\wt\sfp_!=\wt\sfp_*$ is well-defined, it suffices to show that $r_!\circ\sfq^*(\CF)$ is 
well-defined for every $\CF\in \Dmod(\Bun^\mu_M)$. 

\medskip

This follows from the ULA property of $r_!(k_{\Bun^\mu_P})$
because $r_!\circ \sfq^*(\CF)=r_!(k_{\Bun^\mu_P})\overset{*}\otimes\wt\sfq^*(\CF)$.

\qed

\ssec{The functor $\on{CT}^\mu_!$}  \label{CT_!}

Our next task is to analyze the existence of the left adjoint of the functor $\Eis_*\,$. 

\sssec{}

Let $P^-$ be a parabolic opposite to $P$. We will identify the Levi factors of $P$ and $P^-$ via 
the embedding of $M\simeq P\cap P^-$ into $P$ and $P^-$, respectively. 

\medskip

In particular, we have the diagram

\begin{equation} \label{e:- diag}
\xy
(15,0)*+{\Bun_G}="X";
(-15,0)*+{\Bun^\mu_M.}="Y";
(0,15)*+{\Bun^\mu_{P^-}}="Z";
{\ar@{->}^{\sfp^-} "Z";"X"};
{\ar@{->}_{\sfq^-} "Z";"Y"};
\endxy
\end{equation}

\medskip

Let $\Eis^{\mu,-}_*$, $\Eis^{\mu,-}_!$ and  $\on{CT}^{\mu,-}_*$ be the counterparts of $\Eis^\mu_*$, $\Eis^\mu_!$ 
and  $\on{CT}^\mu_*$, obtained via the diagram \eqref{e:- diag}. 

\sssec{}

We are now ready to formulate our main result:

\begin{thm}  \label{t:main}
The functor 
$$\on{CT}_!^\mu:\Dmod(\Bun_G)\to \Dmod(\Bun^\mu_M),$$ left adjoint to $\Eis_*^\mu$, 
exists and is canonically isomorphic to $\on{CT}^{\mu,-}_*$.
\end{thm}

Thus we have a sequence of three functors
\[
\quad \Eis_!^{\mu,-}\, , \quad\on{CT}^{\mu,-}_*=  \on{CT}_!^\mu\,,\quad \Eis_*^\mu 
\] 
in which each neighboring pair forms an adjoint pair of functors.

\sssec{}

\thmref{t:main} can be tautologically restated as follows:

\begin{thm}  \label{t:main adj}
The functors $(\on{CT}^{\mu,-}_*,\Eis_*^\mu)$
form an adjoint pair.
\end{thm}

In \secref{s:proof of adj} we will prove \thmref{t:main adj} by essentially repeating the argument
from the paper \cite{DrGa3} that gives a new proof of a theorem of T.~Braden (see \cite{Br}) on 
hyperbolic restrictions.

\medskip

In \secref{s:mainproof} we will give a proof of \thmref{t:main adj} by directly deducing it from the
above theorem of Braden.

\ssec{Description of the co-unit of the adjunction}

\sssec{}   \label{sss:co-unit}

Let us specify the co-unit of the adjunction for the functors $(\on{CT}^{\mu,-}_*,\Eis_*^\mu)$
in \thmref{t:main adj}, i.e., the natural transformation 
\begin{equation} \label{e:expl co-unit}
\on{CT}^{\mu,-}_*\circ \Eis_*^\mu\to \on{Id}_{\Dmod(\Bun^\mu_M)}.
\end{equation}

\medskip

Consider the diagram
$$
\xy
(20,0)*+{\Bun_G}="X";
(-20,0)*+{\Bun^\mu_M,}="Y";
(0,20)*+{\Bun^\mu_P}="Z";
(40,20)*+{\Bun^\mu_{P^-}}="W";
(60,0)*+{\Bun^\mu_M}="U";
(20,40)*+{\Bun^\mu_{P}\underset{\Bun_G}\times \Bun^\mu_{P^-}}="V";
(20,80)*+{\Bun_M^\mu}="T";
{\ar@{->}_{\sfp} "Z";"X"};
{\ar@{->}^{\sfq} "Z";"Y"};
{\ar@{->}^{\sfp^-} "W";"X"};
{\ar@{->}_{\sfq^-} "W";"U"}; 
{\ar@{->}^{'\sfp^-} "V";"Z"};
{\ar@{->}_{'\sfp} "V";"W"};
{\ar@{->}_{\sfj} "T";"V"};
{\ar@{->}_{\on{id}} "T";"Y"};
{\ar@{->}^{\on{id}} "T";"U"};
\endxy
$$
where $\sfj$ is the (open) embedding of the big Bruhat cell, i.e., the locus where the reductions of a given $G$-bundle
to $P$ and $P^-$ are mutually transversal.\footnote{The morphism $\sfj$ is of the type  considered in \lemref{l:2fiber product}(ii) below.} 

\medskip

By base change,
$$\on{CT}^{\mu,-}_*\circ \Eis_*^\mu=(\sfq^-)_*\circ (\sfp^-)^!\circ \sfp_*\circ \sfq^!
\simeq (\sfq^-)_*\circ ({}'\sfp)_*
\circ ({}'\sfp^-)^!\circ \sfq^!.$$

Now, the open embedding $\sfj$ gives rise to a natural transformation 
$$\on{Id}_{\Dmod(\Bun^\mu_P\underset{\Bun_G}\times \Bun^\mu_{P^-})}\to \sfj_*\circ \sfj^*\simeq \sfj_*\circ \sfj^!,$$ and hence to 
\begin{multline*}
\on{CT}^{\mu,-}_*\circ \Eis_*^\mu\simeq (\sfq^-)_*\circ ({}'\sfp)_*
\circ ({}'\sfp^-)^!\circ \sfq^!\to \\
\to (\sfq^-)_*\circ ({}'\sfp)_*\circ \sfj_*\circ \sfj^!
\circ ({}'\sfp^-)^!\circ \sfq^!\simeq \\
\simeq (\sfq^-\circ {}'\sfp\circ \sfj)_*\circ
(\sfq\circ {}'\sfp^-\circ \sfj)^!\simeq \on{id}_*\circ \on{id}^!=\on{Id}_{\Dmod(\Bun^\mu_M)}.
\end{multline*}

This is the sought-for co-unit of the adjuction. 

\sssec{}

Let us observe that \thmref{t:main} can be reformulated also as saying that the functors 
$(\Eis_!^{\mu,-},\on{CT}_!^{\mu})$ form an adjoint pair. 
The above description of the co-unit for the $(\on{CT}^{\mu,-}_*,\Eis_*^\mu)$-adjunction gives rise
to a description of the unit of the $(\Eis_!^{\mu,-},\on{CT}_!^{\mu})$-adjunction:

\medskip

By definition, the functors $\Eis_!^{\mu,-}$ and $\on{CT}_!^{\mu}$ are the left adjoints of the functors
$\on{CT}_*^{\mu,-}$ and $\Eis_*^{\mu}$, respectively. Now, the unit map
$$\on{Id}_{\Dmod(\Bun_M^\mu)}\to \on{CT}_!^\mu\circ \Eis_!^{\mu,-}$$
is obtained by passing to left adjoints in the map \eqref{e:expl co-unit}.

\ssec{Cuspidality}

\sssec{}

We will say that an object $\CF\in \Dmod(\Bun_G)$ is \emph{*-cuspidal} 
if $\on{CT}^\mu_*(\CF)=0$ for all proper parabolics of $G$.

\medskip

We will say that an object $\CF\in \Dmod(\Bun_G)$ is \emph{!-cuspidal}
$\on{CT}^\mu_!(\CF)=0$ for all proper parabolics of $G$.

\sssec{}

From \thmref{t:main} we obtain:

\begin{cor}
The notions of !- and *-cuspidality coincide.
\end{cor}

Hence, from now on, we will rename the above property as just cuspidality. Let $\Dmod(\Bun_G)_{\on{cusp}}$
be the full subcategory of $\Dmod(\Bun_G)$ formed by cuspidal objects. By construction, $\Dmod(\Bun_G)_{\on{cusp}}$
is cocomplete (i.e., closed under colimits). 

\medskip

Let us denote by $\Dmod(\Bun_G)_{\Eis,!}$ (resp., $\Dmod(\Bun_G)_{\Eis,*}$) 
the full subcategory of $\Dmod(\Bun_G)$ generated\footnote{``generated by $A$" means ``the smallest cocomplete DG subcategory, containing $A$"}
by the essential images of the functors $\Eis^\mu_!$ (resp., $\Eis^\mu_*$) for all proper parabolics. 

\medskip

Viewing the notion of cuspidality from the *-perspective, we have:
$$\Dmod(\Bun_G)_{\on{cusp}}=\left(\Dmod(\Bun_G)_{\Eis,!}\right)^\perp.$$

\sssec{}

In addition, we have the inclusion 
\begin{equation} \label{e:cusp incl}
^\perp\!\left(\Dmod(\Bun_G)_{\Eis,*}\right)\subset \Dmod(\Bun_G)_{\on{cusp}}
\end{equation}
that comes from 
\[
^\perp\!\left(\Dmod(\Bun_G)_{\Eis,*}\right)\subset
\underset{P,\mu}\cap\, {}^\perp\!\left(\on{Im}(\Eis_*^\mu)\right)
=\underset{P,\mu}\cap\, \on{ker}(\on{CT}^\mu_!)=\Dmod(\Bun_G)_{\on{cusp}}\,.
\]

It is not clear whether the inclusion \eqref{e:cusp incl} is an equality; we conjecture that it is not
if $X$ has genus $\geq 2$.

\sssec{}

Let us note another important property of cuspidal objects:

\begin{prop}    \label{p:cuspidality}
There exists an open substack $\jmath:\CU\hookrightarrow \Bun_G$ such that

\begin{enumerate}

\item[(i)] The intersection of $\CU$ with each connected component of  $\Bun_G$ is quasi-compact;

\item[(ii)] For any $\CF\in \Dmod(\Bun_G)_{\on{cusp}}\,$, the canonical maps
$$\jmath_!\circ \jmath^*(\CF)\to \CF\to \jmath_*\circ \jmath^*(\CF)$$
are isomorphisms.\footnote{The statement about  $\jmath_!\circ \jmath^*(\CF)$ should be understood as follows: 
the partially defined functor $\jmath_!$ is defined on $\jmath^*(\CF)$ and the map 
$\jmath_!\circ \jmath^*(\CF)\to \CF$ is an isomorphism.}
\end{enumerate}
\end{prop}

The proof is given in Appendix~\ref{s:support}. It is parallel to the proof of a similar statement in the classical 
theory of automorphic forms. In fact, in Appendix~\ref{s:support} we describe an explicit open substack $\CU$ 
with the properties required in \propref{p:cuspidality}.

\section{Proof of \thmref{t:main adj}: constructing the unit of the adjunction}  \label{s:proof of adj}

We will prove \thmref{t:main adj} (and thereby \thmref{t:main}) by mimicking the new proof of Braden's theorem given in \cite{DrGa3}. 

\medskip

The idea of the proof is to define a natural transformation
\begin{equation} \label{e:unit}
\on{Id}_{\Dmod(\Bun_G)}\to  \Eis_*^\mu\circ \on{CT}^{\mu,-}_*
\end{equation}
and to show that \eqref{e:expl co-unit} and \eqref{e:unit} define
an adjunction datum. 

\medskip

In this section we will construct the natural transformation \eqref{e:unit}. 
The fact that \eqref{e:expl co-unit} and \eqref{e:unit} indeed define an adjunction datum will be proved in 
\secref{s:Verifying}.

\ssec{Digression: functors given by kernels}

As was mentioned in \secref{ss:Conventions}, all algebraic stacks in this paper are assumed locally QCA.

\sssec{} \label{sss:blacktriangle}

Let $f:\CY_1\to \CY_2$ be a \emph{quasi-compact} map. We note that in addition to the usual de Rham direct image functor
$$f_*:\Dmod(\CY_1)\to \Dmod(\CY_2)$$
(defined as in \cite[Sect. 7.4]{DrGa1}),
there exists another canonically defined functor
$$f_\blacktriangle:\Dmod(\CY_1)\to \Dmod(\CY_2).$$

\medskip

The main feature of the functor $f_\blacktriangle$ is that, unlike $f_*$, it is \emph{continuous}. In addition,
it has the following properties (all of which fail for $f_*$):

\smallskip

\noindent(i) The formation $f\rightsquigarrow f_\blacktriangle$ is compatible with compositions of morphisms;

\smallskip

\noindent(ii) It satisfies base change (with respect to !-pullbacks);

\smallskip

\noindent(iii) It satisfies the projection formula. 

\medskip

We have a natural transformation
\begin{equation} \label{e:triangle to *}
f_\blacktriangle\to f_*\, .
\end{equation}

If the morphism $f$ is representable, or, more generally, \emph{safe}, then \eqref{e:triangle to *}
is an isomorphism. 

\medskip 

All of the above facts are established in \cite[Sect. 9.3]{DrGa1}.

\begin{rem}
Technically, in \cite[Sect. 9.3]{DrGa1}, the functor $f_\blacktriangle$ was defined when $\CY_2$
(and, hence, $\CY_1$) is quasi-compact. However, the base-change property ensures that
the definition canonically extends to the case of general quasi-compact morphisms.
\end{rem}

\sssec{}  \label{sss:*-adm}

Let $\CY_1$ and $\CY_2$ be two algebraic stacks. We will say that an object
$$\CQ\in \Dmod(\CY_1\times \CY_2)$$
has a \emph{quasi-compact !-support relative to $\CY_2$} if the following condition holds:

\medskip

\noindent For every quasi-compact open substack $U_2\overset{j_2}\hookrightarrow \CY_2$ there exists a 
quasi-compact open substack $U_1\overset{j_1}\hookrightarrow \CY_1$ such that the map
\begin{equation} \label{e:adm isom}
(\id\times j_2)^*(\CQ)\to (j_1\times \id)_*\circ (j_1\times \id)^* \circ (\id\times j_2)^*(\CQ)=
(j_1\times \id)_*\circ (j_1\times j_2)^*(\CQ)
\end{equation}
is an isomorphism in $\Dmod(\CY_1\times U_2)$. 

\medskip

Let us repeat the same in words: the restriction of $\CQ$ to the open
substack $\CY_1\times U_2$ is a *-extension from $U_1\times U_2$
for some quasi-compact open substack $U_1\subset \CY_1$. 

\medskip

Note that $(j_1\times \id)_*\simeq (j_1\times \id)_\blacktriangle$ (indeed, the morphism $j_1\times \id$ is 
representable because it is an open embedding). 

\sssec{}

Let $\CQ\in \Dmod(\CY_1\times \CY_2)$ have a quasi-compact !-support relative to $\CY_2$. 
We claim that such $\CQ$ canonically gives rise to a continuous functor
%
\begin{equation} \label{e:formula for functor}
\sF_\CQ:\Dmod(\CY_1)\to \Dmod(\CY_2), \quad \sF_\CQ(\CF):=(\on{pr}_2)_\blacktriangle \left(\on{pr}_1^*(\CF)\sotimes \CQ\right).
\end{equation}

Since the morphism $\on{pr}_2$ is not assumed quasi-compact, we have to explain how to 
understand formula~\eqref{e:formula for functor}.


\medskip

By the definition of $\Dmod(\CY_2)$ (see \cite[Sect. 2.3]{DrGa2} for a detailed review), the datum of a functor
$\Dmod(\CY_1)\to \Dmod(\CY_2)$ amounts to a compatible \footnote{The compatibility is
with respect to *-restrictions for the inclusions $U'_2\subset U''_2$.} family of functors 
$$\Dmod(\CY_1)\to \Dmod(U_2)$$ for quasi-compact open substacks $U_2\overset{j_2}\hookrightarrow \CY_2\,$. 

\medskip

For a given $U_2\,$, let $U_1\overset{j_1}\hookrightarrow \CY_1$ be as in \secref{sss:*-adm}, and let
$$\on{pr}_2^{U_1,U_2}:U_1\times U_2\to U_2$$ denote the projection.
Then we set
\begin{equation} \label{e:formula for functor open}
j_2^*\left(\sF_\CQ(\CF)\right):=
(\on{pr}_2^{U_1,U_2})_\blacktriangle 
\circ (j_1\times j_2)^*\left(\on{pr}_1^*(\CF)\sotimes \CQ\right)
\end{equation}
(here the $\blacktriangle$-puhsforward is defined because the morphism $\on{pr}_2^{U_1,U_2}$ is quasi-compact). 

\medskip

It is easy to see that the isomorphism \eqref{e:adm isom} implies that the right-hand side 
in \eqref{e:formula for functor open} is independent of the choice of $U_1\,$. 

\medskip

Furthermore, it is easy to see that the base change property for the $\blacktriangle$-puhsforward for 
quasi-compact morphisms implies that the functors \eqref{e:formula for functor open} are indeed
compatible under the inclusions $U'_2\subset U''_2$. 

\sssec{} \label{sss:corr}

As an example, consider a diagram
\begin{equation} \label{e:corr}
\xy
(-15,0)*+{\CY_1}="X";
(15,0)*+{\CY_2,}="Y";
(0,15)*+{\CY}="Z";
{\ar@{->}_{f_1} "Z";"X"};
{\ar@{->}^{f_2} "Z";"Y"};
\endxy
\end{equation}
where the morphism $f_2$ is quasi-compact. In this case the morphism
$$(f_1\times f_2):\CY\to \CY_1\times \CY_2$$
is quasi-compact as well. 

\medskip

Set
$$\CQ:=(f_1\times f_2)_\blacktriangle(\omega_\CY)\in \Dmod(\CY_1\times \CY_2),$$
where $\omega_\CY\in \Dmod(\CY)$ is the dualizing complex.

\medskip

It is easy to see that $\CQ$ has a quasi-compact !-support relative to $\CY_2$. Moreover, in this case the functor
$\sF_\CQ$ identifies canonically with $(f_2)_\blacktriangle\circ f_1^*$. 

\medskip

In particular, for $\CY_1=\CY=\CY_2$ and $f_1=f_2=\id$, we obtain that for
$\CQ:=(\Delta_\CY)_\blacktriangle(\omega_\CY)$, the corresponding
functor $\sF_\CQ$ identifies canonically with $\on{Id}_{\Dmod(\CY)}$. 

\ssec{Plan of the construction}

\sssec{} 

Note that since the morphism $\sfp$ is quasi-compact and representable (in fact, schematic), we have:
$$\Eis_*^\mu=\sfp_*\circ \sfq^!\simeq \sfp_\blacktriangle\circ \sfq^!.$$

Since $\sfq$ is \emph{safe}, we have also
$$\on{CT}_*^\mu=\sfq_*\circ \sfp^!\simeq \sfq_\blacktriangle\circ \sfp^!.$$

\medskip

Hence, the functor $\Eis_*^\mu\circ \on{CT}^{\mu,-}_*$ is given as pull-push along the following diagram:

\begin{equation}  \label{e:the_diagram}
\xy
(20,0)*+{\Bun^\mu_M}="X";
(-20,0)*+{\Bun_G.}="Y";
(0,20)*+{\Bun^\mu_{P^-}}="Z";
(40,20)*+{\Bun^\mu_{P}}="W";
(60,0)*+{\Bun_G}="U";
(20,40)*+{\Bun^\mu_{P^-}\underset{\Bun^\mu_M}\times \Bun^\mu_{P}}="V";
{\ar@{->}^{\sfq^-} "Z";"X"};
{\ar@{->}_{\sfp^-} "Z";"Y"};
{\ar@{->}_{\sfq} "W";"X"};
{\ar@{->}^{\sfp} "W";"U"}; 
{\ar@{->}_{'\sfq} "V";"Z"};
{\ar@{->}^{'\sfq^-} "V";"W"};
\endxy
\end{equation}

Here and in the sequel, by ``push" we understand the $\blacktriangle$-pushforward, so the base change
formula applies to any quasi-compact morphism.

\sssec{}  \label{sss:intr K}

Denote
$$\CQ_0:=\left((\sfp^-\circ {}'\sfq)\times (\sfp\circ {}'\sfq^-)\right)_\blacktriangle
(\omega_{\Bun^\mu_{P^-}\underset{\Bun^\mu_M}\times \Bun^\mu_{P}})\in \Dmod(\Bun_G\times \Bun_G)$$
and
$$\CQ_1:=(\Delta_{\Bun_G})_\blacktriangle(\omega_{\Bun_G})\in \Dmod(\Bun_G\times \Bun_G).$$

\medskip

By \secref{sss:corr} we have:

\begin{lem}  \label{l:two kernels}
The functors $\on{Id}_{\Dmod(\Bun_G)}$ and  $\Eis_*^\mu\circ \on{CT}^{\mu,-}_*$ are given by
$\sF_{\CQ_1}$ and $\sF_{\CQ_0}$, respectively.
\end{lem} 

\sssec{}  \label{sss:properties of K}

We are going to define the natural transformation \eqref{e:unit} by constructing a map 
\begin{equation} \label{e:map of kernels}
\CQ_1\to \CQ_0
\end{equation}
in $\Dmod(\Bun_G\times \Bun_G)$. To do this, we will construct an object 
$$\CQ\in\Dmod(\BA^1\times\Bun_G\times \Bun_G)$$ such that:

\medskip

\noindent(i)  $\iota_0^!(\CQ)=\CQ_0$ and $\iota_1^!(\CQ)=\CQ_1\,$, 
where $\iota_t:\Bun_G\times \Bun_G\to \BA^1\times\Bun_G\times \Bun_G$ is the closed embedding corresponding to $t\in \BA^1$;

\medskip

\noindent(ii) $\CQ$ is $\BG_m$-monodromic with respect to the usual action of $\BG_m$ on 
$\BA^1$; by definition, this means that $\CQ$ belongs to the full subcategory of 
$\Dmod(\BA^1\times\Bun_G\times \Bun_G)$ generated by the essential image of the pullback functor 
\[
\Dmod((\BA^1/\BG_m)\times\Bun_G\times \Bun_G)\to \Dmod(\BA^1\times\Bun_G\times \Bun_G). 
\]

\medskip

Assuming we have such $\CQ$, we take the map $\CQ_1\to \CQ_0$ of \eqref{e:map of kernels} to be the \emph{specialization map} 
$$\on{Sp}_\CQ:\iota_1^!(\CQ)\to\iota_0^!(\CQ)\,,$$ 
which is defined by virtue of the above property (ii). The definition of the 
specialization map is recalled in \secref{sss:specialization} below.

\medskip

The remaining part of \secref{s:proof of adj} is devoted to constructing an object 
$\CQ$ with the above properties (i)-(ii). Informally, $\CQ$ should ``interpolate" between $\CQ_1$ and $\CQ_0\,$.

\begin{rem}
The construction of $\CQ$ is geometric, and except for \secref{sss:defining K} at the very end, we do not use the characteristic 0 
assumption on $k$. The idea is to first construct a group-scheme $\wt{G}$ over $\BA^1$, which ``interpolates" between the groups 
$G$ and $P\underset{M}\times P^-$, and then to construct a stack over $\BA^1$ which ``interpolates" between $\Bun_G$ and 
the stack 
$$\Bun^\mu_{P^-}\underset{\Bun^\mu_M}\times \Bun^\mu_{P}$$
from diagram \eqref{e:the_diagram}.
\end{rem} 

\sssec{The specialization map}   \label{sss:specialization}

Let $\CY$ be a stack. Let $\iota_0,\iota_1:\CY\to \BA^1\times\CY$ denote the closed embeddings 
corresponding to $0\in \BA^1$ and $1\in \BA^1$, respectively. Suppose that 
$\CK\in\Dmod(\BA^1\times \CY)$ is $\BG_m$-monodromic with respect to the usual action of $\BG_m$ on 
$\BA^1$. In this situation we can define the  \emph{specialization morphism} 
\begin{equation}  \label{e:specialization}
\on{Sp}_\CK:\iota_1^!(\CK)\to\iota_0^!(\CK)\,.
\end{equation}

\medskip

We need the following lemma (which is actually a particular case of \propref{p:simple Braden}(2)):

\begin{lem} \label{l:specialization}
Let $\pi :\BA^1\times \CY\to \CY$ denote the projection. Then the functor $\pi_!\,$, left adjoint to $\pi^!$, is defined and is canonically 
isomorphic to $\iota_0^!$ on $\BG_m$-monodromic objects of 
$\Dmod(\BA^1\times \CY)$.
\end{lem}

\begin{proof}
The morphism in one direction
$$\iota_0^!\to \pi_!$$ (assuming that $\pi_!$ exists) is given by
$$\iota_0^!\simeq \pi_!\circ (\iota_0)_!\circ \iota_0^!\to \pi_!.$$

To verify the existence of $\pi_!$ and the fact that the above map is an isomorphism, it suffices to do so for the 
corresponding functors 
$\Dmod((\BA^1/\BG_m)\times \CY)\to \Dmod((\on{pt}/\BG_m)\times \CY)$.
Since 
$$\Dmod((\BA^1/\BG_m)\times \CY)\simeq \Dmod(\BA^1/\BG_m)\otimes \Dmod(\CY)$$
and
$$\Dmod((\on{pt}/\BG_m)\times \CY)\simeq \Dmod(\on{pt}/\BG_m)\otimes \Dmod(\CY),$$
the assertion reduces to the case when $\CY=\on{pt}$, which is straightforward. 
\end{proof}

Now define the specialization morphism \eqref{e:specialization} to be the composition
$$\iota_1^!(\CK)\simeq \pi_!\circ (\iota_1)_!\circ \iota_1^!(\CK)\to \pi_!(\CK)\simeq \iota_0^!(\CK)$$
(note that since $\iota_1$ is a closed embedding the functor $(\iota_1)_!$ is well-defined).

\begin{rem}
One can show (using only general nonsense) that the above composition is equal to the composition
\[
\iota_1^!(\CK )\to\iota_1^!\circ\pi^!\circ\pi_! (\CK )\simeq\pi_! (\CK )\simeq\iota_0^!(\CK ).
\]
\end{rem}

\ssec{Recollections from \cite{DrGa3}}  \label{ss:interpolZ}

In this subsection we will recall some constructions from \cite[Sects. 1-2]{DrGa3}, to be subsequently applied to
the group $G$. 

\sssec{Attractors and repellers}  \label{sss:Z+}

Let $Z$ be a quasi-compact scheme acted on by $\BG_m$.
According to \cite[Proposition 1.3.4 and Corollary 1.5.3(ii)]{DrGa3}, there exist quasi-compact schemes 
$Z^0$, $Z^+$, and $Z^-$ representing the following functors:
$$\Maps(S,Z^0)=\Maps^{\BG_m}(S,Z),$$
$$\Maps(S,Z^+)=\Maps^{\BG_m}(\BA^1\times S,Z),$$
$$\Maps(S,Z^-)=\Maps^{\BG_m}(\BA^1_-\times S,Z),$$
where $S$ is a test affine scheme and $\BA^1$ (resp.~$\BA^1_-$) is the affine line equipped with the usual 
$\BG_m$-action (resp. the $\BG_m$-action opposite to the usual one).

\medskip

The scheme $Z^0$ (resp.~$Z^+$ and $Z^-$) is called the \emph{scheme of $\BG_m$-fixed points}
(resp.~the \emph{attractor} and \emph{repeller}).

\medskip

Let $p^+:Z^+\to Z$ and $q^+:Z^+\to Z^0$ denote the maps corresponding to evaluating a $\BG_m$-equivariant morphism 
$\BA^1\times S\to Z$ at $1\in \BA^1$  and $0\in \BA^1$, respectively. One defines
$p^-:Z^-\to Z$ and $q^-:Z^-\to Z^0$ similarly. 

\medskip

Let $i^+:Z^0\to Z^+$ (resp. $i^-:Z^0\to Z^-$) denote the morphism induced 
by the projection $\BA^1\times S\to S$ (resp.~$\BA^1_-\times S\to S$).

\medskip
 
If $Z$ is affine (the case of interest for us) then the existence of $Z^0$ and $Z^\pm$ (i.e., the representability of the corresponding functors) 
is very easy to prove. Moreover, in this case 
$Z^0$ and $Z^\pm$ are affine, and the morphisms $p^{\pm}:Z^{\pm}\to Z$ are closed embeddings.

\sssec{The family of hyperbolas}    \label{sss:hyperbolas}

We now consider the following family of curves over $\BA^1$, denoted by $\BX$: as a scheme, 
$\BX=\BA^2$, and the map $\BX\to \BA^1$ is $(t_1,t_2)\mapsto t_1\cdot t_2\,$. 
The fibers of this map are hyperbolas; the zero fiber is the coordinate cross, i.e., a degenerate hyperbola.

\medskip

We let $\BG_m$ act on $\BX$ hyperbolically: 
$$\lambda\cdot (t_1,t_2)=(\lambda\cdot t_1\, ,\lambda^{-1}\cdot t_2).$$

\sssec{The scheme $\wt{Z}$}  \label{sss:tilde Z}

According to \cite[Theorem 2.4.2]{DrGa3}, there exists a quasi-compact scheme $\wt{Z}$ over $\BA^1$ representing 
the following functor on the category of schemes over $\BA^1$:
$$\Maps_{\BA^1}(S,\wt{Z}):=\Maps^{\BG_m}(\BX\underset{\BA^1}\times S,Z).$$

\medskip

Again, if $Z$ is affine the existence of $\wt{Z}$ is very easy to prove; moreover, in this case $\wt{Z}$ is 
affine. Let us also mention that if $Z$ is affine and smooth (the case of interest for us) then \propref{p:smooth}(ii) below 
gives a description of $\wt{Z}$ which some readers may prefer to consider as a definition.

\sssec{}

One has a canonical map $\wt{p}:\wt{Z}\to \BA^1\times Z\times Z$. To define it, first note that any 
section 
$\sigma :\BA^1\to\BX$ of the morphism $\BX\to\BA^1$ defines a map $$\sigma^*:\Maps^{\BG_m} 
(\BX\underset{\BA^1}\times S\, ,Z)\to\Maps (S,Z)$$
and therefore a morphism $\wt{Z}\to Z$. Let $\pi_1:\wt{Z}\to Z$ and $\pi_2:\wt{Z}\to Z$ 
denote the morphisms corresponding to the sections
\[
t\mapsto (1,t)\in \BX \quad \text{ and }\quad t\mapsto (t,1)\in \BX\, ,
\]
respectively. 

\medskip

Finally, define 
\begin{equation} \label{e:from tilde}
\wt{p}:\wt{Z}\to \BA^1\times Z\times Z\, 
\end{equation}
to be the morphism whose first component is the tautological projection $\wt{Z}\to \BA^1$, 
and the second and the third components are $\pi_1$ and $\pi_2$, respectively. 

\medskip

It is easy to check that if $Z$ is affine then $\wt{p}$ is a closed embedding (see \cite[Proposition 2.3.6]{DrGa3}). So 
$\wt{p}$ identifies $\wt{Z}$ with a closed subscheme of $\BA^1\times Z\times Z$.

\sssec{}

Let $\wt{Z}_t$ denote the preimage of $t\in \BA^1$ under the projection $\wt{Z}\to \BA^1$. 
Let $\wt{p}_t$ denote the corresponding map $\wt{Z}_t\to Z\times Z$.

\medskip

By definition, $(\wt{Z}_1,\wt{p}_1)$ identifies with $(Z,\Delta_Z)$. For any $t\in \BA^1-\{0\}$, the pair 
$(\wt{Z}_t,\wt{p}_t)$ is the graph of the action of $t\in \BG_m$ on $Z$. Moreover, the morphism $\wt{p}$ induces an isomorphism
\begin{equation}  \label{e:graph}
\BG_m\underset{\BA^1}\times\wt{Z}\iso\Gamma , \quad\quad \Gamma:=\{ (t,z_1, z_2)\,|\,t\cdot z_1=z_2 \}.
\end{equation}

The scheme $\wt{Z}_0$ identifies with $Z^+\underset{Z^0}\times Z^-$ so that the morphism
$\wt{p}_0:\wt{Z}_0\to Z\times Z$ identifies with the composition 
$$Z^+\underset{Z^0}\times Z^-\hookrightarrow Z^+\times Z^- \overset{p^+\times p^-}\longrightarrow Z\times Z.$$ 
The above-mentioned identification comes from the fact that the coordinate cross $\BX_0$ is a union of $\BA^1$ and
$\BA^1_-$ glued together along $0$, see \cite[Proposition 2.2.9]{DrGa3}.

\sssec{The action of $\BG_m\times \BG_m$\,} \label{sss:action on tilde}

In what follows we will need the action of $\BG_m\times \BG_m$ on $\wt{Z}$ that corresponds to the following 
 action of $\BG_m\times \BG_m$ on $\BX$:
$$(\lambda_1,\lambda_2)\cdot (t_1,t_2)=(\lambda_1\cdot t_1\,,\lambda_2\cdot t_2);
\quad \lambda_1,\lambda_2\in \BG_m\, , \quad (t_1,t_2)\in\BX\, .$$

Note that the map $\wt{p}$ is equivariant with respect to the following action of $\BG_m\times \BG_m$ on
$\BA^1\times Z\times Z$:
$$(\lambda_1,\lambda_2)\cdot (t,z_1,z_2)=(\lambda_1^{-1}\cdot \lambda_2^{-1}\cdot t,\lambda_1\cdot z_1\, ,\lambda^{-1}_2\cdot z_2).$$

\sssec{Smoothness}  \label{sss:smoothness}
The following facts are proved in \cite[Propositions 2.5.2 and 2.5.5]{DrGa3}.

\begin{prop}   \label{p:smooth}  \hfill

\smallskip

\noindent{\em(i)} If $Z$ is smooth then so is the morphism $\wt{Z}\to\BA^1$. 

\smallskip

\noindent{\em(ii)} If $Z$ is smooth and affine then the morphism $\wt{p}:\wt{Z}\to \BA^1\times Z\times Z$ induces an isomorphism 
\[
\wt{Z}\iso\overline{\Gamma},
\]
where $\Gamma$ is as in formula~\eqref{e:graph} and $\overline{\Gamma}$ is the scheme-theoretic closure of $\Gamma$ in
$ \BA^1\times Z\times Z\,$.
\end{prop}

We will need the proposition only in the case that $Z$ is affine. Let us give a proof in this case, 
which is different from the one in \cite{DrGa3}.

\begin{proof}
We can assume that $Z$ has pure dimension $n$. This easily implies that 
$\wt{Z}_0=Z^+\underset{Z^0}\times Z^-$ is smooth and of pure dimension $n$.

\medskip

Since $Z$ is affine, $\wt{p}$ is a closed embedding; so we can consider $\wt{Z}$ as a closed subscheme of
$\BA^1\times Z\times Z$. We have $\wt{Z}\cap (\BG_m\times Z\times Z)=\Gamma$, so 
$\overline{\Gamma}\subset \wt{Z}$.

\medskip

Let us show that
\begin{equation}  \label{e:equality of fibers}
(\overline{\Gamma})_t =\wt{Z}_t \mbox{ for all } t\in\BA^1
\end{equation}
(here $(\overline{\Gamma})_t$ is the fiber over $t$). The only nontrivial case is $t=0$.
Both $(\overline{\Gamma})_0$ and $\wt{Z}_0$ have pure dimension $n$, and $\wt{Z}_0$ is smooth. Since 
$(\overline{\Gamma})_0\subset \wt{Z}_0\,$, it remains to check that $(\overline{\Gamma})_0$ meets each 
connected component of $\wt{Z}_0$. This follows from the obvious inclusion $\wt{Z}\supset\BA^1\times\Delta_Z(Z^0)$.

\medskip

By definition, $\overline{\Gamma}$ is flat over $\BA^1$. By \eqref{e:equality of fibers}, $(\overline{\Gamma})_t$ 
is smooth for each $t\in\BA^1$. So $\overline{\Gamma}$ is smooth over $\BA^1$. 
It remains to show that the emebdding $\overline{\Gamma}\mono \wt{Z}$ is an isomorphism. 
By \eqref{e:equality of fibers}, it suffices to check that it is an open embedding. This follows from the next lemma.
\end{proof}

\begin{lem}   \label{l:openness by fibers}
Let $Y$ and $Y'$ be schemes of finite type over a Noetherian scheme $S$, and let $f:Y'\to Y$ be an $S$-morphism.
Assume that: 

\smallskip

\noindent{\em(i)} $Y'$ is flat over $S$,

\smallskip

\noindent{\em(ii)} for any $s\in S$ the morphism $Y'_s\to Y_s$ induced by $f$ is an open embedding.

\smallskip

Then $f$ is an open embedding.
\end{lem}

This follows from Grothendieck's ``Crit\`ere de platitude par fibres" (Corollary 11.3.11 from EGA IV-3).
Here is a direct proof (which is much shorter than the proof of ``Crit\`ere de platitude par fibres"):

\begin{proof}
It suffices to show that $f$ is \'etale at any $y'\in Y'$. Let $y$ and $s$ be the images of $y'$ in $Y$ and $S$.  Let
$A$ be the completed local ring of $S$ at $s$. Let $B$ be
the completed local ring of $Y$ at $y$. Let $B'$ be the completed local ring of
$Y'$ at $y'$. The problem is to show that the homomorphism $\varphi :B\to B'$ induced by $f$ is an isomorphism. 

By assumption, $\varphi$ induces an isomorphism $B/m B\iso B'/mB'$, where $m$ is the maximal ideal of $A$. Since $B'$ is $m$-adically complete this implies that $\varphi$ is surjective. It remains to show that the ideal $I:=\Ker (B\to B')$ equals 0. Since $B'$ is flat over $A$ we have an exact sequence 
$$0\to I/mI\to B/mB\to B'/mB'\to 0.$$
So $I/mI=0$. By Nakayama's lemma, this implies that $I=0$.

\end{proof}

\ssec{The interpolating family of groups}   \label{ss:Interpolating}

\sssec{}    \label{sss:group set-up}

Fix a co-character
\begin{equation} \label{e:co-character}
\gamma:\BG_m\to M
\end{equation}
mapping to the center of $M$, which is dominant and regular with respect to $P$.

\medskip

We will apply the set-up of \secref{ss:interpolZ} when $Z=G$ and the $\BG_m$-action on $G$ is the adjoint action corresponding 
to the co-character \eqref{e:co-character}.

\sssec{}  \label{sss:interp group}

In this situation $Z^+=P$, $Z^-=P^-$, and $Z^0=M$. 

\medskip

Recall that the scheme
$$\wt{G}:=\wt{Z}$$
is a closed subscheme of $\BA^1\times G\times G$. 

\medskip

Since the constructions of \secref{ss:interpolZ} are functorial in $Z$, 
this subscheme 
is a group-scheme over $\BA^1$, which is a closed group-subscheme of the constant group-scheme 
$\BA^1\times G\times G$. By \propref{p:smooth}, $\wt{G}$ is smooth over $\BA^1$.

\medskip

The fiber $\wt{G}_1$ of $\wt{G}$ over $1\in \BA^1$ is the diagonal copy of $G$, and the fiber $\wt{G}_0$
over $0\in \BA^1$ is $P\underset{M}\times P^-$. 

\sssec{}
Since $\wt{G}$ is flat over $\BA^1$ the quotient $(\BA^1\times G\times G)/\wt{G}$ exists as an algebraic space of finite 
type over $\BA^1$. 
In Appendix \ref{s:quasiaffine} we will prove the following statement.

\begin{prop}   \label{p:quasiaffine}
The quotient $(\BA^1\times G\times G)/\wt{G}$ is a quasi-affine scheme.
\end{prop}

\sssec{}

In Appendix \ref{s:Vinberg} we will show how the group scheme $\wt{G}$ can be described using the Vinberg semi-group 
corresponding to $G$, see \propref{p:stabilizer}. 
This description immediately implies \propref{p:quasiaffine} (see  \corref{c:q-aff}). 
However, the proof of \propref{p:quasiaffine} given in Appendix \ref{s:quasiaffine} has the advantage of being short and 
self-contained.

\ssec{The interpolation of moduli of bundles}   \label{ss:interpolG}

\sssec{}

The functor that assigns to an affine scheme $S$ over $\BA^1$ the groupoid of torsors on $S\times X$ with respect to the 
group-scheme 
$(\wt{G}\underset{\BA^1}\times S)\times X$ is an (a priori, non-algebraic) stack over $\BA^1$, denoted by 
$\Bun_{\wt{G}}\,$. The morphism $$\wt{p}:\wt{G}\to \BA^1\times G\times G$$
gives rise to a morphism 
\footnote{We warn the reader of the clash of notations: the map $\wt\sfp:\Bun_{\wt{G}}\to \BA^1\times \Bun_G\times \Bun_G$
introduced above has nothing to do with the map $\wt\sfp:\wt{\Bun}_P\to \Bun_G$ of \secref{sss:BunPtilde}. The symbol
$\wt\sfp$ has been chosen in both cases in order to be consistent with both \cite{DrGa3} and \cite{BG}. 
The two are unlikely to be confused,
as the stack $\wt\Bun_P$ will not appear again in this paper.}
\begin{equation}  \label{e:tilde p}
\wt\sfp:\Bun_{\wt{G}}\to \BA^1\times \Bun_G\times \Bun_G.
\end{equation}

\begin{prop}   \label{p:Bun_tilde_G} \hfill

\smallskip

\noindent{\em(i)} $\Bun_{\wt{G}}$ is an algebraic stack smooth over $\BA^1$, with an affine diagonal.

\medskip

\noindent{\em(ii)} The map $\wt\sfp$ is of finite type and representable. Moreover, it is schematic.
\end{prop}

\begin{proof}
Point (ii) follows from \propref{p:quasiaffine}.

\medskip

Let us prove (i). Statement (ii) implies that the stack $\Bun_{\wt{G}}$ is algebraic and locally of finite type. 
The fact that it has an affine diagonal is immediate from the fact that $\wt{G}$ itself is affine over $\BA^1$. 
It remains to check that the morphism $\Bun_{\wt{G}}\to\BA^1$ is formally smooth. As usual, this follows from 
the fact that any coherent sheaf on $X$ has a trivial $H^2$.
\end{proof}

\sssec{}

Consider the action of $\BG_m\times \BG_m$ on $\wt{G}$ defined as in \secref{sss:action on tilde}. 
The morphism $\wt{G}\to \BA^1$ is $\BG_m\times \BG_m$-equivariant if $\BA^1$ is
equipped  with the following $\BG_m\times \BG_m$-action:
$$(\lambda_1,\lambda_2)\cdot t=\lambda_1^{-1}\cdot \lambda_2^{-1}\cdot t.$$
Moreover, the action of $\BG_m\times \BG_m$ on $\wt{G}$ respects the group 
structure on $\wt{G}$. Therefore it induces a $\BG_m\times \BG_m$-action on the stack $\Bun_{\wt{G}}\,$, 
which covers the above $\BG_m\times \BG_m$-action on $\BA^1$.


\begin{lem}  \label{l:geom equiv}
The map $\wt\sfp:\Bun_{\wt{G}}\to \BA^1\times \Bun_G\times \Bun_G$ is equivariant
with respect to the above $\BG_m\times \BG_m$-action on $\Bun_{\wt{G}}$ and the
$\BG_m\times \BG_m$-action on $\BA^1\times \Bun_G\times \Bun_G$ via the $\BA^1$-factor. 
\end{lem}

\begin{proof}

This follows from the $\BG_m\times \BG_m$-equivariance of the map
$$\wt{p}:\wt{G}\to \BA^1\times G\times G$$
(see \secref{sss:action on tilde}), and the fact that since the $\BG_m\times \BG_m$-action on $G\times G$ is inner, 
the induced action on $\Bun_G\times \Bun_G$ is canonically isomorphic to the trivial one.
\end{proof}

\begin{rem}
The action on $\Bun_{\wt{G}}$ of the subgroup 
\begin{equation}   \label{e:Antidiag}
\{ (\lambda_1,\lambda_2)\in \BG_m\times\BG_m\, |\, \lambda_1\cdot \lambda_2=1\}\subset \BG_m\times\BG_m
\end{equation}
is canonically trivial because its action on $\wt{G}$ is inner. Moreover, the action of the sub\-group~\eqref{e:Antidiag}
on the triple $(\Bun_{\wt{G}}\, , \BA^1\times \Bun_G\times \Bun_G\, , \wt{p})$ is canonically trivial. \end{rem}

\sssec{}

Let $(\Bun_{\wt{G}})_t$ denote the fiber if $\Bun_{\wt{G}}$ over $t\in \BA^1$. Let
$$\wt\sfp_t:(\Bun_{\wt{G}})_t\to \Bun_G\times \Bun_G$$
denote the corresponding map. 

\medskip

By construction, $(\Bun_{\wt{G}})_1$ identifies with $\Bun_{\wt{G}_1}=\Bun_G\,$, and the morphism 
$$\wt\sfp_1:(\Bun_{\wt{G}})_1:\to\Bun_G\times\Bun_G,$$ 
identifies with the diagonal $$\Delta_{\Bun_G}:\Bun_G\to \Bun_G\times \Bun_G.$$

\medskip

Similarly, $(\Bun_{\wt{G}})_0$ identifies with
$\Bun_{\wt{G}_0}\simeq \Bun_{P\underset{M}\times P^-}$. We now claim:

\begin{lem}  \label{l:fiber product}
The natural map $\Bun_{P\underset{M}\times P^-}\to \Bun_P\underset{\Bun_M}\times \Bun_{P^-}$
is an isomorphism.
\end{lem}

The above lemma is a particular case of the following one.

\begin{lem}  \label{l:2fiber product}
Let $$H_1\overset{f_1}\longrightarrow H\overset{f_2}\longleftarrow H_2$$ be a diagram of algebraic groups.

\smallskip

\noindent{\em(i)} If $f_1(H_1)\cdot f_2(H_2)=H$ then the natural map 
$\Bun_{H_1\underset{H}\times H_2}\to \Bun_{H_1}\underset{\Bun_H}\times \Bun_{H_2}$
is an isomorphism.

\smallskip

\noindent{\em(ii)} If $f_1(H_1)\cdot f_2(H_2)$ is open in $H$ then the above map is an open embedding.
\end{lem}

\begin{proof}
It suffices to show that the morphism of stacks
\begin{equation} \label{e:map of stacks}
\on{pt}/(H_1\underset{H}\times H_2)\to  \on{pt}/H_1 \underset{ \on{pt}/H}\times  \on{pt}/H_2
\end{equation}
is an isomorphism if (i) holds and an open embedding if (ii) holds. To this end, consider the action of
$H_1\times H_2$ on $H$ defined by
\[
(h_1,h_2)*h:=f_1(h_1)\cdot h\cdot f_2(h_2)^{-1}, \quad\quad h_i\in H_i\, ,\, h\in H\, 
\]
and note that the morphism \eqref{e:map of stacks} can be obtained from the embedding 
$f_1(H_1)\cdot f_2(H_2)\hookrightarrow H$ by passing to the quotient with respect to the action of $H_1\times H_2$.
\end{proof}

\ssec{The interpolating kernel}
 
\sssec{}  \label{sss:tilde mu}

Let $\Bun_{\wt{G}}^\mu$ denote the open substack of $\Bun_{\wt{G}}$ obtained by removing
from the special fiber 
$$\Bun_{\wt{G}_0}\simeq \Bun_P\underset{\Bun_M}\times \Bun_{P^-}$$
the union of the connected components $\Bun^{\mu'}_P\underset{\Bun^{\mu'}_M}\times \Bun^{\mu'}_{P^-}$
with $\mu'\neq \mu$.  By a slight abuse of notation, the restriction of the map \eqref{e:tilde p} to 
$\Bun_{\wt{G}}^\mu$ will be denoted by the same symbol
\begin{equation} \label{e:wtsfp}
\wt\sfp:\Bun_{\wt{G}}^\mu\to \BA^1\times \Bun_G\times \Bun_G\, .
\end{equation}

By \lemref{l:fiber product},
\begin{equation}  \label{e:tilde mu 0}
(\Bun_{\wt{G}}^\mu)_0\simeq \Bun^\mu_P\underset{\Bun^\mu_M}\times \Bun^\mu_{P^-}\, .
\end{equation}

On the other hand, the map
\begin{equation}  \label{e:tilde mu 1}
(\Bun_{\wt{G}}^\mu)_1\to (\Bun_{\wt{G}})_1\simeq \Bun_G
\end{equation}
is an isomorphism.

\sssec{}  \label{sss:defining K}
By \propref{p:Bun_tilde_G}(ii),
the morphism $\wt\sfp$ in \eqref{e:wtsfp} is quasi-compact, so the functor
$\wt\sfp_\blacktriangle$ is well-defined. 

\begin{rem}
By \propref{p:Bun_tilde_G}(ii), $\wt\sfp$ is representable, and in fact schematic. 
As was already mentioned in \secref{sss:blacktriangle}, this implies that
the morphism $\wt\sfp_\blacktriangle\to \wt\sfp_*$ is, in fact, an isomorphism.
\end{rem}

\medskip

We define $\CQ\in \Dmod(\BA^1\times \Bun_G\times \Bun_G)$ by
$$\CQ:=\wt\sfp_\blacktriangle(\omega_{\Bun_{\wt{G}}^\mu}).$$

\medskip

Let us check that $\CQ$ satisfies properties (i)-(ii) from \secref{sss:properties of K}.

\sssec{}

By \eqref{e:tilde mu 0} and \eqref{e:tilde mu 1} and base change, 
the objects $\CQ_0$ and $\CQ_1$ from \secref{sss:intr K} identify
with the !-restrictions of $\CQ$ under the maps
$$\iota_0,\iota_1:\Bun_G\times \Bun_G\to \BA^1\times \Bun_G\times \Bun_G$$
corresponding to $0\in \BA^1$ and $1\in \BA^1$, respectively. 

\medskip

This establishes property (i).

\sssec{}    \label{sss:Q mon}

By \lemref{l:geom equiv} and base change, $\CQ$ is naturally $\BG_m\times \BG_m$-equivariant, i.e., 
it is naturally a pullback of an object of 
$$\Dmod((\BA^1/(\BG_m\times \BG_m))\times \Bun_G\times \Bun_G),$$
where the $\BG_m\times \BG_m$-action on $\BA^1$ is
$$(\lambda_1,\lambda_2)\cdot t=\lambda_1^{-1}\cdot \lambda_2^{-1}\cdot t.$$

\medskip

In particular, $\CQ$ is $\BG_m$-monodromic with respect to the action of $\BG_m$ on $\BA^1$
by dilations. 

\medskip

This establishes property (ii).

\section{Verification of adjunction}  \label{s:Verifying}

We have the natural transformation \eqref{e:expl co-unit}
$$\on{CT}^{\mu,-}_*\circ \Eis_*^\mu\to \on{Id}_{\Dmod(\Bun^\mu_M)}$$
and the natural transfrmation \eqref{e:unit} 
$$\on{Id}_{\Dmod(\Bun_G)}\to  \Eis_*^\mu\circ \on{CT}^{\mu,-}_*,$$
where the latter is defined using the map $\CQ_1\to \CQ_0$ of \eqref{e:map of kernels}. 

\medskip

It remains to show that these two natural transformations satisfy the adjunction properties. That is, we have to show that the composition
\begin{equation} \label{e:first composition}
\on{CT}^{\mu,-}_*\to \on{CT}^{\mu,-}_*\circ \Eis_*^\mu\circ \on{CT}^{\mu,-}_*\to
\on{CT}^{\mu,-}_*
\end{equation}
is isomorphic to the identity endomorphism of $\on{CT}^{\mu,-}_*$, and
\begin{equation} \label{e:second composition}
\Eis_*^\mu\to \Eis_*^\mu\circ \on{CT}^{\mu,-}_*\circ \Eis_*^\mu\to 
\Eis_*^\mu
\end{equation} 
is isomorphic\footnote{In the future we will 
skip the words ``isomorphic to"  in similar situations. (This is a slight abuse of language since we work with the DG categories of 
D-modules rather than with their homotopy categories.)}
 to the identity endomorphism of $\Eis_*^\mu$.

\medskip

We will do so for the composition \eqref{e:first composition}. The case of \eqref{e:second composition}
is similar and will be left to the reader. 

\medskip

The computation of the composition \eqref{e:first composition} repeats \emph{verbatim} the 
corresponding computation in \cite[Sect. 5]{DrGa3}. We include it for the sake of completeness.

\ssec{The diagram describing the composed functor}

\sssec{}

We will use the notation
\begin{equation}   \label{e:BIG}
\Phi:=\on{CT}^{\mu,-}_*\circ \Eis_*^\mu\circ \on{CT}^{\mu,-}_*=
\left((\sfq^-)_\blacktriangle\circ (\sfp^-)^!\right)\circ \left((\sfp^+)_\blacktriangle \circ (\sfq^+)^!\right)\circ
\left((\sfq^-)_\blacktriangle\circ (\sfp^-)^!\right).
\end{equation}

\medskip

By base change, $\Phi$ is given by pull-push along the following diagram: 

\begin{equation} \label{e:comp diag 1}
\xy
(-20,0)*+{\Bun_M^\mu}="X";
(20,0)*+{\Bun_G\, .}="Y";
(0,20)*+{\Bun^\mu_{P^-}}="Z";
(-40,20)*+{\Bun_P^\mu}="W";
(-60,0)*+{\Bun_G}="U";
(-20,40)*+{\Bun^\mu_P\underset{\Bun_M^\mu}\times \Bun^\mu_{P^-}}="V";
(-100,0)*+{\Bun_M^\mu}="T";
(-80,20)*+{\Bun^\mu_{P-}}="S";
(-60,40)*+{\Bun^\mu_{P-}\underset{\Bun_G}\times \Bun^\mu_P}="R";
(-40,60)*+{\Bun^\mu_{P^-}\underset{\Bun_G}\times \Bun_P^\mu\underset{\Bun_M^\mu}\times \Bun^\mu_{P^-}}="Q";
{\ar@{->}_{\sfq^-} "Z";"X"};
{\ar@{->}^{\sfp^-} "Z";"Y"};
{\ar@{->}^{\sfq^+} "W";"X"};
{\ar@{->}_{\sfp^+} "W";"U"}; 
{\ar@{->} "V";"Z"};
{\ar@{->} "V";"W"};
{\ar@{->}_{\sfq^-} "S";"T"};
{\ar@{->}^{\sfp^-} "S";"U"};
{\ar@{->} "R";"S"};
{\ar@{->} "R";"W"};
{\ar@{->} "Q";"R"};
{\ar@{->} "Q";"V"};
\endxy
\end{equation}

\sssec{}

Recall the stack $\Bun_{\wt{G}}^\mu$, see \secref{sss:tilde mu}. 

\medskip

Consider the stack
\begin{equation}   \label{e:tilde BunP-}
\Bun_{P^-,\wt{G}}^\mu:=\Bun^\mu_{P^-}\underset{\Bun_G}\times \Bun^\mu_{\wt{G}},
\end{equation}
where the fiber product is formed using the composition
\[
\Bun_{\wt{G}}^\mu\overset{\wt\sfp}\longrightarrow 
\BA^1\times \Bun_G\times \Bun_G\to \Bun_G\times \Bun_G \overset{\on{pr}_1}\longrightarrow \Bun_G.
\]

\medskip

Let 
$$\sfr: \Bun_{P^-,\wt{G}}^\mu\to \BA^1\times \Bun_M^\mu\times \Bun_G$$
denote the composition
$$\Bun_{P^-,\wt{G}}^\mu=\Bun^\mu_{P^-}\underset{\Bun_G}\times \Bun^\mu_{\wt{G}}\to  
\BA^1\times \Bun^\mu_{P^-}\times \Bun_G \overset{\id_{\BA^1}\times \sfq^-\times \id_{\Bun_G}}\longrightarrow 
\BA^1\times \Bun_M^\mu\times \Bun_G,$$
where the first arrow is obtained by base change from 
$$\wt\sfp:\Bun_{\wt{G}}^\mu\to \BA^1\times \Bun_G\times \Bun_G.$$

\medskip

Let $(\Bun_{P^-,\wt{G}}^\mu)_t$ denote the fiber of $\Bun_{P^-,\wt{G}}^\mu$ over $t\in \BA^1$. Let $\sfr_t$ denote the corresponding map
$(\Bun_{P^-,\wt{G}}^\mu)_t\to \Bun_M^\mu\times \Bun_G$.

\sssec{}   \label{sss:small diagram}

The isomorphism \eqref{e:tilde mu 0} defines an isomorphism
$$(\Bun_{P^-,\wt{G}}^\mu)_0\simeq \Bun^\mu_{P^-}\underset{\Bun_G}\times \Bun_P^\mu\underset{\Bun_M^\mu}\times \Bun^\mu_{P^-}$$
such that the compositions
$$\Bun^\mu_{P^-}\underset{\Bun_G}\times \Bun_P^\mu\underset{\Bun_M^\mu}\times \Bun^\mu_{P^-}\to
\Bun^\mu_{P^-}\underset{\Bun_G}\times \Bun_P^\mu\to \Bun^\mu_{P^-}\to \Bun^\mu_M$$
and
$$\Bun^\mu_{P^-}\underset{\Bun_G}\times \Bun_P^\mu\underset{\Bun_M^\mu}\times \Bun^\mu_{P^-}\to
\Bun_P^\mu\underset{\Bun_M^\mu}\times \Bun^\mu_{P^-}\to \Bun^\mu_{P^-}\to \Bun_G$$
from diagram \eqref{e:comp diag 1} are equal, respectively, to the compositions
\[
(\Bun_{P^-,\wt{G}}^\mu)_0\overset{\sfr}\longrightarrow \Bun_M^\mu\times  \Bun_G \overset{\on{pr}_1}\longrightarrow \Bun_M^\mu
\]
and
\[
(\Bun_{P^-,\wt{G}}^\mu)_0\overset{\sfr}\longrightarrow \Bun_M^\mu\times \Bun_G 
\overset{\on{pr}_2}\longrightarrow \Bun_G.
\]

\medskip

Hence, the functor $\Phi$ is given by pull-push along the diagram

\begin{equation}  \label{e:comp diag B}
\xy
(-20,0)*+{\Bun_M^\mu}="X";
(20,0)*+{\Bun_G.}="Y";
(0,20)*+{(\Bun_{P^-,\wt{G}}^\mu)_0}="Z";
{\ar@{->}_{\on{pr}_1\circ \sfr}  "Z";"X"};
{\ar@{->}^{\on{pr}_2\circ \sfr}  "Z";"Y"};
\endxy
\end{equation}

\ssec{The natural transformations at the level of kernels}  \label{ss:nat trans via kernels}

\sssec{}

Denote
$$\CS:=\sfr_\blacktriangle (\omega_{\Bun_{P^-,\wt{G}}^\mu})\in \Dmod(\BA^1\times \Bun_M^\mu\times \Bun_G).$$

\medskip

As in \secref{sss:Q mon} one shows that
$$\CS\in \Dmod(\BA^1\times \Bun_M^\mu\times \Bun_G)^{\BG_m\on{-mon}}.$$

\medskip

Set also
$$\CS_0:=(\sfr_0)_\blacktriangle (\omega_{(\Bun_{P^-,\wt{G}}^\mu)_0})\in \Dmod(\Bun_M^\mu\times \Bun_G)$$
and
$$\CS_1:=(\sfr_1)_\blacktriangle (\omega_{(\Bun_{P^-,\wt{G}}^\mu)_1})\in \Dmod(\Bun_M^\mu\times \Bun_G).$$

\sssec{}

By Sects. \ref{sss:corr} and \ref{sss:small diagram}, the functor $\Phi$ identifies with $\sF_{\CS_0}$. Let
us now describe the natural transformations 
$$\Phi\to (\sfp^-)_\blacktriangle\circ (\sfq^-)^! \text{ and } (\sfp^-)_\blacktriangle\circ (\sfq^-)^! \to \Phi$$
at the level of kernels.

\medskip

Set
$$\CT:=(\sfq^-\times \sfp^-)_\blacktriangle(\omega_{\Bun^\mu_{P^-}}).$$ 
We have 
$$(\sfp^-)_\blacktriangle\circ (\sfq^-)^! \simeq \sF_{\CT}.$$

\sssec{}   \label{sss:j tilde}

Recall the open embedding 
$$\sfj:\Bun_M^\mu\hookrightarrow \Bun^\mu_{P^-}\underset{\Bun_G}\times \Bun^\mu_P.$$

\medskip

Let $\sfj_{P^-}$ denote the open embedding
$$\Bun^\mu_{P^-}\hookrightarrow \Bun^\mu_{P^-}\underset{\Bun_G}\times \Bun_P^\mu\underset{\Bun_M^\mu}\times \Bun^\mu_{P^-}
\simeq (\Bun_{P^-,\wt{G}}^\mu)_0,$$
obatined by base change. 

\medskip 

We have
$$(\sfq^-\times \sfp^-)=\sfr_0\circ \sfj_{P^-},\quad \Bun^\mu_{P^-}\to \Bun_M^\mu\times \Bun_G.$$

\sssec{}

By construction, the natural transformation $\Phi\to (\sfp^-)_\blacktriangle\circ (\sfq^-)^!$
comes from the map 
\begin{equation} \label{e:S_0 to T}
\CS_0\to \CT
\end{equation}
equal to
$$(\sfr_0)_\blacktriangle (\omega_{(\Bun_{P^-,\wt{G}}^\mu)_0})\to
(\sfr_0)_\blacktriangle\circ (\sfj_{P^-})_\blacktriangle (\omega_{\Bun^\mu_{P^-}})\simeq (\sfq^-\times \sfp^-)_\blacktriangle(\omega_{\Bun^\mu_{P^-}}),$$
where the first arrow comes from
$$\omega_{(\Bun_{P^-,\wt{G}}^\mu)_0}\to (\sfj_{P^-})_*\circ (\sfj_{P^-})^*(\omega_{(\Bun_{P^-,\wt{G}}^\mu)_0})\simeq 
(\sfj_{P^-})_*(\omega_{\Bun^\mu_{P^-}})\simeq
(\sfj_{P^-})_\blacktriangle(\omega_{\Bun^\mu_{P^-}}).$$

\sssec{}

The identification $(\Bun^\mu_{\wt{G}})_1\simeq \Bun_G$ of \eqref{e:tilde mu 1} gives rise to an identification
\begin{equation}  \label{e:- 1}
(\Bun_{P^-,\wt{G}}^\mu)_1\simeq \Bun^\mu_{P^-},
\end{equation}
so that 
$$\sfr_1=(\sfq^-\times \sfp^-).$$

\medskip

Hence, we obtain a tautological identification
\begin{equation} \label{e:T to S_1}
\CT\simeq \CS_1.
\end{equation}

\sssec{}

Now, the map $\on{Sp}_\CS$ of \eqref{e:specialization} defines a canonical map
\begin{equation} \label{e:S_1 to S_0}
\CS_1\to \CS_0.
\end{equation}

By the functoriality of the construction of the natural transformation $\on{Sp}$ (see, e.g., \cite[Sect. 4.1.5]{DrGa3}), 
the natural transformation $(\sfp^-)_\blacktriangle\circ (\sfq^-)^!\to \Phi$
comes from the map
\begin{equation} \label{e:T to S_0}
\CT\to \CS_1\to \CS_0,
\end{equation}
equal to the composition of \eqref{e:T to S_1} and \eqref{e:S_1 to S_0}.

\sssec{}

We obtain that in order to prove that the composition \eqref{e:first composition} is the identity map, 
it suffices to show that the composed map
\begin{equation} \label{e:composed kernels}
\CT\to \CS_1\to \CS_0\to \CT
\end{equation}
is the identity map on $\CT$. 

\ssec{Passing to an open substack}

\sssec{}

Recall the open embedding
$$\sfj_{P^-}:\Bun^\mu_{P^-}\hookrightarrow  (\Bun_{P^-,\wt{G}}^\mu)_0,$$
introduced in \secref{sss:j tilde}. 

\medskip

Let 
$$\overset{\circ}\Bun{}_{P^-,\wt{G}}^\mu\subset \Bun_{P^-,\wt{G}}^\mu$$
be the open substack obtained by removing the closed substack
$$\left((\Bun_{P^-,\wt{G}}^\mu)_0-\sfj_{P^-}(\Bun^\mu_{P^-})\right)\subset (\Bun_{P^-,\wt{G}}^\mu)_0\subset 
\Bun_{P^-,\wt{G}}^\mu.$$

\medskip

Let $(\overset{\circ}\Bun{}_{P^-,\wt{G}}^\mu)_t$ denote the fiber of $\overset{\circ}\Bun{}_{P^-,\wt{G}}^\mu$ over $t\in \BA^1$. 

\medskip

Note that the isomorphism $\Bun^\mu_{P^-}\to (\Bun_{P^-,\wt{G}}^\mu)_1$ of \eqref{e:- 1}
defines an isomorphism 
\begin{equation} \label{e:fiber at 1 open}
\Bun^\mu_{P^-}\to (\overset{\circ}\Bun{}_{P^-,\wt{G}}^\mu)_1.
\end{equation}

\medskip

By definition, the open embedding $\sfj_{P^-}:\Bun^\mu_{P^-}\hookrightarrow  (\Bun_{P^-,\wt{G}}^\mu)_0$ 
defines an \emph{isomorphism}
\begin{equation} \label{e:fiber at 0 open}
\Bun^\mu_{P^-}\to (\overset{\circ}\Bun{}_{P^-,\wt{G}}^\mu)_0.
\end{equation}

\sssec{}

Let 
$$\osfr:\overset{\circ}\Bun{}_{P^-,\wt{G}}^\mu\to \BA^1\times \Bun_M^\mu\times \Bun_G 
\text{ and }
\osfr_t:\overset{\circ}\Bun{}_{P^-,\wt{G}}^\mu\to \Bun_M^\mu\times \Bun_G$$
denote the corresponding maps.

\medskip

\medskip

Set
$$\oCS:=\osfr_\blacktriangle(\omega_{\overset{\circ}\Bun{}_{P^-,\wt{G}}^\mu}),$$
and also
$$\oCS_0:=(\osfr_0)_\blacktriangle(\omega_{(\overset{\circ}\Bun{}_{P^-,\wt{G}}^\mu)_0}) \text{ and }
\oCS_1:=(\osfr_1)_\blacktriangle(\omega_{(\overset{\circ}\Bun{}_{P^-,\wt{G}}^\mu)_1}).$$

\medskip

The open embedding $\overset{\circ}\Bun{}_{P^-,\wt{G}}^\mu\hookrightarrow \Bun_{P^-,\wt{G}}^\mu$ gives rise to the maps
$$\CS\to \oCS,\quad  \CS_0\to \oCS_0, \quad \CS_1\to \oCS_1.$$

\medskip

As in \secref{ss:nat trans via kernels}, we have the natural transformations
\begin{equation}  \label{e:composed kernels open}
\CT\to \oCS_1\to \oCS_0\to \CT.
\end{equation}

Moreover, the diagram 
$$
\CD
\CT   @>>>  \CS_1  @>>> \CS_0  @>>>  \CT  \\
@V{\id}VV  @VVV    @VVV   @VV{\id}V   \\
\CT  @>>>  \oCS_1   @>>>   \oCS_0  @>>>  \CT
\endCD
$$
commutes. 

\medskip

Hence, we obtain that in order to show that the composed map \eqref{e:composed kernels} is the identity map,
it suffices to show that the composed map \eqref{e:composed kernels open} is the identity map. 

\ssec{Digression: description of fiber products}

\sssec{}

Let $Z$ be a quasi-compact scheme equipped with a $\BG_m$-action.
Consider the fiber product $Z^-\underset{Z}\times \wt{Z}$, formed using the composition
$$\wt{Z}\overset{\wt{p}}\longrightarrow \BA^1\times Z\times Z\to Z\times Z \overset{\on{pr}_1}\longrightarrow Z.$$

\medskip

In \cite[Sect. 2.6]{DrGa3} we define a canonical morphism
\begin{equation}  \label{e:canonical_maps}
\BA^1\times Z^-\to Z^-\underset{Z}\times \wt{Z}
\end{equation}
of schemes over $\BA^1$. 

\sssec{}

If $Z$ is affine (the case of interest for us) the definition from \cite{DrGa3} can be reformulated as follows. 
Recall that if $Z$ is affine then $Z^-$ is a closed subscheme of $Z$ and
$\wt{Z}$ is a closed subscheme of $\BA^1\times Z\times Z$, so 
$Z^-\underset{Z}\times \wt{Z}$ is a closed subscheme of $\BA^1\times Z\times Z$. 

\medskip

Now define the map 
$$\BA^1\times Z^-\to Z^-\underset{Z}\times \wt{Z}\subset\BA^1\times Z\times Z$$ by
\begin{equation}   \label{e:can_map2}
(t,z)\mapsto (t\,,t^{-1}\cdot z\, , z)\in\BA^1\times Z\times Z\,.
\end{equation}

It is easy to see that the image of the map \eqref{e:can_map2} indeed belongs to 
$Z^-\underset{Z}\times \wt{Z}$. Indeed, it suffices to consider the case $t\neq 0$,
when this is obvious. 

\sssec{}

By \cite[Proposition 2.6.3 and Remark 2.6.4]{DrGa3}, for any $Z$ the moprhism \eqref{e:canonical_maps} is an open 
embedding, and if $Z$ is affine it is an isomorphism. 

\medskip

The latter is very easy to check if $Z$ is affine and smooth: 
in this case the map \eqref{e:canonical_maps} is a morphism between \emph{smooth} schemes over $\BA^1$, 
and it suffices to check that it induces an isomorphism between the fibers over any $k$-point of $\BA^1$.

\sssec{}

Consider now the situation of \secref{sss:group set-up}, i.e.,  $Z=G$ and the $\BG_m$-action on $G$ is the adjoint 
action corresponding to a co-character $\gamma:\BG_m\to M$ such that $\gamma(\BG_m)$ is contained in the 
center of $M$ and $\gamma$ is dominant and regular with respect to $P$. 

\medskip

Then the map \eqref{e:canonical_maps} is a homomorphism 
\begin{equation}    \label{e:Gcanonical_map-}
\BA^1\times P^-\to P^-\underset{G}\times \wt{G}
\end{equation}
of group-schemes over $\BA^1$. 
Its composition with the embedding $P^-\underset{G}\times \wt{G}\mono\BA^1\times G\times G$
is given by $$(t,g)\mapsto (t\,,\gamma (t)^{-1}\cdot g\cdot \gamma (t)\, , g).$$

\sssec{}

The homorphism \eqref{e:Gcanonical_map-} induces a maps of moduli of bundles
\begin{equation} \label{e:-}
\BA^1\times \Bun_{P^-}\to \Bun_{P^-}\underset{\Bun_G}\times \Bun_{\wt{G}}.
\end{equation}

\begin{lem} \label{l:open}
The map \eqref{e:-} is an open embedding. 
\end{lem} 

This lemma is a counterpart of \cite[Proposition 2.6.3]{DrGa3}.

\begin{proof} 
Note that \lemref{l:openness by fibers} remains valid for stacks. So it suffices to show that the map \eqref{e:-} induces an open embedding 
of fibers over any $t\in \BA^1$. This is clear if $t\ne 0$, so it remains to consider the case of $t=0$. 

\medskip

The morphsim in question is the map
\begin{multline*}
\Bun_{P^-}\to  \Bun_{P^-}\underset{\Bun_G} \times \Bun_{P\underset{M}\times P^-}\simeq 
\Bun_{P^-}\underset{\Bun_G}\times (\Bun_P \underset{\Bun_M}\times \Bun_{P^-})=\\
=(\Bun_{P^-}\underset{\Bun_G}\times \Bun_P) \underset{\Bun_M}\times \Bun_{P^-}
\end{multline*}
which equals the map
$$\Bun_{P^-}\simeq 
\Bun_M \underset{\Bun_M}\times \Bun_{P^-}\overset{\sfj\times \on{id}}
\hookrightarrow (\Bun_{P^-}\underset{\Bun_G}\times \Bun_P) \underset{\Bun_M}\times \Bun_{P^-}\; .$$

\medskip

It is an open embedding because so is the morphism 
$\sfj :\Bun_M\to\Bun_{P^-}\underset{\Bun_G}\times \Bun_P $.
\end{proof}
 
 \begin{rem}
One can also prove \lemref{l:open} using a variant of \lemref{l:2fiber product}(ii). In this variant $H_1\,$, 
$H_2\,$, and $H$ are flat group-schemes over some scheme $S$ (e.g., $S=\BA^1$) and instead of openness of 
$f_1(H_1)\cdot f_2(H_2)$ one requires flatness of the morphism $H_1\underset{S}\times H_2\to H$ given by 
$(h_1, h_2)\mapsto f_1(h_1)\cdot f_2(h_2)\,$.
\end{rem}
 
\ssec{The key argument}  

\sssec{}

By construction, the open embedding \eqref{e:-} defines an \emph{isomorphism}
$$\BA^1\times \Bun^\mu_{P^-}\to \overset{\circ}\Bun{}_{P^-,\wt{G}}^\mu$$
with the following properties:

\begin{itemize}

\item 
The map $$\osfr:\overset{\circ}\Bun{}_{P^-,\wt{G}}^\mu\to  \BA^1\times \Bun_M^\mu\times \Bun_G$$
identifies with the map
$$\id_{\BA^1}\times (\sfq^-\times \sfp^-).$$

\item 
The isomorphism $\Bun^\mu_{P^-}\to (\overset{\circ}\Bun{}_{P^-,\wt{G}}^\mu)_1$ of \eqref{e:fiber at 1 open}
corresponds to the identity map
$$\Bun^\mu_{P^-}\simeq (\BA^1\times \Bun^\mu_{P^-})\underset{\BA^1}\times \{1\}\simeq \Bun^\mu_{P^-}.$$

\item 
The isomorphism $\Bun^\mu_{P^-}\to (\overset{\circ}\Bun{}_{P^-,\wt{G}}^\mu)_0$ of \eqref{e:fiber at 0 open}
corresponds to the identity map
$$\Bun^\mu_{P^-}\simeq (\BA^1\times \Bun^\mu_{P^-})\underset{\BA^1}\times \{0\}\simeq \Bun^\mu_{P^-}.$$

\end{itemize}

\sssec{}

Hence, we obtain that the composition \eqref{e:composed kernels open} identifies with
$$\CT\simeq \iota_1^!(\omega_{\BA^1}\boxtimes \CT) \to \iota_0^!(\omega_{\BA^1}\boxtimes \CT) \simeq \CT,$$
where 
$$\iota_1^!(\omega_{\BA^1}\boxtimes \CT) \to \iota_0^!(\omega_{\BA^1}\boxtimes \CT)$$
is the specialization map \eqref{e:specialization} for the object
$$\omega_{\BA^1}\boxtimes \CT \in \Dmod(\BA^1\times \CZ^0\times \CZ).$$

\medskip

However, the latter is readily seen to be the identity map. 

\section{An alternative proof}  \label{s:mainproof}

In this section we will give an alternative proof of \thmref{t:main adj} by directly deducing it from
Braden's theorem (\cite{Br}). It has the advantage of being more elementary than the proof of \thmref{t:main adj}
given in \secref{s:proof of adj}, if one accepts Braden's theorem as a black box. 

\medskip

However, the two proofs are closely related, because
Braden's theorem itself can be proved by an argument similar to one used in the proof of 
\thmref{t:main adj} from \secref{s:proof of adj}. 

\ssec{Contraction principle}  \label{ss:contr princ}

\sssec{}    \label{sss:mon}

Let $\CY$ be a stack equipped with an action of $\BG_m$. We let
$$\Dmod(\CY)^{\BG_m,\on{mon}}\subset \Dmod(\CY)$$ denote the corresponding monodromic subcategory, i.e., the 
full subcategory generated by the essential image of the pullback functor 
$\Dmod(\CY/\BG_m)\to \Dmod(\CY)$. 

\medskip

Note that if the $\BG_m$-action is trivial then 
$\Dmod(\CY)^{\BG_m,\on{mon}}=\Dmod(\CY)$ (because the morphism $\CY\to \CY/\BG_m$ admits a section).

\sssec{}  \label{sss:stacky contraction}

Let $\CY$ be an algebraic stack equipped
with an action of $\BA^1$,  where $\BA^1$ is viewed as a monoid with respect to multiplication. 
(A concrete example of this situation will be considered in \secref{sss:contraction on moduli simple} below).

\medskip

Define the stack $\CY^0$ by
$$\Maps(S,\CY^0):=\Maps^{\BA^1}(S,\CY), \quad S\in \on{Sch}^{\on{aff}}.$$

\medskip

The groupoid $\Maps(S,\CY^0)$ admits also the following description (which was used in \cite[Sect. C.5]{DrGa2} as a definition): 
objects of  $\Maps(S,\CY^0)$ are pairs $(y,\alpha)$, where 
$y\in \Maps(S,\CY)$ and $\alpha$ is an isomorphism $0\cdot y\iso y$ such that the two isomorphisms 
\[
0\cdot (0\cdot y)=(0\cdot 0)\cdot y=0\cdot y \overset{\alpha}\iso y \text{ and }
0\cdot (0\cdot y) \overset{0\cdot \alpha}\iso 0\cdot y \overset{\alpha}\iso y \]
$0\cdot (0\cdot y)\rightrightarrows y$ are equal to each other. Morphisms in $\Maps(S,\CY^0)$ are defined as
morphisms between the $y$'s that intertwine the data of $\alpha$. 

\medskip

From the latter description, it is easy to see that $\CY^0$ is again a (locally QCA) stack. 

\sssec{}

The forgetful map and the action of $0\in \BA^1$ define a retraction
\begin{equation} \label{e:stacky contraction}
\CY^0\overset{\imath}\longrightarrow \CY\overset{\sfq}\longrightarrow \CY^0.
\end{equation}

Under these circumstances we will say that the we are given a \emph{contraction} of $\CY$ 
onto $\CY^0\overset{\imath}\longrightarrow \CY$.

\medskip

Since $\sfq\circ \imath=\on{id}_{\CY^0}$, it follows that $\imath$ is quasi-compact and representable.

\sssec{}

From now on we will assume that the morphism $\sfq$ is quasi-compact. In this case the functor
$$\sfq_\blacktriangle:\Dmod(\CY)\to \Dmod(\CY^0)$$ is defined. 

\medskip

We have the following assertion: 

\medskip

\begin{prop}  \label{p:simple Braden}  \hfill

\smallskip

\noindent{\em(1)} The partially defined left adjoint $\imath^*:\Dmod(\CY)\to \Dmod(\CY^0)$ of $\imath_*\simeq \imath_\blacktriangle$ is defined
on $\Dmod(\CY)^{\BG_m,\on{mon}}$, and we have a canonical isomorhism
$$i^*|_{\Dmod(\CY)^{\BG_m,\on{mon}}}\simeq q_\blacktriangle|_{\Dmod(\CY)^{\BG_m,\on{mon}}} \; .$$
More precisely, for each $\CF\in \Dmod(\CY)^{\BG_m,\on{mon}}$
the natural map
$$q_\blacktriangle(\CF)\to q_\blacktriangle\circ i_*\circ i^*(\CF)\simeq q_\blacktriangle\circ i_\blacktriangle\circ i^*(\CF)=
(q\circ i)_\blacktriangle\circ i^*(\CF)=i^*(\CF)$$
is an isomorphism.

\smallskip

\noindent{\em(2)} The partially defined left adjoint $q_!:\Dmod(\CY)\to \Dmod(\CY^0)$ of $q^!$ is defined
on $\Dmod(\CY)^{\BG_m,\on{mon}}$, and we have a canonical isomorhism
$$q_!|_{\Dmod(\CY)^{\BG_m,\on{mon}}}\simeq i^!|_{\Dmod(\CY)^{\BG_m,\on{mon}}} \; .$$
More precisely, for each $\CF\in \Dmod(\CY)^{\BG_m,\on{mon}}$
the natural map
$$i^!(\CF)\to i^!\circ q^!\circ q_!(\CF) =(q\circ i)^!\circ q_!(\CF) =q_! (\CF)$$
is an isomorphism.
\end{prop}  

For the proof see \cite[Theorem C.5.3]{DrGa2}.

\sssec{A contraction of $\Bun_P$ onto $\Bun_M\,$}  \label{sss:contraction on moduli simple}  \hfill

\medskip

Let us take $\CY=\Bun_P\,$. Let $\gamma:\BG_m\to M$ be a co-character as in \eqref{e:co-character}. The resulting 
adjoint action of $\BG_m$ on $P$ extends to an action of $\BA^1$ on $P$ such that $0\in \BA^1$
acts as
$$P\twoheadrightarrow M=P\cap P^-\hookrightarrow P.$$
The action of $\BA^1$ on $P$ induces on action of $\BA^1$ on $\Bun_P\,$.

\medskip

It is easy to check that the resulting diagram \eqref{e:stacky contraction} identifies with
$$\Bun_M\overset{\iota}\longrightarrow \Bun_P\overset{\sfq}\longrightarrow \Bun_M,$$
where the maps $\iota$ and $\sfq$ come from the above homomorphisms $M\hookrightarrow P$ and $P\twoheadrightarrow M$.

\medskip

Note that the $\BG_m$-action on $\Bun_P$, corresponding to the above $\BA^1$-action,
is canonically isomorphic to the trivial one\footnote{This trivialization is \emph{not} compatible with the projection $\sfq:\Bun_P\to \Bun_M$, 
so the $\BG_m$-action on the \emph{fibers} of the morphism $\Bun_P\to\Bun_M$ is non-trivial\,!} (because the $\BG_m$-action on $P$ 
comes from a homomorphism $\BG_m\to P$ and the conjugation action of $P$ on itself).

\medskip

Therefore, the inclusion
$$\Dmod(\Bun_P)^{\BG_m,\on{mon}}\subset \Dmod(\Bun_P)$$
is an equality. So by \propref{p:simple Braden}, the functors
$$\sfq_!:\Dmod(\Bun_P)\to  \Dmod(\Bun_M) \text{ and } \iota^*:\Dmod(\Bun_P)\to \Dmod(\Bun_M)$$
left adjoint to $\sfq^!$ and $\iota_*\simeq \iota_\blacktriangle$ are defined and are canonically isomorhic to 
$\iota^!$ and $\sfq_\blacktriangle$, respectively.  

\medskip

Note also that in this example, $\sfq_\blacktriangle\simeq\sfq_*$, because the morphism $\sfq$ is safe. 

\medskip

The same is, of course, true for $P$ replaced by $P^-$, when we replace the above action of $\BG_m$
by its inverse.

\ssec{Hyperbolic restrictions and Braden's theorem}  \label{ss:hyper}

In this subsection we will recall the statement of Braden's theorem in the set-up and notations
of \cite[Sect. 3.1]{DrGa3}.  

\sssec{}  \label{sss:pro}

The material in this subsection uses the notion of pro-completion $\on{Pro}(\bC)$ of
a DG category $\bC$; we refer the reader to \cite[Appendix A]{DrGa3}, where the corresponding definitions
are given.

\medskip

Let $f:\CY_1\to \CY_2$ be a quasi-compact and safe morphism between stacks. Consider the functors
$$f^!:\Dmod(\CY_2)\to \Dmod(\CY_1) \text{ and } f_*\simeq f_\blacktriangle:\Dmod(\CY_1)\to \Dmod(\CY_2),$$
We will regard their partially defined left adjoints as functors
$$f_!:\Dmod(\CY_1)\to \on{Pro}(\Dmod(\CY_2)) \text{ and } f^*:\Dmod(\CY_2)\to \on{Pro}(\Dmod(\CY_1)).$$

\sssec{}

Let $Z$, $Z^0$, $Z^{\pm}$, $p^{\pm}$,  $i^{\pm}$ be as in \secref{sss:Z+}. 
\medskip

Consider the corresponding commutative (but not necessarily Cartesian) diagram
\begin{equation} \label{e:attr rep diagram}
\CD
Z^0  @>{i^+}>>  Z^+  \\
@V{i^-}VV    @VV{p^+}V   \\
Z^-  @>{p^-}>> Z.
\endCD
\end{equation}

We enlarge it to the diagram
\begin{equation} \label{e:square with arrow}
\xy
(0,0)*+{Z^+\underset{Z}\times Z^-}="X";
(30,0)*+{Z^+}="Y";
(0,-30)*+{Z^-}="Z";
(30,-30)*+{Z,}="W";
(-20,20)*+{Z^0}="U";
{\ar@{->}_{p^-} "Z";"W"};
{\ar@{->}^{p^+} "Y";"W"};
{\ar@{->}_{'p^-} "X";"Y"};
{\ar@{->}^{'p^+} "X";"Z"};
{\ar@{->}_{j} "U";"X"};
{\ar@{->}_{i^-} "U";"Z"};
{\ar@{->}^{i^+} "U";"Y"};
\endxy
\end{equation}
in which the square is Cartesian, and where (according to \cite[Proposition 1.9.4]{DrGa3}) the map
$$j:Z^0\to Z^+\underset{Z}\times Z^-$$
is an open embedding. 

\medskip

We consider the categories
$$\Dmod(Z)^{\BG_m\on{-mon}},\,\, \Dmod(Z^+)^{\BG_m\on{-mon}}, \Dmod(Z^-)^{\BG_m\on{-mon}}$$ and   
$$\Dmod(Z^0)^{\BG_m\on{-mon}}=\Dmod(Z^0).$$

Consider the functors
$$(p^+)^!:\Dmod(Z)^{\BG_m\on{-mon}}\to \Dmod(Z^+)^{\BG_m\on{-mon}} \text{ and } 
(i^-)^!:\Dmod(Z^-)^{\BG_m\on{-mon}}\to \Dmod(Z^0).$$

Consider also the functors 

$$(p^-)^*: \Dmod(Z)^{\BG_m\on{-mon}}\to \on{Pro}(\Dmod(Z^-)^{\BG_m\on{-mon}})$$  and 
$$(i^+)^*: \Dmod(Z^+)^{\BG_m\on{-mon}}\to \on{Pro}(\Dmod(Z^0)),$$
left adjoint in the sense of \cite[Sect. A.3]{DrGa3} to
$$(p^-)_*:\Dmod(Z^-)^{\BG_m\on{-mon}}\to \Dmod(Z)^{\BG_m\on{-mon}}$$ and 
$$(i^+)_*: \Dmod(Z^0)\to  \Dmod(Z^+)^{\BG_m\on{-mon}},$$
respectively. 

\sssec{}

Consider the composed functors
$$(i^+)^*\circ (p^+)^! \text{ and } (i^-)^!\circ (p^-)^*,\quad 
\Dmod(Z)^{\BG_m\on{-mon}}\to \on{Pro}(\Dmod(Z^0)).$$

They are called the functors of \emph{hyperbolic restriction}.  

\medskip 

We note that there is a canonical natural transformation
\begin{equation} \label{e:Braden trans}
(i^+)^*\circ (p^+)^! \to (i^-)^!\circ (p^-)^*.
\end{equation}

Namely, the natural transformation \eqref{e:Braden trans} is obtained via the $((i^+)^*,(i^+)_*)$-adjunction from the natural transformation
$$(p^+)^! \to (i^+)_*\circ (i^-)^!\circ (p^-)^*,$$
defined in terms of \eqref{e:square with arrow} as follows:
\begin{multline*}
(p^+)^!\to (p^+)^!\circ (p^-)_* \circ (p^-)^* \simeq ({}'p^-)_*\circ ({}'p^+)^!\circ (p^-)^*\to \\
\to ({}'p^-)_*\circ j_*\circ j^! \circ ({}'p^+)^! \circ (p^-)^*
\simeq  (i^+)_*\circ (i^-)^!\circ (p^-)^*,
\end{multline*}
where the map 
$$\on{Id}\to j_*\circ j^!$$
comes from the $(j^!,j_*)$-adjunction using the fact that $j$ is an open embedding, and where $(p^+)^!\circ (p^-)_*\simeq ({}'p^-)_*\circ ({}'p^+)^!$
is the base change isomorphism.

\sssec{}

Braden's theorem of \cite{Br} (as stated in \cite[Theorem 3.1.6]{DrGa3}) reads: 

\begin{thm} \label{t:Braden original}  
The functors 
$$(i^+)^*\circ (p^+)^! \text{ and } (i^-)^!\circ (p^-)^*, \quad \Dmod(Z)^{\BG_m\on{-mon}}\to \on{Pro}(\Dmod(Z^0))$$
take values in $\Dmod(Z^0)\subset \on{Pro}(\Dmod(Z^0))$
and the map
\eqref{e:Braden trans}
is an isomophism.
\end{thm}

\ssec{The setting for stacks}   

\sssec{}  \label{sss:stacky version}

Assume now that we are given a commuttaive diagram of algebraic stacks:

\begin{equation} \label{e:stacky Braden diag}
\xy
(0,0)*+{\CY^+\underset{\CY}\times \CY^-}="X";
(30,0)*+{\CY^+}="Y";
(0,-30)*+{\CY^-}="Z";
(30,-30)*+{\CY,}="W";
(-20,20)*+{\CY^0}="U";
{\ar@{->}_{\sfp^-} "Z";"W"};
{\ar@{->}^{\sfp^+} "Y";"W"};
{\ar@{->}_{'\sfp^-} "X";"Y"};
{\ar@{->}^{'\sfp^+} "X";"Z"};
{\ar@{->}_{\sfj} "U";"X"};
{\ar@{->}_{\iota^-} "U";"Z"};
{\ar@{->}^{\iota^+} "U";"Y"};
\endxy
\end{equation}
where the square is Cartesian, and the map $\sfj:\CY^0\to \CY^+\underset{\CY}\times \CY^-$
is an open embedding. 

\medskip

We will assume that all morphisms in \eqref{e:stacky Braden diag} are quasi-compact and safe. 

\medskip

Consider the functors 
$$(\iota^-)^*\circ (\sfp^-)^! \text{ and } (\iota^+)^!\circ (\sfp^+)^*: \Dmod(\CY)\to \on{Pro}(\Dmod(\CY^0)).$$

\medskip

As in the case of \eqref{e:Braden trans}, we obtain a natural transformation
\begin{equation} \label{e:Braden trans stacks}
(\iota^-)^*\circ (\sfp^-)^!\to (\iota^+)^!\circ (\sfp^+)^*.
\end{equation}

\sssec{Definition of hyperbolicity for stacks}  \label{sss:hyperb stacks}

We will say that the diagram
\eqref{e:stacky Braden diag} is \emph{hyperbolic} if the functors
$(\iota^-)^* \circ (\sfp^-)^!$ and $(\iota^+)^! \circ (\sfp^+)^*$ take values in $$\Dmod(\CY^0)\subset  \on{Pro}(\Dmod(\CY^0)),$$ and 
the map \eqref{e:Braden trans stacks} is an isomorphism. 

\sssec{Checking hyperbolicity}

Let $Z$ be again as in \secref{sss:Z+}. Let us be given a commutative cube
\begin{equation} \label{e:cube}
\xy
(0,0)*+{Z}="X";
(-20,20)*+{Z^-}="Y";
(-20,50)*+{Z^0}="Z";
(0,30)*+{Z^+}="W";
(50,0)*+{\CY,}="cX";
(30,20)*+{\CY^-}="cY";
(30,50)*+{\CY^0}="cZ";
(50,30)*+{\CY^+}="cW";
{\ar@{->}^{i^-} "Z";"Y"};
{\ar@{->}^{p^+} "W";"X"};
{\ar@{->}^{p^-} "Y";"X"};
{\ar@{->}^{i^+} "Z";"W"};
{\ar@{->}^{\iota^-} "cZ";"cY"};
{\ar@{->}^{\sfp^+} "cW";"cX"};
{\ar@{->}^{\sfp^-} "cY";"cX"};
{\ar@{->}^{\iota^+} "cZ";"cW"};
{\ar@{->}^{\psi}  "X";"cX"};
{\ar@{->}^{\psi^-}  "Y";"cY"};
{\ar@{->}^{\psi^0}  "Z";"cZ"};
{\ar@{->}_{\psi^+}  "W";"cW"};
\endxy
\end{equation}
where the maps
$$\psi:Z\to \CY,\,\,\psi^+:Z^+\to \CY^+,\,\, \psi^-:Z^-\to \CY^-,\,\,  \psi^0:Z^0\to \CY^0$$
are smooth, and $\psi^0$ surective. Assume also that the morphism $\psi:Z\to \CY$
can be given a $\BG_m$-equivariant structure (with respect to the trivial action of $\BG_m$
on $\CY$). 

\medskip

We claim:

\begin{thm} \label{t:hyperb}
Under the above circumstances the diagram \eqref{e:stacky Braden diag} is hyperbolic.
\end{thm}

The proof is given in Appendix \ref{s:check hyperb}. 

\ssec{Hyperbolocity and adjunction}  

\sssec{}  \label{sss:q for stacks}

Assume that in the situation of diagram \eqref{e:stacky Braden diag}, 
the morphisms $\iota^-$ and $\iota^+$ in \eqref{e:stacky Braden diag}
admit left inverses, denoted $\sfq^-$ and $\sfq^+$, respectively. Assume also that
the morphisms $\sfq^{\pm}$ are quasi-compact.

\medskip

The diagram

$$
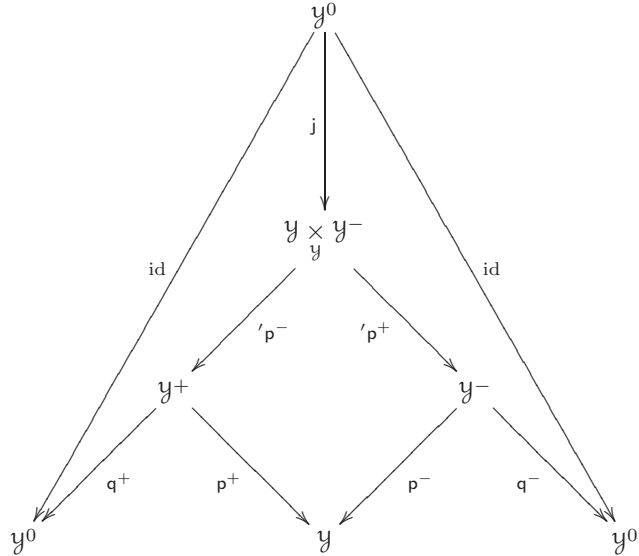
\xy
(20,0)*+{\CY}="X";
(-20,0)*+{\CY^0}="Y";
(0,20)*+{\CY^+}="Z";
(40,20)*+{\CY^-}="W";
(60,0)*+{\CY^0}="U";
(20,40)*+{\CY\underset{\CY}\times \CY^-}="V";
(20,70)*+{\CY^0}="T";
{\ar@{->}_{\sfp^+} "Z";"X"};
{\ar@{->}^{\sfq^+} "Z";"Y"};
{\ar@{->}^{\sfp^-} "W";"X"};
{\ar@{->}_{\sfq^-} "W";"U"}; 
{\ar@{->}^{'\sfp^-} "V";"Z"};
{\ar@{->}_{'\sfp^+} "V";"W"};
{\ar@{->}_{\sfj} "T";"V"};
{\ar@{->}_{\on{id}} "T";"Y"};
{\ar@{->}^{\on{id}} "T";"U"};
\endxy
$$
gives rise to a natural transformation
\begin{equation} \label{e:adj abstr}
(\sfq^-)_\blacktriangle\circ (\sfp^-)^!\circ (\sfp^+)_* \circ (\sfq^+)^!\simeq 
(\sfq^-)_\blacktriangle\circ (\sfp^-)^!\circ (\sfp^+)_\blacktriangle \circ (\sfq^+)^!\to \on{Id}_{\Dmod(\CY^0)},
\end{equation}
which by adjunction gives rise to a natural transformation
\begin{equation} \label{e:qp}
(\sfq^-)_\blacktriangle\circ (\sfp^-)^!\to (\sfq^+)_!\circ (\sfp^+)^*, \quad \Dmod(\CY)\to \on{Pro}(\Dmod(\CY^0)).
\end{equation}

\sssec{}   \label{sss:adj and hyp}

Assume now that the natural transformation
$$(\sfq^-)_\blacktriangle\circ (\iota^-)_*\simeq (\sfq^-)_\blacktriangle\circ (\iota^-)_\blacktriangle
\overset{\sim}\to \on{Id}_{\Dmod(\CY^0)}$$
identifies the functor $(\sfq^-)_\blacktriangle$ with the left adjoint $(\iota^-)^*$ 
of $(\iota^-)_*$.

\medskip

Assume also that the natural transformation
$$(\iota^+)^!\circ (\sfq^+)^!\overset{\sim}\to \on{Id}_{\Dmod(\CY^0)}$$
identifies the functor $(\iota^+)^!$ with the left adjoint $(\sfq^+)_!$ of  
$(\sfq^+)^!$.

\medskip

In particular, both of the above left adjoints take values in
$$\Dmod(\CY^0)\subset \on{Pro}(\Dmod(\CY^0)).$$

\medskip

(For example, such behavior of the morphisms $(\sfq^-,\imath^-)$ and $(\sfq^+,\imath^+)$ occurs in the situation
described in \secref{sss:stacky contraction}.)

\medskip

In this case diagram chase shows that the following diagram of natural transformations commutes:
$$
\CD
(\iota^-)^*\circ (\sfp^-)^!  @>{\text{\eqref{e:Braden trans stacks}}}>>  (\iota^+)^!\circ (\sfp^+)^* \\
@V{\sim}VV    @VV{\sim}V  \\
(\sfq^-)_\blacktriangle\circ (\sfp^-)^!   @>{\text{\eqref{e:qp}}}>> (\sfq^+)_!\circ (\sfp^+)^*.
\endCD
$$

From here we obtain:

\begin{lem}   \label{l:adj and hyp}
The diagram \eqref{e:stacky Braden diag} is hyperbolic if and only if the natural transformation
$$\left((\sfq^-)_\blacktriangle\circ (\sfp^-)^!\right)\circ \left((\sfp^+)_* \circ (\sfq^+)^!\right) \to \on{Id}_{\Dmod(\CY^0)}
$$
of \eqref{e:adj abstr} defines the co-unit of an adjunction, making the functor $(\sfq^-)_\blacktriangle\circ (\sfp^-)^!$
into a left adjoint of $(\sfp^+)_* \circ (\sfq^+)^!$.
\end{lem}

\sssec{}

Note that \lemref{l:adj and hyp}, combined with \propref{p:simple Braden}, imply that the assertion of 
\thmref{t:main adj} (with the co-unit specified in \secref{sss:co-unit}) is equivalent to the following:

\begin{cor}
The diagram
$$
\CD
\Bun_M^\mu @>{\iota}>>  \Bun_P^\mu  \\
@V{\iota^-}VV    @VV{\sfp}V  \\
\Bun^\mu_{P^-}  @>{\sfp^-}>>  \Bun_G
\endCD
$$
is hyperbolic. 
\end{cor}

\ssec{Proof of \thmref{t:main adj}: reduction to the quasi-compact case}  \label{ss:red to qc}

\sssec{}

We need to show that for $\CF_G\in \Dmod(\Bun_G)$ and $\CF_M\in \Dmod(\Bun^\mu_M)$, 
the map 
\begin{equation} \label{e:Hom to check}
\CHom_{\Dmod(\Bun_G)}(\CF_G,\Eis^\mu_*(\CF_M))\to \CHom_{\Dmod(\Bun^\mu_M)}(\on{CT}^{\mu,-}_*(\CF_G),\CF_M),
\end{equation}
induced by \eqref{e:expl co-unit}, is an isomorphism.

\medskip

In this subsection we will show that it is sufficient to prove that \eqref{e:Hom to check} is an isomorphism for any $\CF_M$ of the form 
$(\jmath_M)_*(\CF'_M)$, where 
$$U_M\overset{\jmath_M}\hookrightarrow \Bun^\mu_M$$
is an open quasi-compact substack and $\CF'_M\in \Dmod(U_M)$. 

\sssec{}

Let $\on{open-qc}_G$ denote the poset of open quasi-compact substacks of $\Bun_G$, and let $\on{open-qc}_M$
be the corresponding poset for $\Bun^\mu_M$. 

\medskip

For every
$$(U_M\overset{\jmath_M}\hookrightarrow \Bun^\mu_M)\in \on{open-qc}_M$$ we have a commutative diagram

$$
\CD
\CHom(\CF_G,\Eis^\mu_*(\CF_M))  @>>>  \CHom(\CF_G,\Eis^\mu_*\circ 
(\jmath_M)_*\circ (\jmath_M)^*(\CF_M))  \\
@VVV   @VVV  \\
\CHom(\on{CT}^{\mu,-}_*(\CF_G),\CF_M) 
@>>> \CHom(\on{CT}^{\mu,-}_*(\CF_G),(\jmath_M)_*\circ (\jmath_M)^*(\CF_M)).
\endCD
$$

\medskip

Hence, it suffices to show that the maps
\begin{multline} \label{e:inv limit M}
\CHom_{\Dmod(\Bun_M)}(\on{CT}^{\mu,-}_*(\CF_G),\CF_M)\to \\
\to \underset{U_M\in \on{open-qc}_M}{\underset{\longleftarrow}{lim}}\, 
\CHom_{\Dmod(\Bun_M)}(\on{CT}^{\mu,-}_*(\CF_G),(\jmath_M)_*\circ (\jmath_M)^*(\CF_M))
\end{multline}
and 
\begin{multline}  \label{e:inv limit G}
\CHom_{\Dmod(\Bun_G)}(\CF_G,\Eis^\mu_*(\CF_M))\to \\
\to \underset{U_M\in \on{open-qc}_M}{\underset{\longleftarrow}{lim}}\, 
\CHom_{\Dmod(\Bun_G)}(\CF_G,\Eis^\mu_*\circ 
(\jmath_M)_*\circ (\jmath_M)^*(\CF_M)) 
\end{multline}
are isomorphisms.

\sssec{}

The fact that \eqref{e:inv limit M} is an isomorphism is evident from the definition of $\Dmod(\Bun_M)$ as
$$\underset{U_M\in \on{open-qc}_M}{\underset{\longleftarrow}{lim}}\, \Dmod(U_M),$$
since
\begin{multline*} 
\CHom_{\Dmod(\Bun_M)}(\on{CT}^{\mu,-}_*(\CF_G),(\jmath_M)_*\circ (\jmath_M)^*(\CF_M))\simeq \\
\simeq \CHom_{\Dmod(U_M)}((\jmath_M)^*\circ \on{CT}^{\mu,-}_*(\CF_G),(\jmath_M)^*(\CF_M)).
\end{multline*}

Similarly, the map
\begin{multline}   \label{e:inv limit G next}
\CHom_{\Dmod(\Bun_G)}(\CF_G,\Eis^\mu_*(\CF_M))\to \\
\to \underset{U_G\in \on{open-qc}_G}{\underset{\longleftarrow}{lim}}\, 
\CHom_{\Dmod(\Bun_G)}(\CF_G,(\jmath_G)_*\circ (\jmath_G)^*\circ \Eis^\mu_*(\CF_M))
\end{multline}
is an isomorphism. 

\sssec{}

Note now that we have a commutative diagram
$$
\xy
(-20,0)*+{\CHom(\CF_G,\Eis^\mu_*(\CF_M))}="X";
(0,-30)*+{\underset{U_G\in \on{open-qc}_G}{\underset{\longleftarrow}{lim}} \CHom(\CF_G,(\jmath_G)_*\circ (\jmath_G)^*\circ \Eis^\mu_*(\CF_M))}="Y";
(70,0)*+{ \underset{U_M\in \on{open-qc}_M}{\underset{\longleftarrow}{lim}} \CHom(\CF_G,\Eis^\mu_*\circ (\jmath_M)_*\circ (\jmath_M)^*(\CF_M))}="Z";
(60,-50)*+{\underset{U_G,U_M}{\underset{\longleftarrow}{lim}}
\CHom(\CF_G,(\jmath_G)_*\circ (\jmath_G)^*\circ \Eis^\mu_*\circ (\jmath_M)_*\circ (\jmath_M)^* (\CF_M))}="W";
{\ar@{->}^{\text{\eqref{e:inv limit G next}}} "X";"Y"};
{\ar@{->}^{\text{\eqref{e:inv limit G}}} "X";"Z"};
{\ar@{->}^{\text{\eqref{e:inv limit G next}}} "Z";"W"};
{\ar@{->} "Y";"W"};
\endxy
$$

It remains to show that the map
\begin{multline*} 
\underset{U_G\in \on{open-qc}_G}{\underset{\longleftarrow}{lim}} \CHom(\CF_G,(\jmath_G)_*\circ (\jmath_G)^*\circ \Eis^\mu_*(\CF_M)) \to \\
\to 
\underset{U_G,U_M}{\underset{\longleftarrow}{lim}}
\CHom(\CF_G,(\jmath_G)_*\circ (\jmath_G)^*\circ \Eis^\mu_*\circ (\jmath_M)_*\circ (\jmath_M)^* (\CF_M))
\end{multline*}
is an isomorphism. 

\medskip

However, we claim that for any fixed $U_G$, the corresponding map
\begin{multline*} 
\CHom(\CF_G,(\jmath_G)_*\circ (\jmath_G)^*\circ \Eis^\mu_*(\CF_M)) \to \\
\to 
\underset{U_M\in \on{open-qc}_M}{\underset{\longleftarrow}{lim}}
\CHom(\CF_G,(\jmath_G)_*\circ (\jmath_G)^*\circ \Eis^\mu_*\circ (\jmath_M)_*\circ (\jmath_M)^* (\CF_M))
\end{multline*}
is already an isomorphism. In fact, the map
$$(\jmath_G)^*\circ \Eis^\mu_*(\CF_M)\to 
(\jmath_G)^*\circ \Eis^\mu_*\circ (\jmath_M)_*\circ (\jmath_M)^* (\CF_M)$$
is an isomorphism,
whenever $U_M$ is such that 
$$\sfq(\sfp^{-1}(U_G)\subset U_M,$$ 
and such $U_M$ are cofinal in $\on{open-qc}_M$. 

\ssec{Proof of \thmref{t:main adj}: reduction to the question of hyperbolicity}  

\sssec{}

According to \secref{ss:red to qc}, it suffices to show that for an open quasi-compact substack
$$U_M\overset{\jmath_M}\hookrightarrow \Bun^\mu_M,$$
the natural transformation
$$\jmath_M^*\circ \on{CT}_*^{\mu,-}\circ \Eis^\mu_*\circ (\jmath_M)_*\overset{\text{\eqref{e:expl co-unit}}}
\longrightarrow \jmath_M^*\circ (\jmath_M)_*\to \on{Id}_{\Dmod(U_M)}$$
defines the co-unit of an adjunction. 

\sssec{}

Let $U_G\subset \Bun_G$ be a quasi-compact open substack such that
$$\sfp(\sfq^{-1}(U_M))\subset U_G \text{ and } \sfp^-((\sfq^-)^{-1}(U_M))\subset U_G.$$

\medskip

Consider the corresponding diagram
$$
\xy
(30,0)*+{U_M\underset{\Bun_M}\times \Bun_P}="Y";
(0,-30)*+{U_M\underset{\Bun_M}\times \Bun_{P^-}}="Z";
(30,-30)*+{U_G,}="W";
(-20,20)*+{U_M}="U";
{\ar@{->}_{\sfp_U^-} "Z";"W"};
{\ar@{->}^{\sfp_U} "Y";"W"};
{\ar@{->}^{\iota_U^-} "U";"Z"};
{\ar@{->}_{\iota_U} "U";"Y"};
{\ar@<-1.3ex>_{\sfq_U} "Y";"U"};
{\ar@<1.3ex>^{\sfq^-_U} "Z";"U"};
\endxy
$$

\medskip

It suffices to show that the functor
$$(\sfq_U)_!\circ \sfp_U^*:\Dmod(U_G)\to \on{Pro}(\Dmod(U_M))$$
takes values in $\Dmod(U_M)$, and that the natural transformation 
\begin{equation} \label{e:nat trans U}
(\sfq^-_U)_*\circ (\sfp_U^-)^!\to (\sfq_U)_!\circ (\sfp_U)^*
\end{equation}
of \eqref{e:qp} is an isomorphism. 

\sssec{}

The contraction of $\Bun_P$ (resp., $\Bun_{P^-}$) onto $\Bun_M$ (see \secref{sss:contraction on moduli simple})
and \propref{p:simple Braden} imply that the morphisms
$$\sfq_U:U_M\underset{\Bun_M}\times \Bun_P\rightleftarrows U_M:\iota_U$$
and 
$$\sfq^-_U:U_M\underset{\Bun_M}\times \Bun_{P^-}\rightleftarrows U_M:\iota^-_U$$
satisfy the assumptions of \secref{sss:adj and hyp}.

\medskip

Hence, applying \lemref{l:adj and hyp}, we obtain that it is enough to show that the diagram
\begin{equation} \label{e:U hyperb}
\CD
U_M @>{\iota_U}>> U_M\underset{\Bun_M}\times \Bun_P  \\
@V{\iota^-_U}VV   @VV{\sfp_U}V   \\
U_M\underset{\Bun_M}\times \Bun_{P^-}  @>{\sfp^-_U}>>  U_G
\endCD
\end{equation} 
is hyperbolic in the sense of \secref{sss:hyperb stacks}. 

\ssec{Proof of \thmref{t:main adj}: verification of hyperbolicity}  \label{ss:proof via Braden}

\sssec{}

Fix any point $x\in X$. For $H=G,P,P^-$ or $M$, let $\Bun_H^{n\cdot x}$ be the stack classifying $G$-bundles with a 
structure of level $n$ at $x$. It is known that for $n$ large enough, the open substacks 
$$U_G\underset{\Bun_G}\times \Bun_G^{n\cdot x}\subset \Bun_G^{n\cdot x}$$
$$U_M \underset{\Bun_M}\times \Bun_M^{n\cdot x}\subset \Bun_M^{n\cdot x},$$
$$U_M \underset{\Bun_M}\times \Bun_{P}^{n\cdot x}\simeq
(U_M \underset{\Bun_M}\times \Bun_{P}) \underset{\Bun_{P}}\times \Bun_{P}^{n\cdot x}$$
and 
$$U_M \underset{\Bun_M}\times \Bun_{P^-}^{n\cdot x}\simeq
(U_M \underset{\Bun_M}\times \Bun_{P^-}) \underset{\Bun_{P^-}}\times \Bun_{P^-}^{n\cdot x}$$
are actually quasi-compact separated schemes. 

\medskip

By making $n$ even larger, we can assume that the maps in the diagram
\begin{equation} \label{e:level structure}
\CD
U_M \underset{\Bun_M}\times \Bun_M^{n\cdot x} @>>>  U_M \underset{\Bun_M}\times \Bun_{P}^{n\cdot x} \\
@VVV    @VVV   \\
U_M \underset{\Bun_M}\times \Bun_{P^-}^{n\cdot x}  @>>> U_G\underset{\Bun_G}\times \Bun_G^{n\cdot x}
\endCD
\end{equation}
are locally closed embeddings. 

\medskip

Combining, we obtain the following diagram:

\begin{equation} \label{e:Bun U}
\xy
(0,0)*+{U_G\underset{\Bun_G}\times \Bun_G^{n\cdot x}}="X";
(-20,20)*+{U_M \underset{\Bun_M}\times \Bun_{P^-}^{n\cdot x}}="Y";
(-20,50)*+{U_M \underset{\Bun_M}\times \Bun_M^{n\cdot x}}="Z";
(0,30)*+{U_M \underset{\Bun_M}\times \Bun_{P}^{n\cdot x}}="W";
(50,0)*+{U_G}="cX";
(30,20)*+{U_M \underset{\Bun_M}\times \Bun_{P^-}}="cY";
(30,50)*+{U_M}="cZ";
(50,30)*+{U_M \underset{\Bun_M}\times \Bun_P}="cW";
{\ar@{->} "Z";"Y"};
{\ar@{->} "W";"X"};
{\ar@{->} "Y";"X"};
{\ar@{->} "Z";"W"};
{\ar@{->} "cZ";"cY"};
{\ar@{->}  "cW";"cX"};
{\ar@{->} "cY";"cX"};
{\ar@{->} "cZ";"cW"};
{\ar@{->} "X";"cX"};
{\ar@{->} "Y";"cY"};
{\ar@{->} "Z";"cZ"};
{\ar@{->} "W";"cW"};
\endxy
\end{equation}

\sssec{}

Let $\gamma:\BG_m\to M$ be a co-character as in \eqref{e:co-character}, and consider the corresponding
homomorphism 
$$\BG_m\to H.$$ 
The adjoint action of $\BG_m$ on $H$ defines a $\BG_m$-action on each $\Bun_H^{n\cdot x}$ so that
the forgetful maps
$$\Bun_H^{n\cdot x}\to \Bun_H$$
are $\BG_m$-equivariant. In particular, the scheme 
$$Z:=U_G\underset{\Bun_G}\times \Bun_G^{n\cdot x}$$
acquires a $\BG_m$-action, which is \emph{locally linear}, i.e., $Z$ can be covered by affine subschemes
that are preserved by the $\BG_m$-action. 

\medskip

Consider the corresponding diagram of \eqref{e:attr rep diagram}:
\begin{equation}  \label{e:Z diag again}
\CD
Z^0 @>>>  Z^+ \\
@VVV   @VVV   \\
Z^-  @>>>  Z
\endCD
\end{equation}
for $Z$ as above.

\medskip

By \thmref{t:hyperb}, in order to show that the diagram \eqref{e:U hyperb} is hyperbolic, it is enough to
show that the digram \eqref{e:level structure} admits a map to the diagram \eqref{e:Z diag again},
which is an open embedding \emph{on a Zariski neighborhood of $U_M \underset{\Bun_M}\times \Bun_M^{n\cdot x}$}. 
(Indeed, in this case the diagram \eqref{e:level structure} will identify with the diagram \eqref{e:attr rep diagram}
for an open $\BG_m$-invariant subscheme of $Z$.) 

\sssec{}

We construct the corresponding maps 
\begin{equation} \label{e:incl M}
U_M \underset{\Bun_M}\times \Bun_M^{n\cdot x}\to Z^0,
\end{equation} 
\begin{equation} \label{e:incl P}
U_M \underset{\Bun_M}\times \Bun_{P}^{n\cdot x}\to Z^+,
\end{equation} 
and
\begin{equation} \label{e:incl P-}
U_M \underset{\Bun_M}\times \Bun_{P^-}^{n\cdot x}\to Z^0
\end{equation} 
as follows.

\medskip

The map \eqref{e:incl M} is evident, since the adjoint action of $\BG_m$ on $M$ is trivial. 

\medskip 

The map \eqref{e:incl P} (resp.,  \eqref{e:incl P-}) is given by the $\BA^1$-action on $\Bun_{P}^{n\cdot x}$
(resp., $\Bun_{P^-}^{n\cdot x}$) that contracts it onto $\Bun_{M}^{n\cdot x}$ as in 
\secref{sss:contraction on moduli simple}. 

\sssec{}

The map \eqref{e:incl M} (resp., \eqref{e:incl P}, \eqref{e:incl P-}) is a locally closed embedding, because
its composition with the map $Z^0\to Z$ (resp.,, $Z^+\to Z$, $Z^-\to Z$) is. 

\medskip

It remains to show that the maps \eqref{e:incl M}, \eqref{e:incl P} and \eqref{e:incl P-} are open embeddings
on a Zariski neighborhood of $U_M \underset{\Bun_M}\times \Bun_M^{n\cdot x}$. 

\medskip

Since the left-hand sides are
smooth, it is enough to show that maps in question give rise to isomorphisms of tangent spaces at every
$k$-point of $U_M \underset{\Bun_M}\times \Bun_M^{n\cdot x}$.

\medskip

Let $z$ be a $k$-point of $Z$, where $Z$ is as in \secref{ss:hyper}.
The $\BG_m$-action on $Z$ induces a linear action of $\BG_m$ on the tangent
space $T_z(Z)$. It is easy to see that the subspaces
$$T_z(Z^0),\,\, T_z(Z^+) \text{ and } T_z(Z^-)$$
identify, respectively, with the subspace of zero, non-negative and non-positive characters of $\BG_m$
on $T_z(Z)$.

\medskip

Let $Z:=U_G\underset{\Bun_G}\times \Bun_G^{n\cdot x}$, and let $z$ be a geometric point of 
$U_M \underset{\Bun_M}\times \Bun_M^{n\cdot x}$, corresponding to an $M$-bundle $\CF_M$. 

\medskip

The tangent space $T_z(U_G\underset{\Bun_G}\times \Bun_G^{n\cdot x})$ identifies with $H^1(X,\fg_{\CF_M}(-n\cdot x))$,
while the tangent spaces 
$$T_z(U_M \underset{\Bun_M}\times \Bun_M^{n\cdot x}),\,\, T_z(U_M \underset{\Bun_M}\times \Bun_P^{n\cdot x}) \text{ and }
T_z(U_M \underset{\Bun_M}\times \Bun_{P^-}^{n\cdot x})$$
identify with
$$H^1(X,\fm_{\CF_M}(-n\cdot x)),\,\, H^1(X,\fp_{\CF_M}(-n\cdot x)) \text{ and } H^1(X,\fp^-_{\CF_M}(-n\cdot x)),$$
respectvely.

\medskip

This makes the required assertion manifest.

\appendix

\section{Hyperbolicity for stacks: proof of \thmref{t:hyperb}}  \label{s:check hyperb}

With no restriction of generality, we can assume that $\CY$  is quasi-compact (and hence QCA). Indeed,
otherwise, replace it by the image of the map $\psi$.

\ssec{Step 1}

In this subsection we will show that the functors
\begin{equation} \label{e:hyperb restr}
(\iota^-)^*\circ (\sfp^-)^! \text{ and } (\iota^+)^*\circ (\sfp^+)^!
\end{equation} 
take values in $\Dmod(\CY^0)\subset  \on{Pro}(\Dmod(\CY^0))$. By symmetry, it is sufficient to treat 
the case of the former functor.

\sssec{}

Note that for $\CF\in \Dmod(\CY)$, the object
$$(\psi^0)^*\circ (\iota^-)^*\circ (\sfp^-)^!(\CF),$$
is, isomorphic to $(i^-)^*\circ (\psi^-)^*\circ (\sfp^-)^!(\CF)$ and the latter is, up to a cohomological
shift, isomorphic to $(i^-)^*\circ (p^-)^!\circ \psi^*(\CF)$, since the morphisms $\psi$ and $\psi^-$ are both
smooth.

\medskip

Since $\psi^*(\CF)\in \Dmod(Z)^{\BG_m\on{-mon}}$,
from \thmref{t:Braden original}, we obtain that 
$$(i^-)^*\circ (p^-)^!\circ \psi^*(\CF)\in \Dmod(Z^0).$$

\medskip

Hence, we obtain that 
$$(\psi^0)^*\circ (\iota^-)^*\circ (\sfp^-)^!(\CF)\in \Dmod(Z^0).$$

Now, the fact that $(\iota^-)^*\circ (\sfp^-)^! (\CF)$ belongs to $\Dmod(\CY^0)$ follows from the next
general assertion:

\begin{lem}  \label{l:left adjoint as pullback}
Let $\CZ\overset{g}\longrightarrow \CY_1\overset{f}\longrightarrow \CY_2$ be maps
between quasi-compact stacks with $g$ smooth and surjective. Then for $\CF\in \Dmod(\CY_2)$,
the object $f^*(\CF)$ belongs to $\Dmod(\CY_1)$ if and only if 
$g^*\circ f^*(\CF)$ belongs to $\Dmod(\CZ)$.
\end{lem}

\sssec{Proof of \lemref{l:left adjoint as pullback}}

The ``only" if direction is obvious. For the ``if" direction, with no restriction of generality, we can assume that $\CZ$ 
is a quasi-compact scheme.

\medskip

Let $\CZ^\bullet$ denote the \v{C}ech nerve of the map $g$, and let $g^\bullet:\CZ^\bullet\to \CY_1$ denote the
corresponding maps. By assumption $(g^\bullet)^*\circ f^*(\CF)$ is a well-defined an object of $\on{Tot}(\Dmod(\CZ^\bullet))$.
Since $g$ is smooth and surjective, the functor
\begin{equation} \label{e:to Tot}
(g^\bullet)^*:\Dmod(\CY_1)\to \on{Tot}(\Dmod(\CZ^\bullet))
\end{equation}
is an equivalence. Let $\CF'\in \Dmod(\CY_1)$ denote the object corresponding to $(g^\bullet)^*\circ f^*(\CF)$
under the equivalence \eqref{e:to Tot}. We claim that $\CF'\simeq f^*(\CF)$. 

\medskip

It suffices to show that for $\CF_1\in \Dmod(\CY_1)$, we have a canonical ismorphism
$$\CHom_{\Dmod(\CY_1)}(\CF',\CF_1)\simeq \CHom_{\Dmod(\CY_2)}(\CF,f_*(\CF_1)).$$

We have:
$$\CHom_{\Dmod(\CY_1)}(\CF',\CF_1)\simeq 
\on{Tot}\left(\CHom_{\Dmod(\CZ^\bullet)}((g^\bullet)^*\circ f^*(\CF),(g^\bullet)^*(\CF_1))\right),$$
which can be rewritten by adjunction as
\begin{multline*}
\on{Tot}\left(\CHom_{\Dmod(\CY_2)}(\CF,(f\circ g^\bullet )_*\circ (g^\bullet)^*(\CF_1))\right)\simeq  \\
\simeq \CHom_{\Dmod(\CY_2)}\left(\CF, \on{Tot}((f\circ g^\bullet)_* \circ (g^\bullet)^*(\CF_1))\right).
\end{multline*}

Hence, it suffices to show that the canonical map
$$f_*(\CF_1)\to \on{Tot}((f\circ g^\bullet)_* \circ (g^\bullet)^*(\CF_1))$$
is an isomorphism. However, this is given by \cite[Lemma 7.5.2]{DrGa1}. 

\qed

\ssec{Step 2}

In this subsection we will how that the functors 
\begin{equation} \label{e:hyperb restr other}
(\iota^-)^!\circ (\sfp^-)^* \text{ and } (\iota^+)^!\circ (\sfp^+)^*
\end{equation} 
take values in $\Dmod(\CY^0)\subset  \on{Pro}(\Dmod(\CY^0))$. Again, by symmetry, it is sufficient to do this
for $(\iota^-)^!\circ (\sfp^-)^*$.

\sssec{}

Note that the functor $(\iota^-)^!$ admits a (possibly non-continuous) \emph{right} adjoint, denoted $((\iota^-)^!)^R$.
Hence, our assertion is that the functor $(\sfp_-)_*\circ ((\iota^-)^!)^R$ admits a \emph{left} adjoint, which takes  
values in $\Dmod(\CY^0)$, rather than in $\on{Pro}(\Dmod(\CY^0))$. 

\medskip

Since $\CY$ was assumed QCA, the category 
$\Dmod(\CY)$ is compactly generated. Hence, it is enough to show that the value of $(\iota^-)^!\circ (\sfp^-)^*$ 
on $\CF\in \Dmod(\CY)^c$ belongs to $\Dmod(\CY^0)$.

\sssec{}

Recall from \cite[Sect. 4.2.4]{GR} that the category $\on{Pro}(\Dmod(\CY^0))$ inherits a t-structure from 
$\Dmod(\CY^0)$. We claim:

\begin{lem}   \label{l:pro bdd}
For $\CF\in \Dmod(\CY)^c$, the object $(\iota^-)^!\circ (\sfp^-)^*(\CF)\in \on{Pro}(\Dmod(\CY^0))$ is cohomologically
bounded, i.e., belongs to $\on{Pro}(\Dmod(\CY^0))^{\geq -n,\leq n}$
for some $n$. 
\end{lem}

\begin{proof}

The fact that  $(\iota^-)^!\circ (\sfp^-)^*(\CF)$ is bounded above follows by combining the following
three facts: (i) $\CF$ is bounded above; (ii) the functor $(\iota^-)^!$ has a bounded cohomological amplitude;
(iii) the functor $\sfp^-_*$, right adjoint to $(\sfp^-)^*$, is left t-exact.

\medskip

Let us show that $(\iota^-)^!\circ (\sfp^-)^*(\CF)$ is bounded below. Consider the object $\BD_{\CY}(\CF)\in \Dmod(\CY)^c$, 
where $\BD_\CY$ denotes the Verdier duality functor on $\CY$ (see \cite[Sect. 7.3.4 and Corollary 8.4.2]{DrGa1}). 

\medskip

The functor
$(\sfp^-)^!$ has a bounded cohomological amplitude, hence $(\sfp^-)^!\circ \BD_{\CY}(\CF)$ belongs to
$\Dmod(\CY^-)^{\leq m}$ for some $m$. Hence, we can write $(\sfp^-)^!\circ \BD_{\CY}(\CF)\in \Dmod(\CY^-)$ as a filtered colimit
$$\underset{i}{\underset{\longrightarrow}{lim}}\, \CF^-_i,$$
where $\CF^-_i\in \Dmod(\CY^-)^{\leq m}$. In this case  
$$(\sfp^-)^*(\CF)\simeq \underset{i}{\underset{\longleftarrow}{``lim"}}\, \BD_{\CY^-}(\CF^-_i).$$

However, the functor $\BD_{\CY^-}$ has a bounded cohomological amplitude, so all $\BD_{\CY^-}(\CF^-_i)$ belong to
$\Dmod(\CY^-)^{\geq -l}$ for some $l$. Now, 
$$(\iota^-)^!\circ (\sfp^-)^*(\CF)\simeq 
\underset{i}{\underset{\longleftarrow}{``lim"}}\, (\iota^-)^!\circ \BD_{\CY^-}(\CF^-_i),$$
and the boundedness below follows from the fact that the functor $(\iota^-)^!$ has a 
bounded cohomological amplitude.

\end{proof}

\sssec{}

Let us now observe that 
for any $\CF\in \Dmod(\CY)$, the object $(\psi^0)^!\circ (\iota^-)^!\circ (\sfp^-)^*(\CF)$
belongs to $\Dmod(Z^0)$. Indeed, we rewrite 
$$(\psi^0)^!\circ (\iota^-)^!\circ (\sfp^-)^*\simeq (\iota^0)^!\circ (\psi^-)^!\circ  (\sfp^-)^*,$$
and the latter is isomorphic, up to a cohomological shift to
$$(i^-)^!\circ (\psi^-)^*\circ  (\sfp^-)^*\simeq (i^-)^!\circ (p^-)^*\circ \psi^*.$$

However, $\psi^*(\CF)\in \Dmod(Z)^{\BG_m\on{-mon}}$, and hence
$$(i^-)^!\circ (p^-)^*\circ \psi^*(\CF)\in \Dmod(Z^0)$$ by 
\thmref{t:Braden original}.

\medskip

Since $\psi^0$ is smooth (and so $(\psi^0)^!$ is isomorphic to $(\psi^0)^*$ up to a cohomological shift),
we obtain that the object $(\psi^0)^*\circ (\iota^-)^!\circ (\sfp^-)^*(\CF)$
belongs to $\Dmod(Z^0)$

\sssec{}

Now, the fact that $(\iota^-)^!\circ (\sfp^-)^*(\CF)$ belongs to $\Dmod(\CY^0)$ for $\CF\in \Dmod(\CY)^c$ follows
from \lemref{l:pro bdd} and the next observation:

\begin{lem}
Let $g:\CZ\to \CW$ be a smooth surjective map between quasi-compact stacks. Let
$\wt\CF$ be a \emph{cohomologically bounded} object of $\on{Pro}(\Dmod(\CW))$. Then 
$\wt\CF\in \Dmod(\CW)$ if and only if $g^*(\wt\CF)$ belongs to $\Dmod(\CZ)\subset \on{Pro}(\Dmod(\CZ))$.
\end{lem}

\begin{proof}

The ``only if" direction is obvious. For the ``if" direction we proceed as follows. 

\medskip

Let $g^\bullet:\CZ^\bullet\to \CW$ be the \v{C}ech nerve of $g:\CZ\to \CW$. By assumption,
$(g^\bullet)^*(\wt\CF)$ gives rise to a well-defined object of
$$\on{Tot}(\Dmod(\CZ^\bullet)).$$

Applying the inverse of the equivaence
\begin{equation} \label{e:! Tot}
(g^\bullet)^*:\Dmod(\CY)\to \on{Tot}(\Dmod(\CZ^\bullet)),
\end{equation}
from $(g^\bullet)^*(\wt\CF)$ we obtain a well-defined object of $\Dmod(\CW)$ that we denote by $\CF$.  
We claim that $\wt\CF\simeq \CF$. 

\medskip

By construction, the object  $\CF\in \Dmod(\CY)$ is cohomologically bounded. 
Assume that $\wt\CF$ belongs to $\on{Pro}(\Dmod(\CW))^{\geq -n,\leq n}$. By 
\cite[Sect. 4.2.4]{GR}, we have a canonical equivalence 
$$\on{Pro}(\Dmod(\CW))^{\geq -n,\leq n}\simeq \on{Pro}\left(\Dmod(\CW)^{\geq -n,\leq n}\right).$$

\medskip

Hence, in order to establish an isomorphism $\CF\simeq \wt\CF$, it is sufficient to construct a 
functorial isomorphism
$$\CHom(\wt\CF,\CF')\simeq \CHom(\CF,\CF')$$
for $\CF'\in \Dmod(\CW)$ \emph{which is also cohomologically bounded}. 

\medskip

By construction,
$$\CHom(\wt\CF,\CF')\simeq \on{Tot}\left(\CHom\left((g^\bullet)^*(\wt\CF),(g^\bullet)^*(\CF')\right)\right).$$

Now, the fact that $\wt\CF$ is bounded above and that $\CF'$ is bounded below
implies that
$$\on{Tot}\left(\CHom\left((g^\bullet)^*(\wt\CF),(g^\bullet)^*(\CF')\right)\right)
\simeq
\CHom\left(\wt\CF, \on{Tot}\left((g^\bullet)_*\circ (g^\bullet)^*(\CF')\right)\right).$$

\medskip

Now,
$$\CF'\to \on{Tot}\left((g^\bullet)_*\circ (g^\bullet)^*(\CF')\right)$$
is an isomorphism since \eqref{e:! Tot} is an equivalence. 

\end{proof}

\ssec{Step 3}

Having established that both sides of \eqref{e:Braden trans stacks} belong to $\Dmod(\CY^0)$, we are finally
ready to show that the map \eqref{e:Braden trans stacks}, is an isomorphism. 

\medskip

Since the functor 
$$(\psi^0)^!:\Dmod(\CY^0)\to \Dmod(Z^0)$$
is conservative, it suffices to show that the natural transformation 
\begin{equation} \label{e:Braden trans stacks pullback}
(\psi^0)^*\circ (\iota^-)^*\circ (\sfp^-)^*\to (\psi^0)^*\circ  (\iota^+)^!\circ (\sfp^+)^*,
\end{equation}
induced by \eqref{e:Braden trans stacks}, is an isomorphism. 

\sssec{}

Consider the map $\beta^-:Z^-\to \CY^-\underset{\CY}\times Z$. 
Denote
$$\CK^-:=(\beta^-)^*(\omega_{\CY^-\underset{\CY}\times Z})\in \Dmod(Z^-).$$
The object $\CK^-$ is isomorphic to $\omega_{Z^-}$ up to a cohomological
shift, since both $Z^-$ and $\CY^-\underset{\CY}\times Z$ are smooth over $\CY^-$. 

\medskip

Note that for $\CF\in \Dmod(\CY)$ there exists a canonical isomorphism
\begin{equation} \label{e:* to ! -}
(\psi^-)^*\circ (\sfp^-)^!(\CF)\simeq \CK^-\sotimes (p^-)^!\circ \psi^*(\CF).
\end{equation}

\sssec{}

Consider now the map
$\beta^0:Z^0\to \CY^0\underset{\CY^+}\times Z^+$. 
Denote
$$\CK^0:=(\beta^+)^*(\omega_{\CY^0\underset{\CY^+}\times Z^+})\in \Dmod(Z^0).$$
The object $\CK^0$ is isomorphic to $\omega_{Z^0}$ up to a cohomological
shift, since both $Z^-$ and $\CY^0\underset{\CY^+}\times Z^+$ are smooth over $\CY^0$. 

\medskip

Note that for $\CF^+\in \Dmod(\CY^+)$ there exists a canonical isomorphism
\begin{equation} \label{e:* to ! 0}
(\psi^0)^*\circ (\iota^+)^!(\CF^+)\simeq \CK^0\sotimes (i^+)^!\circ (\psi^+)^*(\CF^+).
\end{equation}

\sssec{}

Note now that in the diagram 
$$
\xy
(0,0)*+{Z^- \underset{\CY^-\underset{\CY}\times Z}\times(\CY^0\underset{\CY^+}\times Z^+)}="X";
(40,0)*+{\CY^0\underset{\CY^+}\times Z^+}="Y";
(0,-30)*+{Z^-}="Z";
(40,-30)*+{\CY^-\underset{\CY}\times Z}="W";
(-40,30)*+{Z^0}="U";
{\ar@{->}_{\beta^-} "Z";"W"};
{\ar@{->}^{\sfp^-\times i^+} "Y";"W"};
{\ar@{->} "X";"Y"};
{\ar@{->} "X";"Z"};
{\ar@{->} "U";"X"};
{\ar@{->}_{i^-} "U";"Z"};
{\ar@{->}^{\beta^+} "U";"Y"};
\endxy
$$
the map
$$Z^0\to (\CY^0\underset{\CY^+}\times Z^+)\underset{\CY^-\underset{\CY}\times Z}\times Z^-$$
is an open embedding. Hence, we obtain a canonical isomorphism
\begin{equation} \label{e:K}
\CK^0\to (i^-)^!(\CK^-).
\end{equation}

\sssec{}

Now, diagram chase shows that for $\CF\in \Dmod(\CY)$ the following diagram 
commutes
$$
\CD
(\psi^0)^*\circ (\iota^-)^*\circ (\sfp^-)^!(\CF)  @>{\text{\eqref{e:Braden trans stacks pullback}}}>>  (\psi^0)^*\circ  (\iota^+)^!\circ (\sfp^+)^* (\CF) \\
@V{\sim}VV   @VV{\text{\eqref{e:* to ! 0}}}V   \\
(i^-)^*\circ (\psi^-)^*\circ (\sfp^-)^!(\CF) & &   \CK^0\sotimes (\iota^+)^! \circ (\psi^+)^* \circ (\sfp^+)^* (\CF)  \\
@V{\text{\eqref{e:* to ! -}}}VV    @VV{\sim}V   \\
(i^-)^*\left(\CK^-\sotimes (p^-)^!\circ \psi^*(\CF)\right) & & \CK^0\sotimes (\iota^+)^! \circ (p^+)^*\circ \psi^*(\CF)   \\
@V{\sim}VV         @VV{\text{\eqref{e:K}}}V  \\
(i^-)^!(\CK^-)\sotimes (i^-)^*\circ (p^-)^!\circ \psi^*(\CF)  
@>{\text{\eqref{e:Braden trans}}}>>  (i^-)^!(\CK^-)\sotimes (i^+)^!\circ (p^+)^*\circ \psi^*(\CF). 
\endCD
$$

Hence, we obtain that \eqref{e:Braden trans stacks pullback} is an isomorphism, as desired.

\section{The support of cuspidal objects of $\Dmod(\Bun_G)$} \label{s:support}
The goal of this Appendix is to prove \propref{p:cuspidality}, i.e., to 
construct an open substack $\CU\subset\Bun_G$ having quasi-compact intersection with each connected component of $\Bun_G$ and
such that the $*$-support and $!$-support of each object of 
$\Dmod(\Bun_G)_{\on{cusp}}$ are contained in $\CU$.

\ssec{Definition of $\CU$}   \label{ss:CU}

\sssec{}

Let $\Lambda_G$ denote the coweight lattice of $G$. 
Set $\Lambda_G^{\BQ}:=\Lambda_G\otimes\BQ$  and let $\Lambda_G^{+,\BQ}\subset \Lambda_G^{\BQ}$ be the dominant cone.

\medskip

We will use the Harder-Narasimhan-Shatz stratification 
of $\Bun_G$ (see \cite[Sect. 7.4]{DrGa2}).

\medskip

Its strata are labeled by elements $\lambda\in\Lambda_G^{+,\BQ}$ and denoted by $\Bun_G^{(\lambda )}$.
The set of those $\lambda$ for which $\Bun_G^{(\lambda )}\ne\emptyset$ is discrete in 
$\Lambda_G\otimes\BR$.

\sssec{}

Set
$$\CU:=\underset{\lambda\in \Sigma}\bigcup\, \Bun_G^{(\lambda)},
$$
where $\Sigma$ is the set of all $\lambda\in\Lambda_G^{+,\BQ}$ such that the image of $\lambda$ in
$\Lambda_{G_{\on{adj}}}^{\BQ}$ is $\le (2g-2)\rho_G$. Here $g$ is the genus of the curve $X$, 
$\rho_G\in\Lambda_{G_{\on{adj}}}^{\BQ}$ is the half-sum of positive coroots, and the ordering $\le$ is the usual one
(i.e., $\lambda_1\le\lambda_2$ if $\lambda_2 -\lambda_1$ is a linear combination of simple coroots with 
non-negative coefficients).

\begin{lem}
The union of strata $\CU$ is an open substack of $\Bun_G$. Its intersection with each connected component of 
$\Bun_G$ is quasi-compact.
\end{lem}

\begin{proof}
Use the standard properties of the Harder-Narasimhan-Shatz stratification (see \cite[Lemma 7.4.9]{DrGa2}) and the 
following obvious fact: for any $\lambda_0\in \Lambda_G^{\BQ}$,  the union of 
$\Bun_G^{(\lambda)}$ for $\lambda\in (\lambda_0+\Lambda_{[G,G]}^{\BQ})\cap\Lambda_G^{+,\BQ}$ 
is both open and closed.
\end{proof}

\ssec{Statement of the result}

\sssec{}

Let $\jmath:\CU\hookrightarrow\Bun_G$ denote the embedding. In \secref{ss:proof} we will prove the following

\begin{prop}   \label{p:cusp_support}
For any $\CF\in \Dmod(\Bun_G)_{\on{cusp}}\,$, the canonical maps
$$\jmath_!\circ \jmath^*(\CF)\to \CF\to \jmath_*\circ \jmath^*(\CF)$$
are isomorphisms. 
\end{prop}

The statement about  $\jmath_!\circ \jmath^*(\CF)$ will be understood as follows: the partially 
defined\footnote{In fact, the embedding $\jmath:\CU\hookrightarrow\Bun_G$ has the remarkable property that the functor $\jmath_!$ is defined 
on all of $\Dmod(\CU)$, see \cite[Theorem 9.1.2]{DrGa2}. We will not use this fact.} functor 
$\jmath_!$ is defined on $\jmath^*(\CF)$ and the map $\jmath_!\circ \jmath^*(\CF)\to \CF$ is an isomorphism.

\begin{rem}    \label{r:torus}
If $G$ is a torus then $\Sigma=\Lambda_G^{+,\BQ}=\Lambda_G^{\BQ}\,$, $\CU=\Bun_G$, and 
$\Dmod(\Bun_G)_{\on{cusp}}=\Dmod(\Bun_G)$. So the proposition holds tautologically.
\end{rem}

\begin{rem}   \label{r:g=0}
If $g=0$ and $G$ is not a torus then $\Sigma=\emptyset$ and therefore $\CU=\emptyset$. In this case the proposition says that
$\Dmod(\Bun_G)_{\on{cusp}}=0$. 
\end{rem}

\begin{rem}   \label{r:again g=0}
Let $g=0$ and $\CF\in \Dmod(\Bun_G)$. If the map 
$\CF\to \jmath_*\circ \jmath^*(\CF)$ is  an isomorphism then so is
the map  $\jmath_!\circ \jmath^*(\CF)\to \CF$. This follows from the fact that the open substack $\CU\subset\Bun_G$ 
equals either $\emptyset$ or $\Bun_G$ (see Remarks~\ref{r:torus}-\ref{r:g=0}).
\end{rem}

\ssec{The key lemma} 

\sssec{}

Recall that for every parabolic $P\subset G$ we denote by $$\sfp:\Bun_P\to\Bun_G \text{ and }
\sfq:\Bun_P\to\Bun_M$$ the corresponding morphisms (as usual, $M$ is the Levi quotient of $P$).

\begin{lem}    \label{l:2cuspidality}
Let $\Sigma$ 
be as in \secref{ss:CU}. Let $\lambda\in\Lambda_G^{+,\BQ}$, 
$\lambda\not\in \Sigma$. Then there exists a proper parabolic $P\subset G$ and a locally closed substack 
$V\subset\Bun_M$ such that:

\begin{enumerate}
\item The fiber of $\sfq:\Bun_P\to\Bun_M$ over any geometric point of $V$ has only one isomorphism class of geometric points;

\item $\sfq^{-1}(V)\subset\sfp^{-1}(\Bun_G^{(\lambda)})$;

\item The morphism $\sfp|_{\sfq^{-1}(V)}:\sfq^{-1}(V)\to\Bun_G^{(\lambda)}$ is
an isomorphism if $g>0$ and is surjective if $g=0$.
\end{enumerate}
\end{lem}

\begin{rem}   \label{r:fiber of q}
Let $N(P)$ denote the unipotent radical of $P$.
For any $\CP\in \Bun_ {M}$ let $N(P)_{\CP}$ denote the group-scheme 
$N(P)_{\CP}$ over $X$ (here we regard $ {M}$
as acting on $N(P)$ via the embedding $ M\hookrightarrow P$).
Then the fiber of $\sfq$ over $\CP\in \Bun_ {M}$
is the stack of $N(P)_{\CP}$-torsors on $X$. 
Condition (1) in \lemref{l:2cuspidality} says that if $\CP\in V$ then all such torsors are trivial. 
\end{rem}

\begin{proof}

Let us first assume that $g=0$, i.e., $X=\BP^1$. We can also assume that $G$ is not a torus 
(otherwise $\Sigma=\Lambda_G^{+,\BQ}$ and there is nothing to prove).
Take $P=B$ to be a Borel subgroup (so $M=T$ is the Cartan). 
Define $V\subset\Bun_T$ to be the substack of $T$-bundles of degree $\lambda$. Property (1) follows from the fact that 
$H^1(\BP^1,\CO (n))=0$ for $n\ge 0$. Property (3) holds because any $G$-bundle on $\BP^1$ admits a reduction to a maximal torus 
of $G$ (see \cite{BN} and references therein). Property (2) is clear.

\medskip

Let us now assume that $X$ has genus $g>0$. 
The simple roots of $G$ will be denoted by $\check\alpha_i$.
Define $P$ and $M$ by the following condition: the simple roots of $M$ are those $\check\alpha_i$'s for which
$\langle \lambda\, , \check\alpha_i \rangle\le 2g-2$. We have $P\ne G$ (indeed, let 
$\bar\lambda\in\Lambda_{G_{\on{adj}}}^{+,\BQ}$ be the image of $\lambda$, then the equality $P=G$ would imply that $(2g-2)\rho_G-\bar\lambda$ 
is dominant and therefore $\bar\lambda\le (2g-2)\rho_G\,$, contrary to the assumption that $\lambda\not\in \Sigma$).

\medskip

Note that $\Lambda_G^{+,\BQ}\subset\Lambda_M^{+,\BQ}$, so we can consider the Harder-Narasimhan-Shatz stratum
$\Bun_M^{(\lambda)}$. Set $V:=\Bun_M^{(\lambda)}$. Then $V$ has the required properties. Properties~(2)-(3) follow from the definition of 
the Harder-Narasimhan-Shatz strata (see \cite[Theorem 7.4.3]{DrGa2})
and the assumption $2g-2\ge 0$.   Remark~\ref{r:fiber of q} shows that to prove property~(1) it suffices to check
the equality 
\[
H^1(X, \fn_{\CP})=0 \quad \mbox{for all }\CP\in\Bun_M^{(\lambda)},
\]
where $\fn := \Lie (N)$. This equality holds
by \cite[Proposition~10.1.3]{DrGa2} and the characteristic~0 assumption on $k$.
\end{proof}

\sssec{}

In the situation of Lemma~\ref{l:2cuspidality} consider the corresponding morphism
$$\sfq:\sfq^{-1}(V)\to V.$$ 

\begin{lem}    \label{l:3cuspidality}  \hfill 

\smallskip
\begin{enumerate}
\item[(i)] Each fiber of the morphism $\sfq:\sfq^{-1}(V)\to V$ is a classifying
space of a unipotent group.

\medskip

\item[(ii)] The functors
$$\sfq_*:\Dmod(\sfq^{-1}(V))\to \Dmod(V) \text{ and } \sfq_!:\Dmod(\sfq^{-1}(V))\to \Dmod(V)$$ 
are equivalences (that differ by a cohomological shift), with the inverses provided by
$$\sfq^*:\Dmod(V)\to \Dmod(\sfq^{-1}(V)) \text{ and } \sfq^!:\Dmod(V)\to \Dmod(\sfq^{-1}(V)),$$
respectively.
\end{enumerate}
\end{lem}

\begin{proof}
Statement (i) follows from the interpretation of the fibers of $\sfq$ given in Remark~\ref{r:fiber of q}.
Statement (ii) follows from (i) because $\sfq$ is smooth.
\end{proof}

\ssec{Proof of \propref{p:cusp_support}} \label{ss:proof}

\sssec{}

Let $\Sigma$ be as in \secref{ss:CU}. Let $\lambda\in\Lambda_G^{+,\BQ}$, 
$\lambda\not\in \Sigma$. Let $\imath_\lambda:\Bun_G^{(\lambda )}\hookrightarrow\Bun_G$ be the corresponding 
locally closed embedding.
We need to show that for any cuspidal $\CF\in \Dmod(\Bun_G)$  one has
$$\imath^!_\lambda(\CF)=0,\quad  \imath^*_\lambda(\CF)=0.$$

\medskip

Let $P$ and $V$ be as in \lemref{l:2cuspidality}. Consider the corresponding morphisms
$$\sfp|_{\sfq^{-1}(V)}:\sfq^{-1}(V)\to \Bun_G^{(\lambda)} \text{ and } \sfq|_{\sfq^{-1}(V)} : \sfq^{-1}(V)\to V.$$

\sssec{}

Let us first prove that $\imath^!_\lambda(\CF)=0$. By cuspidality, $\sfq_*\left(\sfp^!(\CF)\right)=0$, so 
$$(\sfq|_{\sfq^{-1}(V)})_*\circ (\sfp|_{\sfq^{-1}(V)})^! \circ\imath^!_\lambda( \CF)=0.$$
By Lemma~\ref{l:3cuspidality}(ii), this means that $(\sfp|_{\sfq^{-1}(V)})^!\circ\imath^!_\lambda( \CF)=0$. Since 
$\sfp|_{\sfq^{-1}(V)}: \sfq^{-1}(V)\to \Bun_G^{(\lambda)}$ is surjective, the functor $(\sfp|_{\sfq^{-1}(V)})^!$ is 
conservative by \cite[Lemma 5.1.6]{DrGa1}. So $\imath^!_\lambda(\CF)=0$.

\sssec{}

Let us now prove that $\imath^*_\lambda(\CF)=0$. By Remark~\ref{r:again g=0}, we can assume that $g>0$,
so the map $\sfp|_{\sfq^{-1}(V)}: \sfq^{-1}(V)\to \Bun_G^{(\lambda)}$ is an isomorphism (see Lemma~\ref{l:2cuspidality}).

\medskip

We have $\on{CT}^\lambda_!(\CF)=0$, where $\on{CT}^\lambda_!$ is the constant term functor with respect to $P$. 
So the adjunction from Theorem~\ref{t:main} implies that $\Hom (\CF ,\Eis^\lambda_*(\CE ))=0$ for all 
$\CE\in\Dmod(\Bun_M^{\lambda})$. 

\medskip

Taking $\CE$ to be the $*$-pushforward of an object 
$\CE'\in\Dmod (V)$ we have
$$\Eis_*(\CE )\simeq (\imath_\lambda)_*\circ (\sfp|_{\sfq^{-1}(V)})_* \circ (\sfq|_{\sfq^{-1}(V)})^! (\CE'),$$
and hence 
\[
\Hom (\CF, (\imath_\lambda)_*\circ  (\sfp|_{\sfq^{-1}(V)})_* \circ (\sfq|_{\sfq^{-1}(V)})^!  (\CE'))=0.
\]
Since  $(\sfp|_{\sfq^{-1}(V)})_*$ and $(\sfq|_{\sfq^{-1}(V)})^!$ are equivalences, we obtain
$$\Hom(\CF,(\imath_\lambda)_*(\CF'))=0, \quad \forall\, \CF'\in \Dmod(\Bun_G^{(\lambda)}),$$
i.e., $\imath^*_\lambda(\CF)=0$. \qed

\section{Proof of \propref{p:quasiaffine}} \label{s:quasiaffine}
In this Appendix (just as in \propref{p:quasiaffine} itself) the characteristic of the ground field $k$ can be arbitrary.

\ssec{Reduction of \propref{p:quasiaffine} to \lemref{l:existence of Y}}

\sssec{}
We have to prove that the quotient $(\BA^1\times G\times G)/\wt{G}$ is a quasi-affine scheme.
We know that this quotient exists as an algebraic space (because $\wt{G}$ is flat over $\BA^1$). 
This space is separated because $\wt{G}$ is closed in  $\BA^1\times G\times G$. 

\sssec{}

Recall that a separated 
quasi-finite morphism is quasi-affine (see \cite[Theorem A.2]{LM}  or  \cite[Ch. II, Theorem 6.15]{Kn}). 
So it remains to construct a quasi-finite morphism from $(\BA^1\times G\times G)/\wt{G}$ to an affine scheme $Y$. 

\medskip

Thus it remains to prove the following lemma:

\begin{lem}     \label{l:existence of Y}
There exists an affine scheme $Y$ equipped with an action of $G\times G$ and a morphism $f:\BA^1\to Y$ 
such that for any $t\in\BA^1 (k)$ the stabilizer of $f(t)$ contains $\wt{G}_t$ as a subgroup of finite index.
\end{lem}

\ssec{Plan of the proof of \lemref{l:existence of Y}}

\sssec{}

For each dominant weight $\check\lambda$ of $G$ we will construct a finite-dimensional 
$(G\times G)$-module $R_{\check\lambda}$, equipped with a polynomial map 
$f_{\check\lambda}:\BA^1\to R_{\check\lambda}\,$.

\medskip

Let $\Stab_t^{\check\lambda}\subset G\times G$ denote the stabilizer of $f_{\check\lambda}(t)$.
We will show that
\begin{equation}   \label{e:Stab1}
\Big(\bigcap_{\check\lambda}\, \Stab_t^{\check\lambda}\Big)_{\red}=\wt{G}_t\, .
\end{equation}

%

\sssec{}

This will immediately imply \lemref{l:existence of Y}: indeed, one can take
\[
Y=\prod_{\check\lambda} R_{\check\lambda}, \quad\quad f=(f_{\check\lambda}),
\]
where $\check\lambda$ runs through the set of all dominant weights of $G$ 
(or if you prefer, a sufficiently large finite collection of weights).

\medskip

Recall that $\wt{G}_t$ is reduced by \propref{p:smooth}(i), so \eqref{e:Stab1} is equivalent to the inclusions
\begin{equation}    \label{e:Stab2}
 \Big(\bigcap_{\check\lambda}\, \Stab_t^{\check\lambda}\Big)_{\red}\subset \wt{G}_t \subset 
 \bigcap_{\check\lambda}\Stab_t^{\check\lambda}.
\end{equation}

Thus it remains to construct $R_{\check\lambda}$ and $f_{\check\lambda}\,$,  and to prove \eqref{e:Stab2}.

\ssec{Construction of $R_{\check\lambda}$ and $f_{\check\lambda}(t)$}
\sssec{Definition of $R_{\check\lambda}$}    \label{sss:Llambda}
Let $V_{\check\lambda}$ denote the irreducible $G$-module with highest weight $\check\lambda$.
Set $R_{\check\lambda}:=\End_k(V_{\check\lambda})$. We have a homomorphism of monoids
\[
\rho_{\check\lambda}:G\to \End_k(V_{\check\lambda}).
\]
The action of $G\times G$ on $R_{\check\lambda}:=\End_k(V_{\check\lambda})$ is defined as follows:  to 
$(g_1,g_2)\in G\times G$ we associate the operator
\[
\End_k(V_{\check\lambda})\to \End_k(V_{\check\lambda})\, ,\quad\quad A\mapsto \rho_{\check\lambda}(g_1)\cdot A
\cdot\rho_{\check\lambda}(g_2)^{-1}.
\]


\sssec{Definition of $f_{\check\lambda}\,$}   \label{sss:flambda}
Recall that to define $\wt{G}$ we fixed in \secref{e:co-character} a co-character
$$\gamma:\BG_m\to M$$ mapping to the center of $M$, which is dominant and regular with respect to $P$. 
Let us also fix a maximal torus $T\subset G$ contained in $M$ and a Borel $B\subset G$ such that
$T\subset B\subset P$.

\medskip

Now, for $t\in\BG_m$ we set
\begin{equation}   \label{e:flambdat}
f_{\check\lambda}(t):=\check\lambda (\gamma (t))\cdot\rho_{\check\lambda}(\gamma (t)^{-1}).
\end{equation}

Thus we get a homomorphism of algebraic groups $\BG_m\to\Aut_k(V_{\check\lambda})\,$. It extends to
a homomorphism of algebraic monoids
\[
f_{\check\lambda}:\BA^1\to\End_k(V_{\check\lambda})
\]
because $\gamma$ is dominant and all the weights of $T$ in $V_{\check\lambda}$ are $\le\check\lambda$.

%

\sssec{The operator $f_{\check\lambda}(0)$}   \label{sss:f(0)}
Set $S_{\check\lambda}:=f_{\check\lambda}(0)$. Then $S_{\check\lambda}$ is a $T$-equivariant idempotent operator 
$V_{\check\lambda}\to V_{\check\lambda}\,$.
The subspaces
\[
V'_{\check\lambda}:=\on{Im}(S_{\check\lambda})\, , \quad V''_{\check\lambda}:=\Ker(S_{\check\lambda})
\]
can be described in terms of the decomposition
\[
V_{\check\lambda}=\bigoplus_{\check\mu}V_{\check\lambda ,\check\mu}
\]
with respect to the weights of $T$. Namely,
\[
V'_{\check\lambda}=\bigoplus_{\check\mu\in A} V_{\check\lambda ,\check\mu}\, , \quad V''_{\check\lambda}=
\bigoplus_{\check\mu\notin A} V_{\check\lambda ,\check\mu}\, ,
\]
where $A$ is the set of weights $\check\mu$ such that $\check\lambda -\check\mu$ is in the root lattice of $M$ (which is a sublattice in that of $G$).

\ssec{The stabilizers $\Stab_t^{\check\lambda}$}

\sssec{}
Let $\Stab_t^{\check\lambda}\subset G\times G$ denote the stabilizer of $f_{\check\lambda}(t)\in\End_k(V_{\check\lambda})$ 
with respect to the action of 
$G\times G$ defined in \secref{sss:Llambda}. Explicitly,
\begin{equation}   \label{e:Stab_tlambda}
\Stab_t^{\check\lambda}=\{ (g_1,g_2)\in G\times G\, |\,\rho_{\check\lambda}(g_1)\cdot f_{\check\lambda}(t)=f_{\check\lambda}(t)\cdot
\rho_{\check\lambda}(g_2)\}.
\end{equation}
Our goal is to prove the inclusions \eqref{e:Stab2}.

\sssec{The case $t\ne 0$}
Let $H_{\check\lambda}\subset G$ denote the kernel of $\rho_{\check\lambda}:G\to \Aut_k(V_{\check\lambda})$.

\begin{lem}
If $t\ne 0$ then 
$\Stab_t^{\check\lambda}=\wt{G}_t\cdot (H_{\check\lambda}\times H_{\check\lambda})$ and
$$\underset{\check\lambda}\bigcap\, \Stab_t^{\check\lambda}=\wt{G}_t\cdot 
\left(\underset{\check\lambda}\bigcap\, (H_{\check\lambda}\times H_{\check\lambda})\right).$$
\end{lem}

\begin{proof}
Recall that $\wt{G}_t=\{ (g_1,g_2)\in G\times G\, |\, g_2=\gamma (t)\cdot g_1\cdot\gamma (t)^{-1}\}$. 
On the other hand, by \eqref{e:flambdat} and \eqref{e:Stab_tlambda}, we have 
\[
\Stab_t^{\check\lambda}=\{ (g_1,g_2)\in G\times G\, |\,\rho_{\check\lambda}(g_2)=
\rho_{\check\lambda}(\gamma (t))\cdot\rho_{\check\lambda}(g_1)\cdot\rho_{\check\lambda}(\gamma (t)^{-1})\}.
\]
The lemma follows.
\end{proof}

The inclusions \eqref{e:Stab2} for $t\ne 0$ follow from the above lemma and the next one.
\begin{lem}   \label{l:almost trivial intersection}
The group $\big(\bigcap\limits_{\check\lambda}\, H_{\check\lambda}\big)_{\red}$ is trivial.
\end{lem}

\begin{proof}
$\big(\bigcap\limits_{\check\lambda}\, H_{\check\lambda}\big)_{\red}$ is a reduced normal subgroup of $G$, which has 
trivial intersection with $T$. The lemma follows.
\end{proof}

\begin{rem}  \label{r:warning}
If $k$ has characteristic 2 and $G=PGL(2)$ then the group scheme $\bigcap\limits_{\check\lambda}\, H_{\check\lambda}$
has order 4. In fact, it equals the kernel of the homomorphism $g:PGL(2)\to SL(2)$ such that the composition
$SL(2)\to PGL(2)\overset{g}\longrightarrow SL(2)$ is the Frobenius endomorphism of $SL(2)$.
\end{rem}

\sssec{The case $t=0$}
Now let us study  $\Stab_0^{\check\lambda}$ (i.e., the stabilizer $\Stab_t^{\check\lambda}$ for $t=0$). By \eqref{e:Stab_tlambda}, we have
\begin{equation}    \label{e:Stab_0lambda}
\Stab_0^{\check\lambda}=\{ (g_1,g_2)\in G\times G\, |\,\rho_{\check\lambda}(g_1)\cdot S_{\check\lambda}=S_{\check\lambda}\cdot
\rho_{\check\lambda}(g_2)\},
\end{equation}
where $S_{\check\lambda}:=f_{\check\lambda}(0)$.

\medskip

Note that the subspace $V'_{\check\lambda}\subset V_{\check\lambda}$ is $M$-stable. So we have a homomorphism
$M\to\Aut_k(V'_{\check\lambda})$ and therefore homomorphisms
\[
P\twoheadrightarrow M\to \Aut_k(V'_{\check\lambda}), \quad P^-\twoheadrightarrow M\to \Aut_k(V'_{\check\lambda})\, .
\]

\begin{lem}   \label{l:theintersection}
The intersection $\Stab_0^{\check\lambda}\cap (P\times P^-)$ is equal to the fiber product of $P$ and $P^-$ over
$\Aut_k(V'_{\check\lambda})$.
\end{lem}

\begin{proof}
Follows from \eqref{e:Stab_0lambda} and the description of $S_{\check\lambda}$ given in \secref{sss:f(0)}.
\end{proof}

Recall that $\wt{G}_0=P\underset{M}\times P^-$.
The above lemma implies that 
\[
 \Big(\bigcap_{\check\lambda}\Stab_0^{\check\lambda} \Big)\cap (P\times P^-)=P\underset{M/H}\times P^-\, ,
 \]
where 
\[
H:= \bigcap_{\check\lambda}\Ker (M\to \Aut_k(V'_{\check\lambda})).
\]
Similarly to \lemref{l:almost trivial intersection}, the group $H_{\red}$ is trivial. Therefore, 
to prove \eqref{e:Stab2} for $t=0$ it remains to show that
\[
 \Big(\bigcap_{\check\lambda}\, \Stab_0^{\check\lambda}\Big)_{\red}\subset  P\times P^-\, .
\]
This follows from the next lemma.

\begin{lem}   \label{l:next lemma}
If $\check\lambda$ is strictly dominant then $(\Stab_0^{\check\lambda})_{\red}\subset P\times P^-$.
\end{lem}

\begin{proof}
Set $V''_{\check\lambda}:=\Ker(S_{\check\lambda})\,$. It is clear that  $\Stab_0^{\check\lambda}\subset K\times K^-$, where 
$$K:=\{g\in G\,|\,\rho_{\check\lambda}(g)(V'_{\check\lambda})=V'_{\check\lambda}\}, \quad\quad
K^-:=\{g\in G\,|\,\rho_{\check\lambda}(g)(V''_{\check\lambda})=V''_{\check\lambda}\}.$$ 
So it remains to show that $K_{\red}=P$, $K^-_{\red}=P^-$. Let us prove that $K_{\red}=P$ (the proof of the other equality is similar).

\medskip

Clearly $P\subset K$, so $K_{\red}$ is a parabolic containing $P$. Thus if $K_{\red}\ne P$ then $K_{\red}$ 
contains the subgroup $SL(2)$ corresponding to some simple root $\check\alpha_i$ of $G$ which is not a root of 
$M$. Then the weight $\check\lambda -\langle\check\lambda\, ,\alpha_i\rangle\cdot\check\alpha_i$ should occur 
in the weight decomposition of $V'_{\check\lambda}$ with respect to $T$. By the definition of $V'_{\check\lambda}\,$, 
this means that $\langle\check\lambda\, ,\alpha_i\rangle=0$, which is contrary to the assumption 
that $\check\lambda$ is strictly dominant.
\end{proof}

Thus we have proved \lemref{l:existence of Y} and \propref{p:quasiaffine}.

\begin{rem}   \label{r:C.4.9}
It is not hard to show  that if  char$\,k\ne 2$ or if the group $[G,G]$ is simply connected then the equality \eqref{e:Stab1}
can be replaced by the stronger equality
\begin{equation}   \label{e:stronger equality}
\bigcap_{\check\lambda}\, \Stab_t^{\check\lambda}=\wt{G}_t\, .
\end{equation}
On the other hand, \eqref{e:stronger equality} does not hold if char$\,k = 2$ and $G=PGL(2)$ (this follows from Remark~\ref{r:warning}).
\end{rem}

\begin{rem}
In \secref{sss:Llambda} we defined the $(G\times G)$-module $R_{\check\lambda}$ to be equal to
$\End_k(V_{\check\lambda})\,$, where  $V_{\check\lambda}$ is the \emph{irreducible} $G$-module with highest weight 
$\check\lambda$. The reader may prefer to define $R_{\check\lambda}$ by 
\begin{equation}   \label{e:Weyl modules}
R_{\check\lambda}:=\End_k(\Delta_{\check\lambda})\times\End_k (\nabla_{\check\lambda})\, ,
\end{equation}
where $\Delta_{\check\lambda}$ (resp.$ \nabla_{\check\lambda}$) is obtained from the 1-dimensional $B$-module corresponding to 
$\check\lambda$ by applying the functor
\[
\{B\mbox{-modules}\}\to \{G\mbox{-modules}\}
\]
left adjoint (resp. right adjoint) to the restriction functor  \{$G$-modules\} $\to \{B$-modules\}. It is not hard to show that if one defines $R_{\check\lambda}$ by \eqref{e:Weyl modules} then the equality \eqref{e:stronger equality} holds for \emph{any} reductive $G$ and \emph{any} characteristic of $k$.
\end{rem}

\section{Relation to Vinberg's semi-group}  \label{s:Vinberg}

In \secref{ss:Relation} we will explain how the group-scheme $\wt{G}$ of \secref{ss:Interpolating} can be 
obtained from the Vinberg semi-group corresponding to $G$ (a.k.a. enveloping semigroup of $G$). 

\medskip

Before doing this, we give a brief review of the standard material on the Vinberg semi-group\footnote{It was defined 
and studied by E.~Vinberg \cite{Vi} in characteristic 0 and then by A.Rittatore in arbit\-ra\-ry characteristic, see 
\cite{Ri1,Ri2,Ri3,Ri4}.}, which is contained in \cite{Vi,Re,Ri1,Ri3,Bo}.
For the general theory of reductive monoids, see  \cite{Pu, Re, Vi} and \cite[Ch.~5, Sect.~27]{Ti}.

\medskip

In this Appendix (as in \secref{ss:Interpolating}) the characteristic of the ground field $k$ can be arbitrary.

\ssec{The group $G_{\on{enh}}\,$}

\sssec{}

Let $G$ be a reductive group. Let $T$ denote its \emph{abstract} Cartan. 
Let $\check\Lambda_G=\check\Lambda_T$ denote the weight lattice of $G$ (=the lattice of characters of $T$).
It is equipped with the usual partial order relation, denoted by $\le\,$.

\sssec{}
Let $Z(G)$ denote the center of $G$. Consider the group
$$G_{\on{enh}}:=(G\times T)/Z(G),$$
where $Z(G)$ maps to $G\times T$ anti-diagonally. Note that $Z(G_{\on{enh}})=T$.

\medskip

Set $T_{\on{adj}}:=T/Z(G)$. We have a canonical homomorphism of algebraic groups
$$G_{\on{enh}}:=(G\times T)/Z(G)\to T/Z(G)=T_{\on{adj}}\, .$$

\ssec{Vinberg's semi-group: definition}
\sssec{}
The Vinberg semi-group of $G$, denoted by $\ol{G_{\on{enh}}}\,$, is an affine algebraic monoid containing 
$G_{\on{enh}}$ as its group of invertible elements. Such a monoid is uniquely determined by the full monoidal 
subcategory 
$$\Rep(\ol{G_{\on{enh}}})\subset \Rep(G_{\on{enh}})$$
consisting of those representations for which the action of 
$G_{\on{enh}}$ extends to that of $\ol{G_{\on{enh}}}\,$. 
The coordinate ring of $\ol{G_{\on{enh}}}$ is reconstructed from 
$\Rep(\ol{G_{\on{enh}}})$ as the set of matrix elements of all representations from 
$\Rep(\ol{G_{\on{enh}}})$. 

\medskip

The category $\Rep(\ol{G_{\on{enh}}})$ is the following. Note that an object $V\in \Rep(G_{\on{enh}})$
canonicaly splits as 
$$V\simeq \underset{\check\mu\in \check\Lambda_T}\oplus\, V_{\check\mu}$$
according to the action of $T=Z(G_{\on{enh}})$. Let us also note that by the definition of $G_{\on{enh}}\,$, 
the central character of the action of $G\hookrightarrow G_{\on{enh}}$ on each $V_{\check\mu}$ equals 
$\check\mu|_{Z(G)}\,$. Now $\Rep(\ol{G_{\on{enh}}})$ is defined to consist of all $V\in \Rep(G_{\on{enh}})$ 
with the following property: for every $\check\mu$, the \emph{weights}
of $V_{\check\mu}$, regarded as a $G$-representation, are $\leq \check\mu\,$.

\medskip

Note that the set of weights of a given $G$-representation is invariantly defined, i.e., 
independent of the choice of the Cartan or Borel subgroups. 

\sssec{} L.~Renner proved that the affine scheme $\ol{G_{\on{enh}}}$ is of finite type and normal 
(this follows from \cite[Thm 5.4 (a)]{Re}).
In the characteristic zero case this is also proved in \cite{Vi}.

\sssec{} The Tannakian formalism provides the following description of the $S$-points of $\ol{G_{\on{enh}}}$
for any $k$-scheme $S$: an element of the monoid $\Maps(S,\ol{G_{\on{enh}}})$ is a collection of 
assigments
\begin{equation} \label{e:pts of V}
V\in \Rep(\ol{G_{\on{enh}}})\quad \rightsquigarrow \quad g_V\in \End_{\CO_S}(V\otimes \CO_S),
\end{equation}
compatible with maps $V_1\to V_2$ in $\Rep(\ol{G_{\on{enh}}})$ and such that
$$g_{V_1\otimes V_2}=g_{V_1}\otimes g_{V_2}.$$

\ssec{The homomorphism $\bar\pi:\ol{G_{\on{enh}}}\to \ol{T_{\on{adj}}}\, $}
\sssec{The algebraic monoid  $\ol{T_{\on{adj}}}\,$}

As before, set $T_{\on{adj}}:=T/Z(G)$.  Let $I$ denote the set of vertices of the Dynkin diagram of $G$. 
The simple root corresponding to $i\in I$ is a homomorphism  $T_{\on{adj}}\to\BG_m\,$. These homorphisms
define a canonical isomorphism
$$T_{\on{adj}}\iso \underset{i\in I}\Pi\, \BG_m\, .$$
The inverse isomorphism is given by the fundamental coweights $\BG_m\to T_{\on{adj}}\,$.

\medskip

Let $\ol{T_{\on{adj}}}$ denote the algebraic monoid $\underset{i\in I}\Pi\, \BA^1$. 
The open embedding $\BG_m\hookrightarrow \BA^1$ induces a canonical open embedding
$$T_{\on{adj}}\hookrightarrow \ol{T_{\on{adj}}}\, .$$

\sssec{The homomorphism $\bar\pi$}   \label{sss:barpi}
Define a homomorphism of algebraic groups $\pi :G_{\on{enh}}\to T_{\on{adj}}$ to be the composition
$$G_{\on{enh}}:=(G\times T)/Z(G)\to T/Z(G)=T_{\on{adj}}\, .$$
The functor $\Rep(T_{\on{adj}})\to\Rep(G_{\on{enh}})$ corresponding to $\pi$ maps $\Rep(\ol{T_{\on{adj}}})$
to   $\Rep(\ol{G_{\on{enh}}})$. So $\pi$ extends to a homomorphism
of algebraic monoids
\begin{equation}   \label{e:pi}
\bar\pi:\ol{G_{\on{enh}}}\to \ol{T_{\on{adj}}}\, .
\end{equation}
Note that 
$$\bar\pi^{-1}(T_{\on{adj}})=G_{\on{enh}}\, .$$
This follows from the fact that for every $V\in\Rep(G_{\on{enh}})$ there exists a character
$\chi :T_{\on{adj}}\to\BG_m$ such that $V\otimes (\chi\circ\pi)\in \Rep(\ol{G_{\on{enh}}})$ 
(here $\chi\circ\pi$ is considered as a 1-dimensional representation of $G_{\on{enh}}$).

\ssec{The non-degenerate locus of Vinberg's semi-group}
In this subsection we recall some facts whose proofs can be found in \cite[Sect.~8]{Vi} in the 
characteristic zero case and in \cite{Ri1,Ri3} in general (there are also some hints in \cite[Sect.~1.1]{Bo}).

\sssec{}
Recall that $G_{\on{enh}}:=(G\times T)/Z(G)$, so a weight of $G_{\on{enh}}$ is a pair 
$$(\check\lambda_1 ,\check\lambda_2)\in \check\Lambda_G\times \check\Lambda_T$$
such that $\check\lambda_1-\check\lambda_2$ belongs to the root lattice. For each dominant  
$\check\lambda\in\Lambda_G$ let $V_{\check\lambda,\check\lambda}\in\Rep (\ol{G_{\on{enh}}})$ denote the irreducible 
representation with highest weight $(\check\lambda ,\check\lambda)$. 

\medskip

Given a point $g\in G_{\on{enh}}(k)$, let $A_g$ denote the set of dominant weights $\check\lambda\in\Lambda_G$
such that the action of $g$ on $V_{\check\lambda,\check\lambda}$ is \emph{non-zero}.

\medskip

The following fact is well-known:

\begin{lem}
The following properties of $g\in G_{\on{enh}}$ are equivalent:

\smallskip

\noindent{\em(i)} $A_g$ is the set of all dominant weights of $G$;

\smallskip

\noindent{\em(ii)} $A_g$ is not contained in any wall of the dominant cone of $\check\Lambda_G\otimes\BQ$.

\end{lem}

It is clear that the set of elements of $G_{\on{enh}}(k)$ with property (ii) corresponds to an open subscheme
of $\ol{G_{\on{enh}}}$; we denote it by $\overset{\circ}{\ol{G_{\on{enh}}}}$. 

\begin{defn}
The subscheme $\overset{\circ}{\ol{G_{\on{enh}}}}\subset \ol{G_{\on{enh}}}$ is called the  \emph{non-dege\-nerate locus}
of $\ol{G_{\on{enh}}}$.
\end{defn}

\sssec{}   \label{sss:2sided translations}
An element $(g_1,g_2)\in G_{\on{enh}}\times G_{\on{enh}}$ defines a regular map $\ol{G_{\on{enh}}}\to\ol{G_{\on{enh}}}$ given by 
$$x\mapsto g_1\cdot x\cdot g_2^{-1}.$$ Thus $G_{\on{enh}}\times G_{\on{enh}}$ acts on the variety  $\ol{G_{\on{enh}}}\,$. 
The open subscheme $\overset{\circ}{\ol{G_{\on{enh}}}}\subset\ol{G_{\on{enh}}}$ is clearly preserved under this action.

\sssec{}   \label{sss:transitive}

It is known that the scheme $\overset{\circ}{\ol{G_{\on{enh}}}}$ is smooth and moreover, the restriction of the morphism
$\bar\pi:\ol{G_{\on{enh}}}\to \ol{T_{\on{adj}}}$ to $\overset{\circ}{\ol{G_{\on{enh}}}}$ is smooth.

\medskip

It is also known that the action of $G\times G$
on $\overset{\circ}{\ol{G_{\on{enh}}}}$ is \emph{transtivite relative to} $\ol{T_{\on{adj}}}\,$, in the sense that the 
action map
\begin{equation}   \label{e:action map}
(G\times G)\times \overset{\circ}{\ol{G_{\on{enh}}}}\to \overset{\circ}{\ol{G_{\on{enh}}}}\underset{\ol{T_{\on{adj}}}}\times \overset{\circ}{\ol{G_{\on{enh}}}}
\end{equation}
is surjective and flat. Moreover, it is known that the morphism \eqref{e:action map} is smooth.

\sssec{}   \label{sss:Stab}
Consider $(G\times G)\times \overset{\circ}{\ol{G_{\on{enh}}}}$ as a group-scheme over 
$\overset{\circ}{\ol{G_{\on{enh}}}}$. Define a closed group-subscheme
$\on{Stab}_{G\times G}\subset (G\times G)\times \overset{\circ}{\ol{G_{\on{enh}}}}$ to be the pre-image of the diagonal
$\overset{\circ}{\ol{G_{\on{enh}}}}\mono\overset{\circ}{\ol{G_{\on{enh}}}}\underset{\ol{T_{\on{adj}}}}\times 
\overset{\circ}{\ol{G_{\on{enh}}}}$ with respect to the map~\eqref{e:action map}. The fiber of $\on{Stab}_{G\times G}$ over 
$x\in \overset{\circ}{\ol{G_{\on{enh}}}} (k)$ is the stabilizer of $x$ in $G\times G$; it will be denoted by 
$\on{Stab}_{G\times G}(x)\,$.

It follows from \secref{sss:transitive} that the scheme $\on{Stab}_{G\times G}$ is \emph{smooth} over 
$\overset{\circ}{\ol{G_{\on{enh}}}}\,$.

\begin{rem}
The torus $T=Z(G_{\on{enh}})$ acts on $\ol{G_{\on{enh}}}$ by translations.
It is easy to show that the resulting action of $T$ on $\overset{\circ}{\ol{G_{\on{enh}}}}$ is free.
\footnote{Let us note that for most groups $G$ (e.g., for $G=SL(3)$) the set of points of $\ol{G_{\on{enh}}}$ 
with trivial stablizer in $T$ is \emph{larger} than $\overset{\circ}{\ol{G_{\on{enh}}}}\,$.} It is also easy to show that the quotient 
$\overset{\circ}{\ol{G_{\on{enh}}}}/T$ depends only on $G_{\on{adj}}:=G/Z(G)\,$. 

\medskip

It is well known that this 
quotient is projective and smooth; this is the 
\emph{wonderful compactification}\footnote{In characteristic 0 the wonderful compactification was 
defined in \cite{DCP} (one of the equivalent definitions from \cite{DCP} is as follows: consider 
$\on{diag}(\fg)\subset\fg\times\fg$ as a point of the Grassmannian of $\fg\times\fg$, then take the closure of its
($G\times G$)-orbit.) Without the characteristic zero assumption, the wonderful compactification was 
constructed in \cite{Str} and then, in a different way, in \cite{BKu}. An axiomatic definition of the wonderful 
compactification and a brief survey can be found in  \cite[Sect.~1]{BP}.} 
of $G_{\on{adj}}:=G/Z(G)$. 

\medskip

It is clear that the subscheme 
$\on{Stab}_{G\times G}(x)\subset G\times G$ corresponding to $x\in \overset{\circ}{\ol{G_{\on{enh}}}} (k)$ 
depends only on the image of $x$ in $\overset{\circ}{\ol{G_{\on{enh}}}}/T$.
\end{rem}

\ssec{The section $\bar\fs:\ol{T_{\on{adj}}}\to 
\ol{G_{\on{enh}}}\,$}

Now let us choose a Cartan \emph{sub}group $T_{\on{sub}}\subset G$ and a Borel $B\supset T_{\on{sub}}\,$.
We will construct a section for the map 
$\bar\pi:\ol{G_{\on{enh}}}\to \ol{T_{\on{adj}}}\,$, which depends on these choices.

\sssec{}  \label{sss:thesection}
Let $\varphi :T\iso T_{\on{sub}}\,$ denote the isomorphism corresponding to the chosen Borel 
$B\supset T_{\on{sub}}\,$. Define 
$$\fs:T_{\on{adj}}=T/Z(G)\to (T_{\on{sub}}\times T)/Z(G)\mono (G\times T)/Z(G)=G_{\on{enh}}$$
to be the homomorphism induced by the map
\[
T\to T_{\on{sub}}\times T, \quad \tau\mapsto (\varphi (\tau )^{-1},\tau ).
\]
The composition 
$T_{\on{adj}}\overset{\fs}\longrightarrow G_{\on{enh}}\overset{\pi}\longrightarrow T_{\on{adj}}$ equals the identity.

\begin{lem}   \label{l:barfs}  \hfill 

\smallskip

\noindent{\em(i)} The homomorphism 
$\fs:T_{\on{adj}}\to G_{\on{enh}}$ extends to a homomorphism of algebraic monoids
$$\bar\fs:\ol{T_{\on{adj}}}\to \ol{G_{\on{enh}}}\,.$$

\smallskip

\noindent{\em(ii)} $\bar\fs ( \ol{T_{\on{adj}}})\subset  \overset{\circ}{\ol{G_{\on{enh}}}}\,$.
\end{lem}

\begin{proof}   
(i) It suffices to show that the functor $\Rep(G_{\on{enh}})\to\Rep(T_{\on{adj}})$ corresponding to the homomorphism 
$\fs:T_{\on{adj}}\to G_{\on{enh}}$ maps $\Rep(\ol{G_{\on{enh}}})$ to  $\Rep(\ol{T_{\on{adj}}})$. 

\medskip

Set 
$$T_{\on{enh}}:=(T\times T)/Z(G)\overset{\varphi\times \on{id}}\simeq (T_{\on{sub}}\times T)/Z(G)\subset G_{\on{enh}}.$$
The weights of $T_{\on{enh}}$ are pairs $(\check\lambda_1 ,\check\lambda_2)\in\check\Lambda_T\times\check\Lambda_T$,
such that $\lambda_1|_{Z(G)}=\lambda_2|_{Z(G)}$. 

\medskip 

Let $V\in\Rep(\ol{G_{\on{enh}}})$. Then each weight 
$(\check\lambda_1 ,\check\lambda_2)$ of $T_{\on{enh}}$ on $V$
satisfies $\lambda_1\le\check\lambda_2$ (by the definition of $\ol{G_{\on{enh}}}\,$). 

\medskip

By the definition of $\fs$, the weights of
$T_{\on{adj}}$ on $V$ have the form $\check\lambda_2-\check\lambda_1\,$, where $\check\lambda_1$ and 
$\check\lambda_2$ are as above. So they are $\ge 0$. Therefore the action of $T_{\on{adj}}$ on $V$ extends to an action of $\ol{T_{\on{adj}}}\,$.

\medskip

\noindent(ii) If $V=V_{\check\lambda ,\check\lambda}$ then one of the weights of $\ol{T_{\on{adj}}}$ in $V$ equals
$\check\lambda-\check\lambda=0$, so $\ol{T_{\on{adj}}}$ acts on this subspace by the identity endomorphism. 
\end{proof}

It is clear that the map $\bar\fs:\ol{T_{\on{adj}}}\to\overset{\circ}{\ol{G_{\on{enh}}}}\subset \ol{G_{\on{enh}}}$ defined in the lemma
is a section for the projection 
$\bar\pi:\ol{G_{\on{enh}}}\to \ol{T_{\on{adj}}}\,$. This section depends on the choice of $T_{\on{sub}}$ and $B$.

\sssec{}  

Let $\on{Stab}_{G\times G}(\bar\fs )$ denote the $\bar\fs$-pullback of the subscheme
$\on{Stab}_{G\times G}\subset \overset{\circ}{\ol{G_{\on{enh}}}}\times (G\times G)$ defined in \secref{sss:Stab}.
Clearly  $\on{Stab}_{G\times G}(\bar\fs )$ is a closed
group-subscheme of $\ol{T_{\on{adj}}}\times (G\times G)$ viewed as a group-scheme over 
$\ol{T_{\on{adj}}}\,$.

\begin{cor}   \label{c:universal q-aff}
$\on{Stab}_{G\times G}(\bar\fs )$ is smooth over $\ol{T_{\on{adj}}}\,$. The quotient
\begin{equation}\label{e:q-affine quotient}
(\ol{T_{\on{adj}}}\times (G\times G))/\on{Stab}_{G\times G}(\bar\fs )
\end{equation}
exists and is isomorphic to $\overset{\circ}{\ol{G_{\on{enh}}}}$ as a scheme over $\ol{T_{\on{adj}}}$
equipped with an action of $G\times G$ and a section. In particular, the quotient is a quasi-affine scheme of finite type.
\end{cor}

\begin{proof}
Follows from Sects. ~\ref{sss:transitive}-\ref{sss:Stab}.
\end{proof}

\ssec{Vinberg's semigroup and the ``interpolating" group-scheme $\wt{G}$}  \label{ss:Relation}
\sssec{}
Let $P,P^-\subset G$ be parabolics opposite to each other and $M:=P\cap P^-$.
Let $$\gamma:\BG_m\to Z(M)$$ be a co-character dominant and regular with respect to $P$.
Using these data we defined in \secref{sss:interp group} a closed group-subscheme 
$\wt{G}\subset\BA^1\times G\times G$. On the other hand, in \secref{sss:Stab} we defined a closed group-subscheme 
$\on{Stab}_{G\times G}\subset (G\times G)\times \overset{\circ}{\ol{G_{\on{enh}}}}=
\overset{\circ}{\ol{G_{\on{enh}}}}\times (G\times G)$. We will show that $\wt{G}$ can, in fact, be obtained from
$\on{Stab}_{G\times G}$ by base change with respect to a certain map 
$\BA^1\to \overset{\circ}{\ol{G_{\on{enh}}}}$ (see \propref{p:stabilizer} below).

\sssec{}
Choose a Cartan subgroup $T_{\on{sub}}\subset G$ contained in $M$; then $T_{\on{sub}}\supset Z(M)$. 
Choose a Borel $B\subset G$ so that $T_{\on{sub}}\subset B\subset P$. Then $T_{\on{sub}}$ identifies with $T$.

\medskip

The composition
\[
\BG_m\overset{\gamma}\longrightarrow Z(M)\mono T_{\on{sub}}\iso T\to T_{\on{adj}}
\]
is a \emph{dominant} co-character $\BG_m\to T_{\on{adj}}\,$, so it extends to a homomorphism of algebraic monoids
\begin{equation} \label{e:map to T adj bar}
\BA^1\to \ol{T_{\on{adj}}}\, .
\end{equation} 
Composing \eqref{e:map to T adj bar} with the map 
$\bar\fs:\ol{T_{\on{adj}}}\to \overset{\circ}{\ol{G_{\on{enh}}}}$ defined in \lemref{l:barfs}, we get a map
\[
\gamma':\BA^1\to  \overset{\circ}{\ol{G_{\on{enh}}}}\, .
\]

\sssec{}
Let $\on{Stab}_{G\times G}(\gamma' )$ denote the $\gamma'$-pullback of the subscheme
$\on{Stab}_{G\times G}\subset \overset{\circ}{\ol{G_{\on{enh}}}}\times (G\times G)$ defined in \secref{sss:Stab}.
Clearly  $\on{Stab}_{G\times G}(\gamma' )$ is a closed
group-subscheme of $\BA^1\times (G\times G)$ viewed as a group-scheme over $\BA^1$.

\begin{prop}  \label{p:stabilizer}
The group-subscheme $\wt{G}\subset\BA^1\times (G\times G)$ defined in \secref{sss:interp group} is equal to 
$\on{Stab}_{G\times G}(\gamma' )$.
\end{prop}

\begin{proof}  
Both $\on{Stab}_{G\times G}(\gamma' )$ and $\wt{G}$ are closed group-subschemes of 
$\BA^1\times (G\times G)$. Both are flat (and in fact, smooth) over $\BA^1$: for $\wt{G}$ this was 
proved in \propref{p:smooth}(i) and for $\on{Stab}_{G\times G}(\gamma' )$ this follows from smoothness of
$\on{Stab}_{G\times G}$ over $\overset{\circ}{\ol{G_{\on{enh}}}}$ (see \secref{sss:Stab}). So it suffices to check that the two 
group-subschemes have the same fiber over any $t\in\BG_m\,$.

\medskip

By formula \eqref{e:graph}, the subscheme $\wt{G}_t\subset G\times G$ is defined by the equation
\begin{equation}   \label{e:eq1}
\gamma (t)\cdot g_1\cdot \gamma(t)^{-1}=g_2\, , \quad g_1,g_2\in G.
\end{equation}
On the other hand, the fiber of $\on{Stab}_{G\times G}(\gamma' )$ over $t\in\BG_m$ is the stabilizer of 
$\gamma'(t)\in \ol{G_{\on{enh}}}$ with respect to the action of $(G\times G)$ on $\ol{G_{\on{enh}}}$ defined in 
\secref{sss:2sided translations}; so this fiber is defined by the equation
\begin{equation}    \label{e:eq2}
g_1\cdot \gamma'(t)\cdot g_2^{-1}=\gamma'(t)\, , \quad g_1,g_2\in G.
\end{equation}

By the definition of $\fs$ from \secref{sss:thesection}, $\gamma'(t)$ differs from $\gamma (t)^{-1}$ only by 
an element of $Z(G_{\on{enh}})$.
So \eqref{e:eq1} and \eqref{e:eq2} are equivalent .
\end{proof}  

\begin{cor}  \label{c:q-aff}
The quotient $(\BA^1\times G\times G)/\wt{G}$ exists and is a quasi-affine scheme of finite type over $\BA^1$. 
\end{cor}

\begin{proof}
By \propref{p:stabilizer}, we can rewrite $(\BA^1\times G\times G)/\wt{G}$   as
\begin{equation}   \label{e:quotient in question}
(\BA^1\times G\times G)/\on{Stab}_{G\times G}(\gamma' ).
\end{equation}
By \corref{c:universal q-aff},  the quotient \eqref{e:q-affine quotient} is a quasi-affine scheme of finite type over 
$\ol{T_{\on{adj}}}\,$. The quotient  \eqref{e:quotient in question} is obtained from \eqref{e:q-affine quotient} 
by base change with respect to the map \eqref{e:map to T adj bar}, so it is also quasi-affine.
\end{proof}

\end{document}